\documentclass[12pt]{article}
\usepackage[psamsfonts]{amssymb}
\usepackage{amsmath}
\usepackage{amsfonts,euscript,amssymb,amsmath,amsthm}
\usepackage[english]{babel}

\newcommand{\bbI}{\mathbb{I}}

\newcommand{\bbR}{\mathbb{R}}
\newcommand{\bbO}{\mathbb{O}}
\newcommand{\bbQ}{\mathbb{Q}}
\newcommand{\bbC}{\mathbb{C}}
\newcommand{\bbN}{\mathbb{N}}

\newcommand{\Image}{\mathop{\rm Im}}

\newcommand{\Real}{\mathop{\rm Re}}

\newenvironment{pf}{\par\noindent\textbf{Proof.}}{\hfill$\square$}
\newenvironment{pfl3}{\par\noindent\textbf{Proof of Lemma \ref{l3}.}}{\hfill$\square$}
\newenvironment{pfl4}{\par\noindent\textbf{Proof of Lemma \ref{l4}.}}{\hfill$\square$}
\counterwithin{equation}{section}

\newtheorem{theorem}{Theorem}
\newtheorem{lemma}{Lemma}
\newtheorem{cor}{Corollary}

\begin{document}

\begin{center}
{\bf \large The Epstein zeta-function contains a positive proportion of non-trivial zeros on the critical line.}
\end{center}
\begin{center}
{\bf I.S.~Rezvyakova}
\let\thefootnote\relax


{\it Steklov Mathematical Institute of RAS, Moscow}
\end{center}

\begin{flushright}
{\small \it To my teacher \\
Professor Anatoly Alekseevich Karatsuba \\
and to my friends}
\end{flushright}
\hspace{1mm}
\begin{center}

{\small We prove that the Epstein zeta-function corresponding to a binary positive definite quadratic form with integer coefficients has a positive proportion of its non-trivial zeros on the critical line. }

\end{center}
\hspace{1mm}

\begin{center}
{\bf \S 1. Introduction. }
\end{center}

The famous Riemann hypothesis claims that all non-trivial zeros of the Riemann zeta function $\zeta(s)$ lie on the critical line $\Real s  = 1/2$. Moreover, similar hypothesis (called the GRH (generalised Riemann hypothesis)) exists for a large class of L-functions. For example, L-functions from S class (or Selberg class) introduced by A.~Selberg in \cite{Selberg_Amalfi} are expected to satisfy the GRH. In 1942 A.~Selberg \cite{Selberg1942} 
expanded our knowledge on distribution of non-trivial zeros of $\zeta(s)$.
His breakthrough result proves that a positive proportion of non-trivial zeros of the Riemann zeta-function lie on the critical line (some earlier results about zeros of $\zeta(s)$ on the critical line are given in
\cite{H_L}, \cite{Selberg_1}, for example). Selberg's method provides similar result for Dirichlet L-functions and, therefore, establishes the positive proportion result for all L-functions from Selberg class of degree one (see necessary notation and properties of L-functions from Selberg class of degree one in \cite{KP_1999}, \cite{Sound} and in the aforementioned work \cite{Selberg_Amalfi}).

It was not until 1983 that the positive proportion result was proved by J.L.~Hafner for L-function of degree two from Selberg class \cite{Hafner1}. He established it for L-functions attached to holomorphic cusp forms of an even weight for the full modular group and trivial character which are also eigenfunctions of all the Hecke operators. In 1987,  Hafner \cite{Hafner2} proved similar result for L-functions of degree two attached to Maass wave forms. In \cite{Rezvyakova_MZ} we extended the result of \cite{Hafner1}  and proved the positive proportion theorem for L-functions of degree 2 attached to automorphic eigen cusp forms for the Hecke congruence group of any integer weight $k \ge 1$ and any character. But there is no result of such type for any primitive L-function from Selberg class of degree 3 or higher (this problem resembles the problem on even moments of zeta-function of order 6 or higher). We mention here that there is a conjecture which states that a degree of L-function from Selberg class is a non-negative integer. This conjecture is true when the degree of L-function is not greater than 2 (see \cite{Richert}, \cite{KP}). Nevertheless, the structure of Selberg class of degree 2 (as well as of higher integer degrees) is not yet completely described.

For deeper understanding of the RH, we shall now switch our attention to linear combinations of L-functions from Selberg class. Such examples are the Davenport-Heilbronn function (which is a linear combination of functions of degree one from Selberg class) or the Epstein zeta-function of a binary positive definite quadratic form (which is a linear combination of functions of degree 2). They illustrate that linear combinations of L-functions could satisfy all the properties of an element from Selberg class (even posess a functional equation of the Riemann type) except having an Euler-product representation and, therefore, have many zeros outside the critical line (\cite{Dav_Heil}, \cite{Voronin}) and violate the RH. Nevertheless, there is a conjecture stating that almost all non-trivial zeros of such linear combinations lie on the critical line. This conjecture was conditionally proved, for example, for the Epstein zeta-function by E.~Bombieri and D.~Hejhal \cite{Bom_Hejhal}, \cite{Bom_Hejhal_1} under the GRH and the Montgomery pair correlation conjecture for Hecke L-functions.
So, it seems that solely the Euler-product representation forces L-function to satisfy the RH.

What is known about the zeros of such type linear combinations? It is known, that all non-trivial zeros of such linear combinations lie in some strip $|\Real s| \le a$, where $a = a_{L}$ is some constant depending on the linear combination $L(s)$ (and we may suppose that $a_{L}$ is the minimal such value). If we denote by $N_{L}(T)$ the number of zeros of  linear combination $L(s)$ in the region $0 < \Image s \le T$ then
$$
N_{L}(T) \asymp T \log T
$$
as in the case of the Riemann zeta-function (the proof uses merely the functional equation).
Setting $N_{L} (\sigma_1, \sigma_2, T)$ to be the number of zeros of $L(s)$ in the region $\sigma_1 \le \Real s \le \sigma_2, \quad 0< \Image s \le T$, where $1/2<\sigma_1<\sigma_2 \le a_L$, we have
$$
N_{L}(\sigma_1, \sigma_2, T) \asymp T,
$$
which shows that quite a few non-trivial zeros lie outside the critical line. Denote by $N_{0} (g,T)$ the number of zeros of the Davenport-Heilbronn function $g(s)$ on the interval $(1/2+iT, 1/2+2iT)$ of the critical line. Following the ideas of \cite{H_L} one can show that
$$
N_{0} (g,T) \gg T.
$$
It was first S.M.~Voronin \cite{Voronin1980}, who showed that the critical line $\Real s = \frac12$ is an exceptional set for the Davenport-Heilbronn function. Namely, for sufficiently large $T$,  Voronin proved that if $\Omega_g (t) = N_{0} (g,T)/T$, then
$$
\Omega_g (t) \gg \exp (\frac{1}{20} \sqrt{\log\log\log\log T}).
$$
In 1989, A.A. Karatsuba \cite{Karatsuba1}, \cite{Karatsuba3}  obtained the following result
\begin{equation}\label{eq333}
\Omega_g (t) \gg (\log T)^{1/2-\varepsilon}
\end{equation}
for arbitrary $\varepsilon >0$ and $T \ge T_0 (\varepsilon)$.
Finally, in the late 90s, A.~Selberg \cite{Selberg1999} proved that a positive proportion of non-trivial zeros of the Davenport-Heilbronn function (or a linear combination of degree one L-functions which also obeys a functional equation) lie on the critical line (see also \cite{Rezvyakova_2015_Dokl} for extra details and for an improved constant in the proportion which is in inverse ratio to the number of functions of a linear combination). A substantial improvement made by Selberg was possible after employing in the proof the value distribution result for the values of L-functions on the critical line (see \cite{Ghosh}, \cite{Tsang} for value-distribution results).

When we make a step from degree 1 to degree 2 L-functions (even in the simpliest case of just one function in linear combination) we come across an obstacle called an ``additive problem'' (which also appears in the estimation of the fourth moment of the Riemann zeta-function). Having overcome this problem what is left for us are the detailed calculations leading us finally to the solution of the initial problem. This was done in
\cite{Hafner1}, \cite{Hafner2}, \cite{Rezvyakova_MZ} for certian degree 2 L-functions from Selberg class. Applying an extra value-distribution ingredient introduced in \cite{Selberg1999} for a linear combination of L-functions from \cite{Rezvyakova_MZ}, we were able to get the result of \cite{Rezvyakova_2016}. For example, we proved that a linear combination with real coefficients of L-functions attached to complex Hecke ideal class group characters of the quadratic imaginary extension of the field of rational numbers has a positive proportion of non-trivial zeros on the critical line $\Real s = 1/2$ (the previous result is presented in \cite{Rezvyakova}).

In the current paper we prove a similar result, announced in \cite{Rezvyakova_2015},
for the Epstein zeta function  corresponding to a binary positive definite quadratic form with integer coefficients (which is a linear combination of Hecke L-functions with complex and real ideal class group characters). The difficulty which arises now is the presence of L-functions with real Hecke characters in a linear combination (in comparison to the situation of \cite{Rezvyakova_2016}) and the main term in the additive problem in this case.

In \S 2 we formulate the statement of the main result of the work . The ideas of the proof introduced by Selberg in \cite{Selberg1999} are presented in \S 3. Further sections contain necessary results for L-functions with real Hecke characters to prove the main theorem.


\begin{center}
{\bf \S 2. Definitions and statement of the main result}
\end{center}

We recall the definition of Hecke characters of the ideal class group. Let $\mathbb Q(\sqrt{-D})$ be the imaginary quadratic extension
of the field of rational numbers $\mathbb Q$,  and $-D$ be the fundamental discriminant.
We denote by $\bbI = Z_{\bbQ
(\sqrt{-D})}$ the ring of algebraic integers of the filed $\bbQ
(\sqrt{-D})$ and by $\bbO$ the ring of principal ideals. The quotient group $\bbI^{*} / \bbO^{*}$ is an Abelian group of finite order $h = [\bbI^{*} : \bbO^{*}]$. Its elements are the equivalence classes of non-zero ideals of the ring $\bbI$ with respect to the following equivalence relation:
$$
\mathfrak{a} \sim \mathfrak{b} \quad (\mathfrak{a},
\mathfrak{b} \in \bbI^{*})
$$
if there is a non-zero principle ideal $(\omega)$ such that
$
\mathfrak{a} = (\omega) \mathfrak{b}.
$
The number $h= h(-D)$ is the {\it class number} of the field $\bbQ(\sqrt{-D})$. For $k=1,2, \ldots, h$, we define a homomorphism $\psi : \bbI^{*} \to \bbC^{*}$ on the equivalence classes
$I_k$ and set
$$
\psi( \mathfrak{a} ) = \psi( I_{k}), \quad \text{ if } \quad
\mathfrak{a}\in I_{k}.
$$
The function $\psi(\mathfrak{a})$ defined on non-zero ideals $\mathfrak{a}$ of the ring $\bbI$ is called a {\it Hecke character of the ideal class group}. All non-zero principle ideals form an identity class in the ideals class group. Hence, $\psi(\mathfrak{a}) =1$ if $\mathfrak{a}$ is a principle ideal. Also, for any non-zero ideal $\mathfrak{a}$ we have
$$
(\psi(\mathfrak{a}))^{h}=1.
$$

A {\it principle} Hecke character $\psi_0$ is given by the relation $\psi_0 (\mathfrak{a})=1$ for every
$\mathfrak{a}$. We say that a Hecke character $\psi$ is {\it complex} if it takes at least one complex (non-real) value, i.e.
$\psi^2 \ne \psi_0$.

An ideal  $\mathfrak{p}$ is {\it prime} if it is divisible by no ideals other than  $\mathfrak{p}$ and $(\mathbf{1})$. Every ideal in
$\bbI^{*}$ can be represented uniquely as a product of prime ideals. The norm of an ideal $\mathfrak{a}$ is the number of residue classes in  $\bbI$ modulo $\mathfrak{a}$. The norm $N\mathfrak{a}$ is a totally multiplicative function (see~\cite{Hecke1923}).
For $\operatorname{Re}s>1$, the Hecke
$L$-function is defined by the formula
$$
L_{\psi}(s)=\sum_{\mathfrak a\in\mathbb I^*} \frac{\psi(\mathfrak
a)}{N\mathfrak a^s} =
\sum\limits_{n=1}^{+\infty} \frac{r_{\psi}(n)}{n^s},
$$
where
$$
r_{\psi}(n)= \sum\limits_{N\mathfrak{a}=n} \psi(\mathfrak{a}).
$$
The series $L_{\psi}(s)$ can be written via the Euler product in the right half-plane $\Real s>1$ as
$$
L_{\psi}(s)=\prod_{\mathfrak p}\biggl(1- \frac{\psi(\mathfrak
p)}{N\mathfrak p^s}\biggr)^{-1},
$$
where the product goes over all prime ideals.

The quadratic field theory gives us some more information on $L_{\psi}(s)$. All the prime ideals of  $\bbI$ are included among the factors of principle ideals $(p)$, where $p$ runs through all the prime numbers. In addition,
\begin{align*}
(p) =
\begin{cases}
\mathfrak{p_1} \mathfrak{p_2}, &\mbox{ if } \chi_D (p) = 1\; (N \mathfrak{p_1} = N \mathfrak{p_2} = p), \\
\mathfrak{p}, &\mbox{ if } \chi_D (p) = -1 \;(N \mathfrak{p} = p^2),\\
\mathfrak{p}^2 , &\mbox{ if } \chi_D (p) = 0 \;(N\mathfrak{p} = p),
\end{cases}
\end{align*}
where $\mathfrak{p}$, $\mathfrak{p_1}$, $\mathfrak{p_2}$ are prime ideals.
We introduce three sets $P_0, P_1, P_{-1}$ ranging all prime numbers by the following criterion:
\begin{equation*}
p \in P_{e} \quad \text{if} \quad \chi_D (p) = e, \quad
\text{where} \quad e=0, \pm 1.
\end{equation*}
Thus, we have an identity
\begin{equation*}
L(s, \psi) = \prod_{p\in P_0} \left( 1
-\frac{\psi(\mathfrak{p})}{p^s} \right)^{-1} \prod_{p\in P_{-1}}
\left( 1 -\frac{1}{p^{2s}} \right)^{-1} \prod_{p\in P_1} \left( 1
-\frac{\psi(\mathfrak{p_1})}{p^s} \right)^{-1} \left( 1
-\frac{\psi(\mathfrak{p_2})}{p^s} \right)^{-1},
\end{equation*}
We see that  if  $p \in P_0$, then
\begin{align*}
(p) = \mathfrak{p}^2, \; \psi(\mathfrak{p}) = \pm 1,
\end{align*}
and if  $p \in P_1$, then
\begin{align*}
(p) = \mathfrak{p_1} \mathfrak{p_2},  \; \psi(\mathfrak{p_1}) =
\overline{\psi(\mathfrak{p_2})}.
\end{align*}
Thus the coefficients of the series $L(s, \psi) = \sum\limits_{n=1}^{+\infty}
\frac{r_{\psi} (n)}{n^s}$ are real and satisfy
$$
r_{\psi} (1) =1, \quad |r_{\psi} (n)| \le \tau(n),
$$
where $\tau(n)$ is the divisor function. Therefore we note, that
two complex conjugate Hecke characters define one and the same Hecke L-function.

The function $L_{\psi}(s)$ has
a meromorphic continuation to the whole complex plane. For the principle character
$\psi=\psi_0$, the function $L(s,\psi)$ has the only pole (which is simple) at $s=1$, and for $\psi \ne \psi_{0}$ it has an analytic continuation to the whole complex plane $\mathbb C$ without singularities.
E. Hecke (\cite{Hecke1917},~\cite{Hecke1926}) established a functional equation for $L_{\psi} (s)$.
\begin{lemma}\label{l1}
Let
\begin{equation}\label{eq3}
\Lambda_{\psi}(s)=\biggl(\frac{2\pi}{\sqrt
D}\biggr)^{-s}\Gamma(s)L_{\psi}(s).
\end{equation}
Then
\begin{equation}\label{eq530}
\Lambda_{\psi}(s)=\Lambda_{\psi}(1-s).
\end{equation}
\end{lemma}
He has also revealed a connection between Hecke L-functions and modular forms.  For $z\in\mathbb
H=\{z\colon\operatorname{Im}z>0\}$ set
$$
f_{\psi}(z)=\sum_{n=1}^{+\infty}r_\psi(n)e^{2\pi inz}.
$$
If  $\psi$ is a complex Hecke character, then $f_{\psi}(z)$ is a
holomorphic cusp form of weight~1 for $\Gamma_0(D)$ and the
character $\chi_D (\cdot)=(\frac{-D}\cdot)$. In other words,
for every matrix $\gamma=\bigl(\begin{smallmatrix}a&b
\\
c&d\end{smallmatrix}\bigr)$ of the group $\Gamma_0(D)$ (which is a
subgroup of $\text{SL}_2(\mathbb Z)$ that consists of all the
matrices~$\gamma$ with $c\equiv0\ (\operatorname{mod}D)$) the following
relation holds:
\begin{equation*}
f_{\psi}(\gamma z) = \chi_{D} (\gamma) (cz +d) f_{\psi}(z),
\end{equation*}
where $\chi_{D} (\gamma) = \chi_{D} (d)$ (see \cite{Hecke1926},~\cite{Iwaniec_book}).

The Hecke L-function associated with a real character is actually a product of two Dirichlet L-functions due to Kronecker's theorem. Namely, it is represented as the product $L_{\chi_{d_1}} (s) L_{\chi_{d_2}} (s)$, where $\chi_{d_j}$ are primitive real Dirichlet characters such that $D = d_1 d_2$, $(d_1, d_2)=1$, $\chi_{d_1} \chi_{d_2} = \chi_{D}$.

Now, for $\Real s >1$, we define the Epstein zeta-function $\zeta_{Q} (s)$ by the formula
$$
\zeta_{Q} (s) = \sum\limits_{(n, m)\ne (0,0)} (Q(n,m))^{-s},
$$
where  $Q(n,m) = an^2 + bnm + c m^2$ is a positive definite quadratic form with integer coefficients and $-D = b^2 - 4ac <0$ is the  fundamental discriminant. It is well known, that $\zeta_{Q} (s)$ can be represented as a linear combination of  L-functions with Hecke ideal class group characters of imaginary quadratic extension of $\mathbb{Q}(\sqrt{-D})$. Namely,
\begin{equation}\label{eq14}
\zeta_{Q} (s) = \frac{\varepsilon_{-D}}{h(-D)}\sum\limits_{j=1}^{h} \overline{\psi_{j} (I_{Q})} L(s, \psi_j),
\end{equation}
where again $h=h(-D)$ is the class number, $\varepsilon_{-D}$ is the number of units and $I_{Q}$ is the ideal
$$I_{Q} = \left( a, b-i\frac{\sqrt{D}}{2}\right).
$$
Combining in the sum (\ref{eq14}) each pair of L-functions corresponding to complex conjugate Hecke characters, we see that $\zeta_{Q} (s)$ is a linear combination of distinct Hecke $L$~-functions with real coefficients. And it also satisfies the functional equation of lemma \ref{l1}:
\begin{equation*}
\biggl(\frac{2\pi}{\sqrt
D}\biggr)^{-s}\Gamma(s) \zeta_{Q} (s) = \biggl(\frac{2\pi}{\sqrt
D}\biggr)^{-1+s}\Gamma(1-s) \zeta_{Q} (1-s)
\end{equation*}

In this work we shall prove that the Epstein zeta-function $\zeta_{Q} (s)$ has a positive proportion of non-trivial zeros on the critical line. More generally, we prove the following theorem.
\begin{theorem}\label{th1}
Suppose that
\begin{equation}\label{eq0}
F(s)  = \sum\limits_{j=1}^{P} c_j L_{j}(s)
\end{equation}
is a linear combination of $P$ distinct Hecke $L$-functions attached to
ideal class group characters of the field $Q(\sqrt{-D})$
formed with the real coefficients $c_j$. Then a positive proportion of non-trivial zeros of $F(s)$ lie on the critical line. Namely, if we denote by $N_0(T)$ the number of zeros of $F(s)$ on the interval $\{ s= 1/2+it, T\le t \le 2T\}$, then for any large positive $T$ we have
$$
N_0(T) \gg T \log T.
$$
\end{theorem}

\begin{center}
{\bf \S 3. Selberg's method for linear combinations}
\end{center}

We present in this section the core of the proof introduced by A. Selberg in \cite{Selberg1999}.
The basic idea of catching a zero of odd order of a real-valued continuous function $f(u)$ on some interval consists in comparison of the following two integrals. If for some $t$ and $H>0$ the inequality
\begin{equation}\label{eq2}
\int\limits_{t}^{t+H} |f(u)| du > \left|\int\limits_{t}^{t+H} f(u) du \right|
\end{equation}
holds, then $f(\cdot)$ changes its sign on $(t, t+H)$ and, thus, has a zero of odd order on that interval.
This idea was used by E.~Landau, H. Bohr, H. Hardy and  J. Littlewood. Note, that from the functional equation (\ref{eq530}), it follows that critical line zeros of $F(s)$ from (\ref{eq0}) correspond
to real zeros of the real-valued function
\begin{equation}\label{eq531}
\biggl(\frac{2\pi}{\sqrt
D}\biggr)^{-1/2 +iu}\Gamma(1/2+iu) F(1/2+iu).
\end{equation}
So we have a candidate for investigating the inequality (\ref{eq2}) and, therefore, zeros of $F(s)$ on the critical line.

Now,  if the subset $E = \{ t \in (T, 2T):  (\ref{eq2}) \text{ holds} \}$ has the measure $\mu E$, then the number of zeros of $f$ on the interval $(T, 2T)$ is not less than $\frac{\mu E}{2H}-1$ (which is not difficult to show).
The measure of the subset $E$ can be calculated during the evaluation of the mean-square of both sides of (\ref{eq2}). Applying the method of \cite{H_L} to (\ref{eq531}) can only give us the result like
$N_0 (T) \gg T$.

A.~Selberg substantially supplemented the initial idea in \cite{Selberg1942} and proved the positive proportion theorem for $\zeta(s)$. He introduced the so-called ``mollifier''  which has only even zeros and,
therefore, does not affect on the zeros of odd order of a continuous real function (if it is multiplied by the``mollifier''). But at the same time the product becomes less oscillating function, the Cauchy inequality becomes more
precise
dealing with the mean-square of both sides of (\ref{eq2}) and, therefore, one is able to
establish the inequality (\ref{eq2}) on the subset of maximal order for smaller $H$.

For fixed $L(s)$ with Euler product, it turned out that a good candidate for the ``mollifier'' is the square of an approximation to $L^{-1/2} (s)$ by a finite Dirichlet series. In \cite{Selberg1942} A.~Selberg offered as a ``mollifier'' a smoothed Dirichlet polynomial attached to $\zeta^{-1/2} (s)$ which allowed him to solve the positive proportion problem for the Riemann zeta-function. The same can be done for Dirichlet L-functions.  In case of the Davenport-Heilbronn function, A.A.~Karatsuba noticed that the two Dirichlet L-functions in the linear combination have the same Euler product decomposition on one-half of primes. And it is turned out that constructing a ``mollifier'' on the part of Euler-product which is the same for both L-functions allows (with some extra trick) to get the result  (\ref{eq333}). The method of A.A.~Karatsuba allows to obtain a non-trivial result for a linear combination of L-functions with complex Hecke ideal class group characters \cite{Rezvyakova} (but, unfortunately, not for the Epstein zeta-function).

Let us present now the method of \cite{Selberg1999} which allows to get the positive proportion result for a class of linear combinations of L-functions. For a given $L$-function a ``mollifier'' is connected to $L^{-1/2} (s)$ (as we have seen) and we choose it as follows. Let $\psi_j$ be any Hecke character of the ideal class group of  $\mathbb{Q} (\sqrt{-D})$.
For $L_j (s) := L_{\psi_j}(s)$
define $\alpha_{j} (\nu)$ by the relation (assuming that $\Real s >1$)
\begin{equation*}
\sum\limits_{\nu =1}^{+\infty} \alpha_{j}(\nu) \nu^{-s}
= L^{-1/2}_j(s).
\end{equation*}
For $X\ge 3$, define  $\eta_j (s)$ as the following Dirichlet polynomial
$$
\eta_j (s) = \sum\limits_{\nu < \sqrt{X}} \frac{\alpha_j (\nu)}
{\nu^{s}} +\sum\limits_{\sqrt{X} \le \nu \le X }  \frac{\alpha_j (\nu)}
{\nu^{s}} \left( 2\dfrac{\log X/\nu}{\log X}\right)  := \sum\limits_{\nu \le X} \frac{\beta_j (\nu)}{\nu^s},
$$
where
\begin{equation*}
\beta_j (\nu ) = \alpha_j (\nu ) \EuScript{L}(\nu)
\end{equation*}
and
\begin{equation*}
\EuScript{L}(\nu) = \begin{cases} 1 &\text{for }\ \nu < \sqrt{X},
\\
2\dfrac{\log X/\nu}{\log X} & \text{for }\ \sqrt{X} \le \nu \le X, \\
0 &\text{otherwise}.
\end{cases}
\end{equation*}
Let $T$ be large enough, $\frac{1}{\log T} \le H \le 1$
and $X$ be a small (but fixed) power of $T$ (i.e., $\log X \asymp \log T$). Define
\begin{equation*}
I_j(t, H)  = \int\limits_{t}^{t+H}
\Lambda_j\left( \frac12 + iu\right) \left| \eta_j \left( \frac12 +iu\right)\right|^{2}
\exp\left( \left(\frac{\pi}{2} -\frac{1}{T}\right) u\right) du,
\end{equation*}
where $\Lambda_j (\cdot)$ is defined by (\ref{eq3}), and suppose that the following estimates hold:
\begin{equation}\label{eq5}
\begin{split}
\int\limits_{T}^{2T} |I_j(t, H)|^2 dt & = O\left( \frac{TH^2}{(H \log T)^{1/3}}\right),
\end{split}
\end{equation}
\begin{equation}\label{eq6}
\int\limits_{T}^{2T} |L_j(1/2+it) \eta_j^2 (1/2+it)|^2 dt = O(T),
\end{equation}
and, setting
$$
M_j(t, H) = \int\limits_{t}^{t+H} L_j (1/2 +iu) \eta_j^2 (1/2+iu) du - H,
$$
that
\begin{equation}\label{eq7}
\begin{split}
\int\limits_{T}^{2T} |M_j(t, H)|^2 dt & = O\left(  \frac{TH^2}{(H\log T)^{1/3}}\right).
\end{split}
\end{equation}
It follows from \cite{Rezvyakova_MZ}, for example, that the estimates (\ref{eq5})--(\ref{eq7}) for $L_j (s)$ imply that a positive proportion of non-trivial zeros of $L_j (s)$ lie on the critical line $\Real s = 1/2$. Moreover, the estimate (\ref{eq6}) allows to prove the following Selberg's density theorem for $L_j(s)$ (see \S4 in  \cite{Rezvyakova_2016}).
\begin{lemma}\label{l00}
Let $\sigma \ge 1/2$, $T\ge 3$, $N_j(\sigma, T)$ be the number of zeros of $L_j (s)$ in the region $\Real s \ge \sigma$, $T \le \Image s \le 2 T$. Then
$$
N_j(\sigma, T) \ll T^{1-c(\sigma -1/2)} \log T
$$
for some absolute constant $c>0$.
\end{lemma}
This estimate is of a precise order for $\sigma = 1/2$ and implies that almost all non-trivial zeros of $L_j (s)$ lie in the vicinity of the critical line. Namely, for any function $\phi(T) \to +\infty$ when $T\to +\infty$, the estimate
$$
N_j \left(\frac{1}{2}+ \frac{\phi (T) }{\log T}, T\right) \ll (T \log T) e^{-c \phi (T)} = o(T \log T)
$$
holds when $T\to +\infty$.

Selberg's novelty in the proof of the positive proportion theorem for linear combinations was an application of the following result on the value distribution of  $\log
|L_j(1/2+it)|$ which relies on the above density theorem.
If we have two distinct Hecke L-functions, then
the difference
$$
\frac{\log|L_{j}(1/2+it)| -  \log|L_{j'}(1/2+it)|}{\sqrt{(n_{j} +n_{j'} )\pi
\log\log t}}
$$
has a normal distribution, where $n_{j}$ equals to $1$ or $2$ depending on
weather $\psi_j$ is a complex or a real Hecke character, respectively.
More precisely: let
$\varkappa_{a,b}$ denote the characteristic function of the
interval
$(a,b)$, then
\begin{equation}\label{eq8}
\int\limits_{T}^{2T} \varkappa_{a,b} \left( \frac{\log|L_{j}(1/2+it)| - \log|L_{j'}(1/2+it)|}{\sqrt{(n_{j} +n_{j'} )\pi \log\log T}} \right) dt =
T \int\limits_{a}^{b} e^{-\pi u^2} du + O \left( T\varphi(T)\right),
\end{equation}
where $\varphi(T) \to 0$ as $T\to +\infty$.
The core of the proof of this value distribution formula are the following mean estimates:
\begin{equation}\label{eq12}
\int\limits_{T}^{2T} \left| \log |L_{\psi}(1/2+it)| - \Real \sum\limits_{p < z} \frac{r_{\psi}(p)}{p^{1/2+it}}\right|^{2k} dt = O(T (Ck)^{4k})
\end{equation}
where $\frac{\log z}{\log T} \asymp \frac{1}{k}$.
A skeleton of the proof of (\ref{eq8}) is presented in \S~6 in \cite{Rezvyakova_2016}.
A very sharp estimate on $\varphi(T)$ was obtained for the Riemann zeta function by K.M.~Tsang in his Ph.D. thesis \cite{Tsang}. Namely, he proved that
$$
\varphi(T) \ll \frac{(\log\log\log T)^2}{\sqrt{\log\log T}} .
$$
An easier method of A.~Ghosh \cite{Ghosh} gives a weaker estimate, but for our task we can use any non-trivial result
$$
\varphi(T)  = o(1).
$$
Fix $0< \varepsilon <1/2$ and suppose that we have (\ref{eq8}) with a monotonically decreasing function $\varphi(T)$
such that $\varphi(T)  \gg (\log\log T)^{-\frac12+\varepsilon}$.
Whence, the subset of the interval $(T, 2T)$ where the inequality
\begin{equation*}
\left|\log|L_{j}(1/2+it)| -  \log|L_{j'}(1/2+it)|\right| \le
(\log\log T)^{\varepsilon}
\end{equation*}
holds has the measure $O\left(T \varphi(T) \right)$.
Thus, most of the time one of $L_j(1/2+it)$ dominates all the other
decisively. We show further that this dominance is somewhat persistant over stretches quite long compared to $\dfrac{1}{\log T}$. Suppose $1 \le H \log T \le \varphi^{-1} (T)$. Define
$$
\Delta_{j} (t, H) = \frac{1}{H} \int\limits_{t}^{t+H} \log
|L_{j}(1/2+iu)| du.
$$
For $0 \le u \le H$  and any positive integer $k$ one can show  from (\ref{eq12}) that (see \S6 in \cite{Rezvyakova_2016})
\begin{equation}\label{eq9}
\int\limits_{T}^{2T} \left( \Delta_{j} (t, H) - \log
|L_{j}(1/2+i(t+u))| \right)^{2k} dt = O\left( T C^{k}
\left(k^{4k} + k^k \log^k (H\log T)\right)\right).
\end{equation}
Integrating over $u$, we obtain
$$
\int\limits_{T}^{2T} \int\limits_{0}^{H} \left( \Delta_{j} (t, H)
- \log |L_{j}(1/2+i(t+u))| \right)^{2k} du dt = O\left( TH C^{k}
\left(k^{4k} + k^k  (\log\log \log T)^{k} \right)\right).
$$
If we denote by $W_j(t)$ the subset of $u \in [0, H]$ for which
\begin{equation*}
\left|\Delta_{j} (t, H) - \log |L_{j}(1/2+i(t+u))|\right| >
(\log\log T)^{\frac{\varepsilon}{2}},
\end{equation*}
we find, choosing $k$ so large that $k\varepsilon > 7$, that the measure of this subset can be estimated from above as
$$
\mu \left(W_j(t)\right) \le \frac{H}{(\log\log T)^3},
$$
except for a subset of $t$ from $(T, 2T)$ of measure $O\left(T(\log\log
T)^{-3} \right)$.

Now we exclude from $(T, 2T)$ all $t$ such that
\begin{equation*}
\left| \log|L_{j}(1/2+it)| -  \log|L_{j'}(1/2+t)| \right| \le
(\log\log T)^{\varepsilon}
\end{equation*}
for some $j \ne j'$, and also exclude all
$t$ such that $\mu \left(W_j(t)\right) > \frac{H}{(\log\log T)^3}$
for some $j$ and get that the interval $(T, 2T)$ apart from a subset of measure $O(T \varphi(T))$ can be  divided into
$P$ (the number of different L-functions in the linear combination) sets $S_j$ such that for each $t \in S_j$ we have for $j'\ne j$
\begin{equation*}
\log|L_j(1/2+it)| -  \log|L_{j'}(1/2+it)| > (\log\log
T)^{\varepsilon},
\end{equation*}
and for $u\in H_t := (0, H) \setminus \bigcup\limits_{k=1}^{P} W_k (t)$
\begin{equation*}
\begin{split}
&\log|L_j(1/2+i(t+u))| -  \log|L_{j'}(1/2+i(t+u))| = \left( \log|L_j(1/2+i(t+u))| -
\Delta_{j} (t, H) \right) \\
&-
\left( \log|L_j(1/2+it)| - \Delta_{j} (t, H) \right)
- \left( \log|L_{j'}(1/2+i(t+u))| - \Delta_{j'} (t, H) \right) \\
&+ \left(
\log|L_{j'}(1/2+it)|  - \Delta_{j'} (t, H) \right)
+ \left(
\log|L_j(1/2+it)| - \log|L_{j'}(1/2+it)| \right) \\
& > (\log\log T)^{\varepsilon} - 4 \left(\log\log T
\right)^{\frac{\varepsilon}{2}} > \frac12  (\log\log T)^{\varepsilon}.
\end{split}
\end{equation*}
From (\ref{eq6}) we see that
$$
\int\limits_{t}^{t+H} |L_j(1/2+iu) \eta_j^2 (1/2+iu)|^2 dt < H
\log\log T
$$
except for a subset of $t\in (T, 2T)$ of measure $O \left(
\dfrac{T}{\log\log T}\right)$. We also exclude those $t$ from $S_j$
without renaming them.

Define $\mathfrak{F}(t)$ attached to the linear combination (\ref{eq0}) by the formula
$$
\mathfrak{F}(t)= \sum\limits_{j=1}^{P} c_j \Lambda_j\left( \frac12 +
it\right).
$$
Notice, that this is a real-valued function for real $t$, and its real zeros are the zeros of our linear combination on the critical line.

Now, for $t\in S_j$, compare two values:
\begin{equation*}
\begin{split}
I(t, H) & = \int\limits_{H_t} \mathfrak{F} (t+u) \left|\eta_j\left(\frac12+i(t+u)\right)\right|^2
\exp\left( \left(\frac{\pi}{2} -\frac{1}{T}\right)
(t+u)\right)du, \\
J(t, H) & = \int\limits_{H_t} \left|\mathfrak{F} (t+u) \eta_j^2\left(\frac12+i(t+u)\right)\right|
\exp\left( \left(\frac{\pi}{2} -\frac{1}{T}\right) (t+u)\right) du.
\end{split}
\end{equation*}
If $J(t, H) > |I(t, H)|$, then $\mathfrak{F}(\cdot)$ changes its sign on $(t,
t+H)$ and so has at least one zero of odd order there.
Using Cauchy's inequality, we find
\begin{equation*}
\begin{split}
I(t, H) & = c_j \int\limits_{u\in H_t} \Lambda_j\left( \frac12+i(t+u)\right)\left|\eta_j\left( \frac12+i(t+u)\right)\right|^2  \exp\left( \left(\frac{\pi}{2} -\frac{1}{T}\right) (t+u)\right) du \\
& + O \left(\int\limits_{u\in H_t} \left|L_j\left( \frac12+i(t+u)\right)\right|\left|\eta_j\left( \frac12+i(t+u)\right)\right|^2
e^{-\frac12
(\log\log T)^{\varepsilon}}du \right) \\
&= c_j \int\limits_{0}^{H}  \Lambda_j\left( \frac12+i(t+u)\right)\left|\eta_j\left( \frac12+i(t+u)\right)\right|^2 \exp\left(
\left(\frac{\pi}{2} -\frac{1}{T}\right) (t+u)\right) du \\
&+ O\left( \sqrt{\frac{H}{(\log\log T)^3}} \cdot \sqrt{H\log\log
T}\right) + O \left(\sqrt{\frac{H\cdot H \log\log T}{e^{(\log\log
T)^{\varepsilon}}}}
\right)\\
&= c_j I_j(t, H) +  O\left(
\frac{H}{\log\log T} \right).
\end{split}
\end{equation*}
Similarly we get
\begin{equation*}
\begin{split}
J(t, H) &= c_j J_j(t, H) + O\left( \frac{H}{\log\log T} \right),
\end{split}
\end{equation*}
where
\begin{equation*}
\begin{split}
J_j(t, H) & = \int\limits_{t}^{t+H} |\Lambda_j (1/2+iu)| |\eta_j(1/2+iu)|^2 \exp\left(
\left(\frac{\pi}{2} -\frac{1}{T}\right) u\right) du.
\end{split}
\end{equation*}
For $T \le t \le 2T$, by means of
Stirling's formula for the Gamma function we obtain
$$
J_j(t, H) \ge e^{-3} \int\limits_{t}^{t+H} |L_j (1/2 +iu) \eta_j^2 (1/2+iu)| du \ge e^{-3} \left( H - |M_j(t, H)|\right).
$$
From (\ref{eq5}) and (\ref{eq7}) it follows, that
$$
|I_j (t, H)| \le \frac{H}{3 e^3} \quad \text{and} \quad |M_j (t, H)| \le
\frac{H}{3}
$$
outside a set of measure $O\left( \frac{T}{(H\log T)^{1/3}}\right)$. Therefore, for $t\in S_j$ except for a subset of measure $O\left(
\frac{T}{(H\log T)^{1/3}}\right)$ the inequality
$$
\frac{|I(t, H)|}{c_j} \le \frac{H}{3 e^3} + O\left( \frac{H}{\log\log T} \right)<
\frac {H}{2 e^3} < \frac{2H}{3 e^3} - O\left( \frac{H}{\log\log T} \right)
\le  \frac{J(t, H)}{c_j},
$$
or
$$
|I(t, H)| < J(t, H)
$$
holds. Taking into account all $j$ we get that the linear combination (\ref{eq0}) changes its sign on $(t, t+H)$ for $t$ from a subset of $(T, 2T)$ of measure
$$
\sum\limits_{j=1}^{P} \mu(S_j) - O\left( \frac{P T}{(H\log
T)^{1/3}}\right) = T - O\left( \frac{P T}{(H\log T)^{1/3}}\right),
$$
which produces more than $\frac{c}{P^3} T\log T$ zeros ($c>0$) of $F$ taking $H
= \frac{A P^3}{\log T}$ with large enough constant $A$.

Therefore, to prove the main theorem we are left to establish the estimates (\ref{eq5})--(\ref{eq7}). For a complex Hecke character $\psi_j$ those estimates are proved in \cite{Rezvyakova_2016}, since in this case $L_j (s)$ corresponds to an automorphic cusp form of weight one for Hecke congruence group.
For any Hecke L-function the estimate (\ref{eq7})  follows from (\ref{eq6}) as it is proved in \S4 of \cite{Rezvyakova_2016}.
The estimate (\ref{eq6}) follows from the following one
$$
\int\limits_{-\infty}^{+\infty} \left|\Lambda_j \left( \frac12 + it\right) \eta_j^2 \left(\frac12 +it\right)
\exp\left( \left(\frac{\pi}{2} -\frac{1}{T}\right) t\right) \right|^2 dt = O\left( T
\right)
$$
as can be seen by Stirling's formula for Gamma-function.
We shall see in the next section that for L-function attached to real Hecke character (i.e., when L-function is a certain product of two Dirichlet L-functions) the above estimate and (\ref{eq5}) follow from the result of lemma~\ref{l0} and theorem \ref{th0}  we are only left to obtain therefore. In the end of this section, we shall introduce some notations that are used throughout the paper.

{\bf Notation.}

\begin{itemize}

\item $\chi_{D} (\cdot) = \left( \frac{-D}{\cdot}\right)$ is the Dirichlet character defined by the Kronecker symbol, where $-D <0$ is the fundamental discriminant.

\item $\chi_{d_1} (\cdot), \chi_{d_2} (\cdot)$ are real primitive Dirichlet characters such that $(d_1, d_2) = 1$, $d_1 d_2 = D$ and $\chi_{d_1}  \chi_{d_2} = \chi_{D}$.

\item For $j=1,2$ and $\Real s >1$,  $L_{\chi_{d_j}} (s) = \sum\limits_{n=1}^{+\infty} \frac{\chi_{d_j} (n)}{n^s}$.

\item The coefficients $r(n)$ of the Dirichlet series $L(s) =\sum\limits_{n=1}^{+\infty} \frac{r (n)}{n^s}$ are defined by the relation $L(s) = L_{\chi_{d_1}} (s)  L_{\chi_{d_2}} (s)$.

\item
$T$ is large enough number,  $\delta = T^{-1}$,
$X$ is a small fixed power of $T$.

\item
$
\phi(u) = \frac{1+ \cos^4\delta}{\cos^4\delta + u^{-4}}.
$

\item $\alpha (\cdot)$ is a multiplicative function defined by the following relation
$L^{-1/2} (s) = \sum\limits_{\nu=1}^{+\infty} \frac{\alpha (\nu)}{\nu^s}$.

\item  $\EuScript{L}(\nu) = \begin{cases} 1 &\text{for }\ \nu < \sqrt{X},
\\
2\dfrac{\log X/\nu}{\log X} & \text{for }\ \sqrt{X} \le \nu \le X, \\
0 &\text{otherwise}.
\end{cases}$.

\item $\beta(\nu) = \alpha(\nu) \EuScript{L}(\nu)$.

\item $\Lambda (s) = \biggl(\frac{2\pi}{\sqrt
D}\biggr)^{-s}\Gamma(s)L (s)$.

\item $F(t) = \Lambda \left( \frac12 + it\right) \left| \sum\limits_{\nu\le X} \frac{\beta(\nu)}{n^{\frac12 +it}}\right|^{2}
\exp\left( \left(\frac{\pi}{2} -\frac{1}{T}\right) t\right)$ is a real-valued function.

\item $I(t, H) = \int\limits_{t}^{t+H} F(u) du$.

\item For positive integers $m$ and $d$ the expression $m \mid d^{\infty}$ means that all prime divisors of $m$ divide $d$.

\item $(a,b)$ denotes the greatest common divisor of the two positive integers $a$ and $b$.

\item $\prod\limits_{p}$ denotes the product over all prime numbers.

\item \label{p1}
$K_{1,1} (m, z) = \frac{1}{(m, d_1^{\infty})^{z}} \chi_{d_1} \left( (m, d_2^{\infty})\right) \chi_{d_2} \left( (m, d_1^{\infty}) \right)
\times
\prod\limits_{\substack{(p, D)=1, \\ p^{\alpha} || m}} \left( 1- \frac{1}{p^{2}} \right)^{-1} \left( 1 - \frac{\chi_D (p)}{p^{z}} \right)^{-1} \left( 1 - \frac{\chi_D (p)}{p} \right)
\times  \chi_{d_1} (p^{\alpha})\left( 1-\frac{1}{p^{z+1}}  + \frac{\chi_D (p^{\alpha+1})}{p^{\alpha z+1}}  \left( 1-\frac{1}{p^{z-1}} \right) \right), $

\item
$K_{2,2} (m, z) = \frac{1}{(m, d_2^{\infty})^{z}} \chi_{d_1} \left( (m, d_2^{\infty})\right)  \chi_{d_2} \left( (m, d_1^{\infty}) \right) \times \prod\limits_{\substack{(p, D)=1, \\ p^{\alpha} || m}} \left( 1- \frac{1}{p^{2}} \right)^{-1} \left( 1 - \frac{\chi_D (p)}{p^{z}} \right)^{-1} \left( 1 - \frac{\chi_D (p)}{p} \right) \times \chi_{d_2} (p^{\alpha}) \left( 1-\frac{1}{p^{z+1}}  + \frac{\chi_D (p^{\alpha+1})}{p^{\alpha z+1}}  \left( 1-\frac{1}{p^{z-1}} \right) \right), $

\item
For $m_1, m_2$ such that $(m_1, m_2)=1$, $d_2\mid m_1$, $d_1\mid m_2$ we set

$K_{1,2} (m_1, m_2, z) = \chi_{d_1} \left( m_1\right) \chi_{d_2} \left( m_2 \right)
\sum\limits_{q\mid (m_1 m_2/ D, D^{\infty})} \frac{1}{q^{z}}\times \\
 \prod\limits_{\substack{(p, D)=1, \\ p^{\alpha} || m_1 m_2}} \left( 1- \frac{\chi^2_{D} (p)}{p^{2}} \right)^{-1} \left( 1 - \frac{1}{p^{z}} \right)^{-1} \left( 1 - \frac{\chi_D (p)}{p} \right)
\left( 1-\frac{\chi_D (p)}{p^{z+1}}  + \frac{\chi_D (p)}{p^{\alpha z+1}}  \left( 1-\frac{\chi_D (p)}{p^{z-1}} \right) \right),
$

\item
$G_N = \prod\limits_{p\mid N} \left( 1+\frac{1}{p^{3/4}}\right)^2$,

\item
$C,C_1, \ldots$ are some absolute constants which may differ at different occurences.

\end{itemize}

\begin{center}
{\bf \S 4. The main assertion}
\end{center}

\begin{lemma}\label{l0}
Let us define
\begin{equation}
\label{eq12-1} G (y) = \biggl| \sum_{n,\nu_1, \nu_2}
\frac{r (n)\beta (\nu_1)\beta (\nu_2)}{\nu_2} \exp\biggl( -\frac{2\pi n
\nu_1}{\sqrt{D}\nu_2} y (\sin\delta+i\cos\delta)\biggr) \biggr|^2.
\end{equation}
Then
$$
\int\limits_{-\infty}^{+\infty} |I (t, H)|^2 dt \le 8 H^2 \int\limits_{1}^{e^{1/H}} G(y) dy + 8 \int\limits_{e^{1/H}}^{+\infty} \frac{G(y)}{\log^2 y} dy,
$$
and also
$$
\int\limits_{-\infty}^{+\infty} |F (t)|^2 dt = \int\limits_{1}^{+\infty} G(y) dy,
$$
where $I (t, H)$ and $F (t)$are defined in the end of \S 3.
\end{lemma}
The proof of this result is contained in \cite{Rezvyakova} (see there lemmas 3 and 4).

\begin{theorem}\label{th0}
For $x \ge 1$ and $G(y)$ given by~\eqref{eq12-1}, let us define
$$
J(x,\theta)=\int_x^{+\infty}G(u)u^{-\theta}\,du.
$$
If $\log^{-1} T \le \theta\le \left(\frac{\log\log T}{\log T} \right)^{2/3}$ and $X \le T^{1/50}$, then the estimates
$$
J(x,\theta) \le J(1, \theta)\ll
\frac{T}{(\theta  \log T)^{1/3}}
$$
hold.
\end{theorem}
The statement of this theorem is the core of the work. Its proof is contained  in \S6  and \S8 (corresponding to estimation of diagonal and non-diagonal terms respectively). Let us derive from it the estimates (\ref{eq5}), (\ref{eq6}) for $\frac{1}{\log T} \le H \le \left( \frac{\log\log T}{\log T}\right)^{2/3}$.

First, we shall obtain the estimate  (\ref{eq6}) from the results of lemma~\ref{l0} and theorem~\ref{th0}. We only need to show due to lemma \ref{l0} that
$$
 \int\limits_{1}^{+\infty} G(y) dy \ll T.
$$
Let us apply theorem~\ref{th0} with
$x=1$ and $\theta=\frac1{\log T}$ and get
$$
\int_1^{T^2}G(u)\,du\ll\int_1^{+\infty}
G(u)u^{-1/\log T}\,du\ll T.
$$
Recalling the definition~\eqref{eq12-1} of $G(y)$, we estimate the
integral over $(T^{2},+\infty)$ as follows:
\begin{align*}
\int_{T^{2}}^{+\infty} G(u)\,du & \ll\sum_{\substack{n_1, n_2
\\
\nu_1,\nu_2, \nu_3, \nu_4}}
\frac{|r(n_1)r(n_2)\beta(\nu_1)\beta(\nu_2)\beta(\nu_3)\beta(\nu_4)|}{\nu_2
\nu_4}
\\
&\qquad\qquad\qquad\times\int_{T^{2}}^{+\infty} \exp\biggl(
-\frac{2\pi}{\sqrt{D}} \biggl(\frac{n_1 \nu_1}{\nu_2}+ \frac{n_2
\nu_3}{\nu_4} \biggr) T^{-1} u \biggr)\,du.
\end{align*}
Using the identity
$$
\int_{T^{2}}^{+\infty}\exp(-ax)\,dx=\frac{e^{-a
T^{2}}}a,
$$
we obtain: {\allowdisplaybreaks
\begin{align*}
\int_{T^{2}}^{+\infty} G(u) du &\ll T
\sum_{\substack{n_1, n_2
\\
\nu_1,\nu_2, \nu_3, \nu_4}}
\frac{|r(n_1)r(n_2)\beta(\nu_1)\beta(\nu_2)\beta(\nu_3)\beta(\nu_4)|}{(n_1
\nu_1 \nu_4 + n_2 \nu_2 \nu_3)}
\\*
&\qquad\qquad\qquad\qquad\qquad\qquad\times \exp\biggl(
-\frac{2\pi}{\sqrt{D}} T \biggl(\frac{n_1 \nu_1}{\nu_2}+
\frac{n_2 \nu_3}{\nu_4} \biggr) \biggr)
\\
& \ll T \sum_{\substack{n_1, n_2
\\
\nu_1,\nu_2, \nu_3, \nu_4}}
\frac{|r(n_1)r(n_2)\beta(\nu_1)\beta(\nu_2)\beta(\nu_3)\beta(\nu_4)|}{\sqrt{n_1
n_2 \nu_1 \nu_2 \nu_3 \nu_4}}
\\*
&\qquad\qquad\qquad\qquad\qquad\qquad\times \exp\biggl(
-\frac{2\pi}{\sqrt{D}} T \biggl(\frac{n_1 \nu_1}{\nu_2}+
\frac{n_2 \nu_3}{\nu_4} \biggr) \biggr)
\\
&\le T \biggl( \sum_{\nu_1,\nu_2 \le X}
\frac{1}{\sqrt{\nu_1 \nu_2}} \sum_{n} \frac{|r(n)|}{\sqrt{n}}
\exp\biggl( -\frac{2\pi}{\sqrt{D}} \frac{n \nu_1}{\nu_2} T
\biggr) \biggr)^2.
\end{align*}
Since $|r(n)|\le\tau(n)$ and $\nu_2\le X\le T^{1/3}$, we have
$$
\sum_{n=1}^{+\infty}\frac{|r(n)|}{\sqrt n}\exp\biggl(
-\frac{2\pi}{\sqrt D}\frac{n\nu_1}{\nu_2} T\biggr)\ll
\exp\bigl(-\sqrt{T}\bigr).
$$
Hence,
$$
\int_{T^{2}}^{+\infty}G(u)du\ll T X\exp
\bigl(-\sqrt{T}\bigr)\ll1.
$$
The proof of the estimate (\ref{eq6}) is now completed.

It follows from lemma~\ref{l0} that to prove (\ref{eq5}) we need to show that
$$
\int_{1}^{e^{1/H}} G(x) dx \ll \frac{T}{(H\log T)^{1/3}}, \quad \int_{e^{1/H}}^{+\infty} \frac{G(x)}{\log^2 x}\,dx  \ll \frac{T H^2}{(H\log T)^{1/3}}.
$$
Setting
$\theta=H$ in theorem~\ref{th0} and noting that $x^{\theta} \ll 1$ for $1\le x\le e^{1/H}$, we obtain
\begin{align*}
\int_{1}^{e^{1/H}} G(x) dx & = -\int_{1}^{e^{1/H}} x^{\theta} \frac{d}{dx} J(x,
\theta) \,dx = -x^{\theta} J(x, \theta) \Big|_{x=1}^{e^{1/H}} +
\theta\int_{1}^{e^{1/H}} x^{\theta-1} J(x, \theta)\,dx
\\
&\ll\frac{T}{(\theta \log T)^{1/3}} \left( 1+ \frac{\theta}{H}\right) \ll\frac{T}{(H \log T)^{1/3}}.
\end{align*}

Since
$$
 \int\limits_{0}^{H} a x^{-a} da =  - H x^{-H} \log^{-1} x -
  (x^{-H} - 1) \log^{-2} x,
$$
then for $x \ge e^{1/H}$ (i.e., when $\log^{-1} x \le H$)\
$$
\log^{-2} x \le 2\int\limits_{0}^{H} a x^{-a} da +  2 x^{-H} H^2.
$$
Therefore,
\begin{align*}
\int_{e^{1/H}}^{+\infty} \frac{G(x)}{\log^2 x}\,dx & \le 2\int\limits_{0}^{H} a J(e^{1/H}, a)  da
+ 2 H^2 J(e^{1/H}, H).
\end{align*}
For $0 \le a \le (\log^{-1} T)$ we shall use the estimate $J(e^{1/H}, a) \le \int\limits_{1}^{+\infty} G(x) dx \ll T$, which we have just obtained. For the remaining interval $a\in (\log^{-1} T, H)$ we shall use the estimate
$J(e^{1/H}, a) \ll \dfrac{T}{(a \log T)^{1/3}}$. Overall, we get
\begin{align*}
\int_{e^{1/H}}^{+\infty} \frac{G(x)}{\log^2 x}\,dx &\ll  \frac{T}{\log^2 T} + \frac{TH^2}{(H\log T)^{1/3}}  \ll  \frac{T H^2 }{(H\log T)^{1/3}}
\end{align*}
since $H \ge \log^{-1} T$.
The estimate (\ref{eq5}) is completed.
Therefore, our aim is to prove theorem~\ref{th0}.

\begin{center}
{\bf \S 5. Estimation of Selberg sums coming from the ``diagonal'' and ``non-diagonal'' terms in case of a real ideal class group Hecke character.}
\end{center}

Lemma \ref{l3} below corresponds to the estimation of the so-called ``diagonal'' term arising in the proof of theorem \ref{th0}. An analogue of this lemma in case of a complex Hecke character can be found in \cite{Rezvyakova_MZ} (lemma 3) or  \cite{Rezvyakova_2016} (lemma 1). Lemma \ref{l4} below corresponds to estimation of the ``non-diagonal'' term which turns out to be at least of the same order as the ``diagonal'' term in contrast to the situation in \cite{Rezvyakova_MZ}, where the ``non-diagonal'' term is very small compared to the ``diagonal'' term. The reason to such a difference is a non-zero main term in the additive problem with coefficients of L-function of a real Hecke character (see lemma \ref{l8}). We shall formulate and prove several auxiliary lemmas in the beginning of this paragraph.

\begin{lemma}\label{l2}
Let
\begin{align*}
&K(m,s) = \prod_{p| m} \biggl( 1+ \frac{r^{2}(p)}{p^{s}} +
\frac{r^{2}(p^2)}{p^{2s}}+\dotsb\biggr)^{-1}
\\
&\qquad\qquad\qquad\qquad\qquad\qquad\times\prod_{p^{\alpha} \| m}
\biggl( r(p^{\alpha}) + \frac{r(p^{\alpha+1})r(p)}{p^{s}} +
\frac{r(p^{\alpha+2})r(p^2)}{p^{2s}}+\dotsb\biggr).
\end{align*}
For fixed $m$, the function $K(m, s)$ is analytic in the region $\Real s >0$ and satisfies the following relations when $\Real s \ge 1/2$:
\begin{align*}
|K(m,s)| &\le \tau_{2000} (m), \\
K(p, s) &= r(p) + O(p^{-\Real s}).
\end{align*}
\end{lemma}

\begin{pf}
By virtue of the equality (\ref{eq21}) given below,
we have for $\Real s > 1$
\begin{align*}
\biggl( 1+ \frac{r^{2}(p)}{p^{s}} +
\frac{r^{2}(p^2)}{p^{2s}}+\dotsb\biggr)^{-1}
=\biggl(1-\frac{\chi_{d_1}^2 (p)}{p^s}\biggr) \biggl(1-\frac{\chi_{d_2}^2 (p)}{p^s}\biggr) \biggl(1-\frac{\chi_D (p)}{p^s}\biggr)^2
\biggl(1-\frac{\chi_D^2 (p)}{p^{2s}}\biggr)^{-1}.
\end{align*}
Hence $K(m,s)$ is analytic in the region $\operatorname{Re}s>0$ for
a fixed~$m$. Furthermore, using the bound $|r(n)|\le \tau(n)$ we find the following estimate when
$\operatorname{Re}s\ge\frac12$:
\begin{align}
\notag |K(m,s)| & \le\prod_{p|m} \biggl( 1-\frac{1}{2} \biggr)^{-1}
\biggl( 1+\frac{1}{\sqrt{2}} \biggr)^4 \tau(m)
\\
\notag &\qquad\qquad\times \prod_{p^{\alpha} \| m} \biggl( 1 +
\frac{\frac{\alpha+2}{\alpha+1}\cdot2}{2^{1/2}} +
\frac{\frac{\alpha+3}{\alpha+1}\cdot3}{2}+
\frac{\frac{\alpha+4}{\alpha+1}\cdot4}{2^{3/2}}+\cdots\biggr)
\\
\notag & \le\tau(m) \prod_{p | m} \biggl( 1 +
\frac{\frac{3}{2}\cdot2}{2^{1/2}} + \frac{\frac{4}{2}\cdot3}{2}+
\frac{\frac{5}{2}\cdot4}{2^{3/2}}+\cdots\biggr) 2 \biggl(
1+\frac{1}{\sqrt{2}} \biggr)^4
\\
\notag & = \tau(m) \prod_{p | m}\biggl( \sum_{n \ge2}
\frac{n(n-1)}{2^{n/2}}\biggr) 2 \biggl( 1+\frac{1}{\sqrt{2}}
\biggr)^4
\\
\label{eq16} &\le\tau(m) \prod_{p | m} 1000 \le\tau(m) \tau_{1000}
(m) \le\tau_{2000} (m).
\end{align}

Further, since by definition $r(p) = \chi_{d_1} (p) + \chi_{d_2} (p)$ for a prime $p$, and $\chi_{d_j} (\cdot)$ are real characters, we find that $r(p) = 0, \pm 1, \pm2$. For $\Real s >1/2$, if $|r(p)| \ne 0$ we  write:
\begin{align*}
K(p, s) &= \biggl(1-\frac{\chi_{d_1}^2 (p)}{p^s}\biggr) \biggl(1-\frac{\chi_{d_2}^2 (p)}{p^s}\biggr) \biggl(1-\frac{\chi_D (p)}{p^s}\biggr)^2
\biggl(1-\frac{\chi_D^2 (p)}{p^{2s}}\biggr)^{-1} \times \\
&\times \biggl( r(p) + \frac{r(p^{2})r(p)}{p^{s}} +
\frac{r(p^{3})r(p^2)}{p^{2s}}+\dotsb\biggr) = r(p) + O(p^{-\Real s}),
\end{align*}
which also holds true when $r(p) = 0$. This completes the proof of the lemma.
\end{pf}

\begin{lemma}\label{l30}
For $K_{l,l} (m, z)$ ($l=1,2$), $K_{1,2} (m_1, m_2,  z)$, defined in the end of \S3, the following formulae hold:
\begin{align}\label{eq211}
K_{l,l} (m, z) &= \frac{1}{(m, d_l^{\infty})^{z}} \chi_{d_l} \left( (m, (D/d_l)^{\infty})\right) \chi_{D/d_l} \left( (m, d_l^{\infty}) \right)
\times \nonumber \\
&\times
\prod\limits_{\substack{(p, D)=1, \\ p^{\alpha} || m}} \left( 1 + \frac{\chi_D (p)}{p} \right)^{-1}
\chi_{d_l} (p^{\alpha}) \times \nonumber \\
&\times  \left( 1 + \frac{\chi_D (p)}{p^{z}} + \frac{\chi_D (p^2)}{p^{2z}} +\ldots  + \frac{\chi_D (p^{\alpha})}{p^{\alpha z}} - \frac{1}{p^{z+1}} \left( 1 + \frac{\chi_D (p)}{p^{z}} + \ldots  + \frac{\chi_D (p^{\alpha-2})}{p^{(\alpha-2) z}}\right)\right).
\end{align}
\begin{align}\label{eq214}
K_{1,2} (m_1, m_2,  z) &=  \chi_{d_1} \left( m_1\right) \chi_{d_2} \left( m_2 \right)
\sum\limits_{q\mid (m_1 m_2/ D, D^{\infty})} \frac{1}{q^{z}}
\prod\limits_{\substack{(p, D)=1, \\ p^{\alpha} || m_1 m_2}} \left( 1+ \frac{\chi_{D} (p)}{p} \right)^{-1} \times \nonumber\\
&\times
\left( 1 + \frac{1}{p^{z}} + \frac{1}{p^{2z}} +\ldots  + \frac{1}{p^{\alpha z}} - \frac{\chi_{D} (p)}{p^{z+1}} \left( 1 + \frac{1}{p^{z}} + \ldots  + \frac{1}{p^{(\alpha-2) z}}\right)
\right).
\end{align}
Also,  for a prime $p$ posessing the property $(p, D) = 1$, we have that
\begin{equation}\label{eq18}
\begin{split}
K_{l,l} (p, z) = \left( \chi_{d_l} (p) + \frac{\chi_{D/d_l} (p)}{p^{ z}}\right)  \left( 1 + \frac{\chi_D (p)}{p} \right)^{-1}.
\end{split}
\end{equation}
and, for $p\mid D$ and $\Real z \ge 0$, that
\begin{equation}\label{eq19}
\begin{split}
|K_{l,l} (p^{\alpha}, z)| \le 1
\end{split}
\end{equation}
\end{lemma}

\begin{pf}
The equality (\ref{eq211}) is easy to prove. We observe that
\begin{align*}
-\frac{1}{p^{z+1}}  + \frac{\chi_D (p^{\alpha+1})}{p^{\alpha z +1}} &= -\frac{1}{p^{z+1}}  \left( 1 - \frac{\chi_D (p^{\alpha-1})}{p^{(\alpha-1) z}}\right) \\
&= -\frac{1}{p^{z+1}}  \left( 1 - \frac{\chi_D (p)}{p^{z}}\right) \left( 1 + \frac{\chi_D (p)}{p^{z}} + \ldots + \frac{\chi_D (p^{\alpha-2})}{p^{(\alpha-2) z}}\right)
\end{align*}
and that similar calculations hold for the difference
$$
1-\frac{\chi_D (p^{\alpha+1})}{p^{(\alpha+1) z}} = \left( 1 - \frac{\chi_D (p)}{p^{z}}\right) \left( 1 + \frac{\chi_D (p)}{p^{z}} + \ldots + \frac{\chi_D (p^{\alpha})}{p^{(\alpha) z}}\right).
$$
Also, for $(p, D) = 1$ we have
$$
 \left( 1- \frac{1}{p^{2}} \right)^{-1} \left( 1 - \frac{\chi_D (p)}{p} \right)  = \left( 1 + \frac{\chi_D (p)}{p} \right)^{-1}.
$$
All this and the definition of $K_{l,l} (m, z)$ imply (\ref{eq211}). The formulae (\ref{eq18}) and
(\ref{eq19}) easily follow from (\ref{eq211}). The proof of (\ref{eq214}) is similar to that of (\ref{eq211}).
\end{pf}

\begin{lemma}\label{l6}
For $\Real z\ge 0$, we have
\begin{align}\label{eq23}
|K_{l,l}(p^{\alpha}, z)| \le \left( 1 - \frac{1}{p} \right)^{-1}  \left(\alpha +1 + \frac{\alpha-1}{p} \right)  \le 3\alpha +1
\end{align}
and
\begin{equation}\label{eq20}
\begin{split}
|K_{l,l}(p^{\alpha}, z)| &\le 3 \tau(p^{\alpha}) \text{ for } \alpha \ge 2,\\
|K_{l,l} (m, z)| &\le \tau^2 (m), \\
|K_{l,j} (m_1, m_2, z)| &\le \tau^2 (m_1 m_2), \\
\left|\frac{d}{dz} K_{l,l} (m, z)\right| &\le \tau^2 (m) \log m.
\end{split}
\end{equation}
Moreover, if $z= it$, $0 \le t \le 1$, $\chi_{D} (2) = -1$, then
\begin{equation}\label{eq25}
\begin{split}
|K_{l,l} (4, z)|\le 1.
\end{split}
\end{equation}

\end{lemma}

\begin{pf}
For  $\Real z >0$, from (\ref{eq18}),  (\ref{eq19})  we have
$$
|K_{l,l} (p, z)| \le 2\cdot 2 = 4 = \tau^2 (p).
$$
We can write the following estimate from above when $\Real z \ge 0$ and $p\ge 2$:
\begin{equation*}
\begin{split}
|K_{l,l}(p^{\alpha}, z)| \le \left( 1 - \frac{1}{p} \right)^{-1}  \left(\alpha +1 + \frac{\alpha-1}{p} \right) \le 3 \alpha +1.
\end{split}
\end{equation*}
Thus, for $\alpha \ge 2$, we get
 \begin{equation*}
\begin{split}
|K_{l,l}(p^{\alpha}, z)| \le 3 (\alpha +1) \le 3\tau(p^{\alpha}) \le \tau^2(p^{\alpha}).
\end{split}
\end{equation*}
Since $K_{l,l}(\cdot, z)$ is a multiplicative function, we obtain the following inequality for $\Real z \ge 0$:
\begin{equation*}
\begin{split}
|K_{l,l}(m, z)| \le \tau^2(m).
\end{split}
\end{equation*}

Similarly, for $l\ne j$, we obtain from (\ref{eq214}) that
\begin{equation*}
\begin{split}
&|K_{l,j}(m_1, m_2, z)| \le
\tau((m_1 m_2, D^{\infty})) \prod\limits_{\substack{(p, D) = 1, \\ p^{\alpha} || m_1 m_2}} 2 \left(\alpha +1 + \frac{\alpha-1}{p} \right) \le \tau^2(m_1 m_2).
\end{split}
\end{equation*}

Now, if $(2, D)= 1$, $\Real z \ge 0$, then we get from (\ref{eq211})
\begin{equation*}
\begin{split}
&|K_{l,l}(4, z)| \le  2 \left| 1 + \frac{\chi_D (2)}{2^{z}}  + \frac{\chi_D (2^2)}{2^{2z}} -  \frac{1}{2^{z+1}} \right| .
\end{split}
\end{equation*}
For $z = it$ ($t \in \bbR$), set $2^{-z}  =  e^{-it \log 2} = a + ib$ ($a = \cos (t \log 2)$, $b = -\sin (t \log 2) = -\sqrt{1-a^2}$). If $\chi_{D} (2) = -1$ and
$0 \le t \le 1$, we therefore obtain
\begin{equation*}
\begin{split}
&\left| 1 + \frac{\chi_D (2)}{2^{z}}  + \frac{\chi_D (2^2)}{2^{2z}} -  \frac{1}{2^{z+1}} \right|^2  = \left| 1 - \frac{3}{2} (a + ib)  + (a+ib)^2 \right|^2
=
\left| - \frac{3}{2} a + 2 a^2 - i \frac{3}{2} b + i 2ab \right|^2\\
&=\left| a + ib \right|^2 \left| 2a  - \frac{3}{2} \right|^2 = \left| 2a  - \frac{3}{2} \right|^2.
\end{split}
\end{equation*}
Since $0.7692 \le \cos (\log 2) \le a \le 1$, then $\left| 2a  - \frac{3}{2}  \right| \le \frac{1}{2}$. Hence, (\ref{eq25}) holds true.

Now, for $\Real z\ge 0$, let us estimate the derivative $ \dfrac{d}{dz} K_{l,l} (m, z)$ by means of the following formula
\begin{align*}
\left| \frac{d}{dz} K_{l,l} (m, z) \right| \le \left | K_{l,l} (m, z)  \frac{d}{dz} \log K_{l,l} (m, z)  \right|.
\end{align*}
Using (\ref{eq211}) we find
\begin{align*}
&\frac{d}{dz} \log K_{l,l} (m, z) = -\log (m, d_l^{\infty})  \\
&-
\sum\limits_{\substack{(p, D)=1, \\ p^{\alpha} || m}}  \log p \left( 1 + \frac{\chi_D (p)}{p^{z}} + \frac{\chi_D (p^2)}{p^{2z}} +\ldots  + \frac{\chi_D (p^{\alpha})}{p^{\alpha z}} - \frac{1}{p^{z+1}} \left( 1 + \frac{\chi_D (p)}{p^{z}} + \ldots  + \frac{\chi_D (p^{\alpha-2})}{p^{(\alpha-2) z}}\right)\right)^{-1}  \times \\
&\times \left( \frac{\chi_D (p)}{p^{z}} + \frac{2\chi_D (p^2)}{p^{2z}} +\ldots  + \frac{\alpha\chi_D (p^{\alpha})}{p^{\alpha z}} - \frac{1}{p} \left(\frac{1}{p^z} + \frac{2 \chi_D (p)}{p^{2z}} + \ldots  + \frac{(\alpha-1)\chi_D (p^{\alpha-2})}{p^{(\alpha-1) z}}\right)\right).
\end{align*}
Hence,
\begin{equation*}
\begin{split}
&\left| \frac{d}{dz} K_{l,l}(m, z) \right| \le   \left| K_{l,l}(m, z) \right| \log (m, d_l^{\infty})  + \sum\limits_{\substack{p^{\alpha} || m, \\ (p, D) = 1}} \left| 1+\frac{\chi_D (p)}{p} \right|^{-1}  \left| K_{l,l}(m/p^{\alpha}, z) \right|  \times \\
&\times \log p
\left|\frac{\chi_D (p)}{p^{z}}  + \ldots +  \frac{\alpha \chi_D (p^{\alpha})}{p^{\alpha  z}} -
\frac{1}{p} \left( \frac{1}{p^{z}}  + \ldots +  \frac{(\alpha-1) \chi_D (p^{\alpha-2})}{p^{(\alpha-1)  z}}\right) \right| \\
&\le \left| K_{l,l}(m, z) \right| \log (m, d_l^{\infty}) + \sum\limits_{\substack{p^{\alpha} || m, \\ (p, D) = 1}} \left( 1-\frac{1}{p} \right)^{-1}  \log p \left| K_{l,l}(m/p^{\alpha}, z) \right|  \times \\
&\times
\left(1 + 2 +\ldots +  \alpha +
\frac{1}{p} \left( 1 + 2+ \ldots +  (\alpha-1) \right) \right) \\
&\le  \left| K_{l,l}(m, z) \right|\log (m, d_l^{\infty}) + \sum\limits_{\substack{p^{\alpha} || m, \\ (p, D) = 1}} 2 \log p \left| K_{l,l}(m/p^{\alpha}, z) \right|
\left(\alpha (\alpha +1)/2 +
\frac{1}{p}  \alpha(\alpha-1)/2 \right) \\
&\le \tau^2 (m) \log (m, d_l^{\infty}) + \sum\limits_{\substack{p^{\alpha} || m, \\ (p, D) = 1}} 2 (\log p) \tau^2 (m/p^{\alpha}) \alpha (\alpha +1) \frac34 \\
&\le
\tau^2 (m) \log (m, d_l^{\infty}) + \sum\limits_{\substack{p^{\alpha} || m, \\ (p, D) = 1}} (\alpha \log p) \tau^2 (m/p^{\alpha}) 2 (\alpha +1)
\le \tau^2 (m) \log m
\end{split}
\end{equation*}
since $2 (\alpha +1) \le \tau^2 (p^{\alpha})$ for a prime $p$.
The lemma is completely proved.
\end{pf}

\begin{lemma}\label{l10}
For a prime $p$ such that $(p, d_1 d_2) = 1$, the following relations for the coefficients of the series $\sum\limits_{n=1}^{+\infty} \dfrac{\alpha(n)}{n^s} =L_{\chi_{d_1}}^{-1/2} (s) L_{\chi_{d_2}}^{-1/2} (s)$ hold:
\begin{equation*}
\begin{split}
&\alpha (p) = -\frac12 (\chi_{d_1} (p)+ \chi_{d_2} (p)), \\
&\alpha (p^{2}) = -\frac18  (\chi^2_{d_1} (p)+ \chi^2_{d_2} (p)) +\frac14 \chi_{D} (p), \\
&\alpha (p^{k}) = 0  \text{ for odd } k\ge 3, \\
&|\alpha (p^{k})| \le\frac12 \cdot \frac34  \cdot \ldots  \cdot \frac{k-3}{k-2} \cdot \frac{1}{k}  \le \frac{1}{2k} \text{ for even } k\ge 4.
\end{split}
\end{equation*}
\end{lemma}

\begin{pf}
Note that $L_{\chi_{d_j}}^{-1/2} (s) = \prod\limits_{p} \left(1 - \dfrac{\chi_{d_j} (p)}{p^{s}} \right)^{1/2}$ and, thus,
the equality for $\alpha (p)$ follows from the Taylor series expansion
$$
\left(1 - \dfrac{\chi_{d_j} (p)}{p^{s}} \right)^{1/2} = 1 -\frac{1}{2} \dfrac{\chi_{d_j} (p)}{p^{s}} - \ldots.
$$
Now, we prove further
statements of the lemma.  If $(p, d_1 d_2)=1$ and  $\chi_{d_1} , \chi_{d_2}$ are real characters, then there are two options:
\begin{equation*}
\left(1 - \dfrac{\chi_{d_1} (p)}{p^{s}} \right)^{1/2} \left(1 - \dfrac{\chi_{d_2} (p)}{p^{s}} \right)^{1/2} =
\begin{cases}
\left(1 - \dfrac{1}{p^{2s}} \right)^{1/2} & \text{ if } \  \chi_{d_1} (p) \chi_{d_2} (p) = -1, \\
\left(1 - \dfrac{\chi_{d_1} (p)}{p^{s}} \right)& \text{ if } \  \chi_{d_1} (p)  = \chi_{d_2} (p).
\end{cases}
\end{equation*}
In the second option $\alpha (p^k) = 0$ for $k\ge 2$ and the proof of the lemma is completed  in this case.
In the first option the statements of the lemma for $\alpha(p^2)$ and $\alpha(p^k)$ for odd $k$ are clear.
For even powers, using the Taylor series expansion
$$
\left(1 - \dfrac{1}{p^{2s}} \right)^{1/2} = \sum\limits_{k=1}^{+\infty}
\frac{\alpha(p^{2k})}{p^{2ks}},
$$
we get for $k \ge 2$ that
\begin{equation*}
\begin{split}
\alpha(p^{2k}) = -\frac12\cdot \frac{1}{2} \cdot \frac{3}{2} \ldots \frac{2k-3}{2} \cdot\frac{1}{k!} =
-\frac{1}{2k}\cdot \frac{1}{2\cdot 1} \cdot \frac{3}{2\cdot 2} \ldots \frac{2k-3}{2 \cdot (k-1)},
 \end{split}
\end{equation*}
which completes the proof of the lemma.

\end{pf}

\begin{lemma}\label{l11}
For $z = it$, $0\le t\le 1$, the following estimate holds
\begin{equation*}
\begin{split}
\frac{|\alpha (2) K_{l,l}(2, z)|}{2} + \frac{|\alpha (4) K_{l,l}(4, z)|}{4}  \le
\begin{cases}
\frac23 & \text{ if } \ (2, D) = 1, \\
\frac13  & \text{ if } \ (2, D) > 1
\end{cases}
\end{split}
\end{equation*}
and for a prime $p\ge 3$ and $\Real z \ge 0$ the estimate
\begin{equation*}
\begin{split}
\frac{|\alpha (p) K_{l,l}(p, z)|}{p} + \frac{|\alpha (p^2) K_{l,l}(p^2, z)|}{p^2} \le
\begin{cases}
\frac59 & \text{ if } \ (p, D) = 1, \\
\frac15  & \text{ if } \ (p, D) > 1
\end{cases}
\end{split}
\end{equation*}
is valid.
\end{lemma}

\begin{pf}
If $\chi_{D} (p) = -1$ (i.e $\chi_{d_1} (p) = -\chi_{d_2} (p)$), then  $\alpha (p) = 0$ (see lemma~\ref{l10}) and, therefore,
$$
\alpha (p) K_{l,l}(p, z) = 0.$$

If $\chi_{D} (p) = 1$, we see from the definition of $\alpha(\cdot)$ and formula (\ref{eq18}) that
\begin{equation*}
\begin{split}
\alpha (p) K_{l,l}(p, z)  &= -\frac12 (\chi_{d_1} (p)+ \chi_{d_2} (p)) \left( \chi_{d_l} (p) + \frac{\chi_{D/d_l} (p)}{p^{ z}}\right) \left( 1 + \frac{\chi_D (p)}{p} \right)^{-1}   \\
&= -\frac12 (1+ \chi_{D} (p) ) \left( 1 + \frac{1}{p^{ z}}\right)  \left( 1 + \frac{\chi_D (p)}{p} \right)^{-1} \\
&= - \left( 1 + \frac{1}{p^{ z}}\right) \left( 1 + \frac{1}{p} \right)^{-1} .
\end{split}
\end{equation*}

In case $\chi_{D} (p) = 0$ (i.e. $(p, D) >1$) we shall employ the estimates $|K_{l,l}(p, z)| \le 1$ for $\Real z \ge 0$ (see (\ref{eq19})) and the estimate  $|\alpha (p)| \le \frac12$ (by definition). Consequently,
if $\Real z \ge 0$ we get the following  inequality:
\begin{equation}\label{eq1004}
|\alpha (p) K_{l,l}(p, z)| \le
\begin{cases}
0 & \text{ for } \ \chi_{D} (p) = -1, \\
2 \cdot \left(1 +\frac1p \right)^{-1} & \text{ for } \ \chi_{D} (p) = 1, \\
\frac12 & \text{ for } \ \chi_{D} (p) = 0.
\end{cases}
\end{equation}

Now, if $\chi_{D} (p) = 1$, then $\alpha (p^2)= 0$ and, hence,  $\alpha (p^2) K_{l,l}(p^2, z) = 0$.
Consequently, for $\Real z \ge 0$, by  lemma~\ref{l10} and formulae (\ref{eq19}), (\ref{eq23}) we conclude that
\begin{equation}\label{eq1002}
|\alpha (p^2) K_{l,l}(p^2, z)| \le
\begin{cases}
3+\frac1p & \text{ for } \  \chi_{D} (p) = -1, \\
0, & \text{ for } \ \chi_{D} (p) = 1, \\
\frac18 & \text{ for } \  \chi_{D} (p) = 0.
\end{cases}
\end{equation}
Therefore, if $(2, D) >1$ then
\begin{equation*}
\begin{split}
\frac{|\alpha (2) K_{l,l}(2, z)|}{2} + \frac{|\alpha (4) K_{l,l}(4, z)|}{4} \le \frac14 + \frac{1}{32} <\frac13.
\end{split}
\end{equation*}
If $p \ge 3$ and $(p, D) >1$, then
\begin{equation*}
\begin{split}
\frac{|\alpha (p) K_{l,l}(p, z)|}{p} + \frac{|\alpha (p^2) K_{l,l}(p^2, z)|}{p^2} \le \frac16 + \frac{1}{72} <\frac15.
\end{split}
\end{equation*}
If $(2, D) = 1$ then one of the two values among $\alpha (2) K_{l,l}(2, z)$ and $\alpha (4) K_{l,l}(4, z)$ is zero.
And hence, for $\Real z \ge 0$, we have
\begin{equation*}
\begin{split}
\frac{|\alpha (2) K_{l,l}(2, z)|}{2} + \frac{|\alpha (4) K_{l,l}(4, z)|}{4} \le \max \left( \frac13, \frac23, \frac78\right) = \frac78.
\end{split}
\end{equation*}
Similarly, if $p \ge 3$, $(p, D) = 1$ and  $\Real z \ge 0$, then
\begin{equation*}
\begin{split}
\frac{|\alpha (p) K_{l,l}(p, z)|}{p} + \frac{|\alpha (p^2) K_{l,l}(p^2, z)|}{p^2} \le \max \left( \frac{2}{p+1}, \frac{3+1/p}{p(p-1)}\right) \le \max \left( \frac12, \frac59\right) = \frac59.
\end{split}
\end{equation*}
For $z=it, 0\le t \le 1$, we get the following finer estimate with the help of (\ref{eq25}):
\begin{equation*}
\begin{split}
\frac{|\alpha (2) K_{l,l}(2, z)|}{2} + \frac{|\alpha (4) K_{l,l}(4, z)|}{4} \le \max \left( \frac23, \frac18\right) = \frac23
\end{split}
\end{equation*}
and the lemma is completely proved now.
\end{pf}

Now we shall formulate  two main statements of this paragraph.
\begin{lemma}\label{l3}
For
complex $z$, let us
define $S(z)$ by the relation
$$
S(z):= S_{X}(z) =\sum_{\nu_1,\dots,\nu_4\le X}
\frac{\beta(\nu_1)\beta(\nu_2)\beta(\nu_3)\beta(\nu_4)}{\nu_2
\nu_4}\biggl(\frac q{\nu_1\nu_3}\biggr)^{1-z}K\biggl(
\frac{\nu_1\nu_4}q,1-z\biggr) K\biggl(\frac{\nu_2
\nu_3}q,1-z\biggr),
$$
where $q=(\nu_1\nu_4,\nu_2\nu_3)$ and $K(\cdot, \cdot)$ is defined in lemma \ref{l2}.
There exists an absolute constant $C \ge 1$ such that for $|z| \le C^{-1}$ with $\Real z \ge - \frac{1}{\log X}$ the estimate
$$
S(z)\ll\frac{X^{2 |z|}}{(\log X)^{2}}
$$
holds.

\end{lemma}

\begin{lemma} \label{l4}
Let us define $S_{l,l}(z)$ ($l=1, 2$) by the formula
$$
S_{l,l}(z) := S_{l,l}(z, X) =\sum_{\nu_1,\dots,\nu_4\le X}
\frac{\beta(\nu_1)\beta(\nu_2)\beta(\nu_3)\beta(\nu_4)}{\nu_2
\nu_4}\biggl(\frac q{\nu_1\nu_3}\biggr)^{1-z}K_{l,l}\biggl(
\frac{\nu_1\nu_4}q,z\biggr)K_{l,l}\biggl(\frac{\nu_2
\nu_3}q, z\biggr),
$$
where $q=(\nu_1\nu_4,\nu_2\nu_3)$ (and $K_{l,l} (m, z)$ are defined on page \pageref{p1}).
Then, for $0 \le \Real z \le  \frac{\log\log X}{\log X}$,  the estimate
\begin{align*}
|S_{l,l}(z)| &\ll
\begin{cases}
&\log^4 (|z|+3), \\
& (|z| + \log^{-1} X)^2 \text{ if } |\Image z|\le 1
\end{cases}
\end{align*}
holds true. And for $z=it$ with
\begin{align}\label{eq1006}
\log^{-1} X \le t \le (\log X)^{-1/3} (\log\log X)^{-2/3}
\end{align}
the following asymptotic formula holds:
\begin{align}\label{eq510}
S_{l,l}(z) &= \sum_{d \le X^2}
\sum_{m|d} \mu(m) \biggl( \frac{d}{m} \biggr)^{1-z} (\text{Main term})^2  +
O \left(  |z|^{3/2} \log^{-1/2} X\right),
\end{align}
where
\begin{align}
 \text{Main term}&=
\sum_{\substack{\delta_1 \delta_4
\equiv0 \ (\operatorname{mod}d) \nonumber
\\
\delta_j | d^{\infty}}} \frac{\alpha(\delta_1)
\alpha(\delta_4)}{\delta_1^{1-z} \delta_4} K_{l,l}\biggl(
\frac{\delta_1 \delta_4 m}{d}, z\biggr) \nonumber
\\
&\times \sum\limits_{(n,d)=1} \frac{1}{n^{2-z}} \sum\limits_{(r, nd)=1}  \frac{\mu(r)}{r^{2-z}}\sum_{\substack{k_1, k_4 | n^{\infty}, \\ (k_1, k_4)=1}} \frac{\alpha(n k_1) \alpha(n k_4) K_{l,l}(n^2 k_1 k_4, z)}{k_1^{1-z} k_4} \nonumber \\
&\times \sum_{l_j | r^{\infty}} \frac{\alpha(r l_1) \alpha(r l_4)}{l_1^{1-z}
l_4} K_{l,l}(r l_1, z) K_{l,l}( r l_4, z) \cdot \frac{C_l |z| (R^{(l,l)}_{z, n d r} H_{z,1}^{(l,l)} (1))^2}{\log^2 X} \nonumber \\
&\times
\left( \log^{1/2} \left(\max\left( 1, X/(\delta_1 k_1 l_1 n r) \right) \right)- \log^{1/2} ( \max\left( 1, \sqrt{X}/(\delta_1 k_1 l_1 n r)\right) )\right)\times  \nonumber\\
&\times \left(\log^{1/2} \left(\max\left( 1, X/(\delta_4 k_4 l_4 n r)\right) \right)-\log^{1/2}(\max\left( 1,\sqrt{X}/(\delta_4 k_4 l_4 n r)\right) )\right), \label{eq509}
\end{align}
$C_l$ is some absolute constant, and $H_{z,N}^{(l,l)} (\cdot)$, $R^{(l,l)}_{z, n d r} (\cdot)$ are defined in (\ref{eq101}), (\ref{eq31}).
\end{lemma}

To prove lemmas \ref{l3}, \ref{l4} we shall use some extra statements.
\begin{lemma}\label{l5}
Let $X_1 \ge 1$,
$N$ be any positive integer, and set for  complex numbers $z, \gamma$
$$
S_{z} (X_1, \gamma, N) = \sum_{\substack{ \nu\le X_1
\\
(\nu, N)=1}} \frac{\alpha(\nu) K(\nu,
1-z)}{\nu^{1-\gamma}} \log \frac{X_1}{\nu}.
$$
If $\Real z \le 1/2$ and  $|\gamma| \le C^{-1}$, where $C$ is a sufficiently large absolute constant, then
the estimate
$$
S_{z} (X_1, \gamma, N) \ll X_1^{|\gamma|} G_{N}
$$
holds true.
\end{lemma}

\begin{pf}

We may suppose that $X_1$ is large enough (otherwise the statement of the lemma is trivial). The proof of the lemma in case of a real character is carried out similarly as in the complex case (see  \cite{Rezvyakova_2016}, lemma 1). First, for $\Real s
>1$, we define the generating function
\begin{equation*}
\begin{split}
h_{z, N} (s) &= \sum\limits_{(n, N)=1} \frac{\alpha(n) K(n,
1-z)}{n^{s}} = \prod\limits_{ (p, N)=1} \left( 1 +
\frac{\alpha(p) K(p, 1-z)}{p^{s}} + \frac{\alpha(p^2) K(p^2, 1-z)}{p^{2s}} +\ldots \right) \\
&= \prod_{p} \left( 1 +
\frac{r^2(p)}{p^{s}}+
\frac{r^2(p^2)}{p^{2s}}+ \ldots\right)^{-1/2}\times \prod\limits_{p\mid N} \left( 1 +
\frac{r^2(p)}{p^{s}}+ \frac{r^2(p^2)}{p^{2s}}+
\ldots\right)^{1/2} \times \\
& \times \prod_{
(p, N)=1} \left( 1 + \frac{r^2(p)}{p^{s}}+
\frac{r^2(p^2)}{p^{2s}}+ \ldots\right)^{1/2} \times \\
&\times \left( 1 + \frac{\alpha(p) K(p, 1-z)}{p^{s}} + \frac{\alpha(p^2) K(p^2, 1-z)}{p^{2s}} +\ldots \right) .
\end{split}
\end{equation*}
We find from lemma \ref{l2} that the product
\begin{equation*}
\begin{split}
P_z (s) &=\prod_{
(p, N)=1} \left( 1 + \frac{r^2(p)}{p^{s}}+
\frac{r^2(p^2)}{p^{2s}}+ \ldots\right)^{1/2} \times \\
&\times \left( 1 + \frac{\alpha(p) K(p, 1-z)}{p^{s}} + \frac{\alpha(p^2) K(p^2, 1-z)}{p^{2s}} +\ldots  \right) \\
&= \prod\limits_{(p, N)=1} \left( 1 +
\frac{r^2(p)}{2 p^{s}} + O_{\sigma}\left( \frac{1}{p^{2\sigma}}\right)
\right) \left( 1 -
\frac{r^2(p)}{2 p^{s}} + O_{\sigma}\left( \frac{1}{p^{\sigma+1-\Real z}}\right) + O_{\sigma}\left( \frac{1}{p^{2\sigma}}\right)\right)\\
&= \prod\limits_{ (p, N)=1} \left( 1 + O_{\sigma}\left(
\frac{1}{p^{2\sigma}}\right) + O_{\sigma}\left(
\frac{1}{p^{\sigma+1-\Real z}}\right)\right)
\end{split}
\end{equation*}
defines an analytic function for $\Real z \le 1/2$ in the half-plane $\Real s
> 1/2$.

Now, for $\Real s>1$ and
$$
D(s) =  \prod\limits_{p} \left( 1 +
\frac{r^2(p)}{p^{s}}+ \frac{r^2(p^2)}{p^{2s}}+
\ldots\right) = \sum\limits_{n=1}^{+\infty} \frac{r^2(n)}{n^s},
$$
the following formula holds (see \cite{Iwaniec_book}, p.232):
\begin{equation}\label{eq21}
D(s) = L_{\chi_{1}^2} (s) L_{\chi_2^2} (s) L^2_{\chi_D} (s) / L_{\chi_D^2} (2s).
\end{equation}
Hence, the following identity holds for
the generating function:
\begin{equation*}
\begin{split}
h_{z, N} (s) = (D(s))^{-1/2} P_z (s) G_{N} (s),
\end{split}
\end{equation*}
where
$$
G_{N} (s) = \prod\limits_{p\mid N} \left( 1 +
\frac{r^2(p)}{p^{s}}+ \frac{r^2(p^2)}{p^{2s}}+
\ldots\right)^{1/2}.
$$
Note, that for fixed $\sigma>1/2$ and $\Real s \ge \sigma$ the following estimate holds:
\begin{equation*}
\begin{split}
|G_{N} (s)| \le G_{N} (\sigma) \ll \prod\limits_{p\mid N} \left( 1 +
\frac{r^2(p)}{p^{\sigma}}\right)^{1/2} \ll \prod\limits_{p\mid N} \left( 1 +
\frac{r^2(p)}{2p^{\sigma}}\right) \ll \prod\limits_{p\mid N} \left( 1 +
\frac{2}{p^{\sigma}}\right) \ll \prod\limits_{p\mid N} \left( 1 +
\frac{1}{p^{\sigma}}\right)^2.
\end{split}
\end{equation*}
Now an application of Perron's summation formula gives
\begin{equation*}
\begin{split}
S_{z}(X_1,\gamma,N) &= \frac{1}{2\pi i}
\int\limits_{1-i\infty}^{1+i\infty} h_{z, N} (s+1-\gamma)
\frac{X_1^{s}}{s^2} ds \\
&= \frac{X_1^{\gamma}}{2\pi i}\int\limits_{1-i\infty}^{1+i\infty}
\frac{X_1^{s}}{(s+\gamma)^2}\frac{P_z (s+1) G_{N} (s+1)
}{\left( D(1+s) \right)^{1/2}} ds.
\end{split}
\end{equation*}
Let us move the path of integration in the last formula from the line $\Real s =1$ to the contour,
constructed by the union of the following ones:
$$K_1=\{ \Real s =0 , |\Image s| \ge 1\},
$$
$$K_2=\{  -2C^{-1} \le\Real s\le 0, |\Image s|=1\},
$$
$$K_3= \{\Real s= -2C^{-1} , |\Image s| \le 1 \}.
$$
Here $C > 4$ is large enough constant so that there are no zeros of $D(s)$ for $\Real s > 1-\frac{4}{C}, |\Image s| \le 2$. Note, that $D^{-1/2} (1+s)$ is analytic to the right of the new contour (the regularity at $s=0$ is provided by the truth of the formula $D^{-1} (1+s) = s^2 (c+O(s)), \; c\ne 0$, in the vicinity of $s=0$). We shall also pass a pole of the integrand at $s=-\gamma$.

Thanks to the following bound for Dirichlet L-functions
$$
L^{-1}_{\chi} (1 +it) \ll \log (|t|+2),
$$
on $K_1$ we shall use the estimate
$$
D(1+it)^{-1/2} \ll \log^2 (|t|+3).
$$
Also, on $K_1, K_2, K_3$  the following bounds hold:
$$
P_z (1+s) \ll 1 \quad \text{for} \quad  \Real z \le 1/2,
$$
$$
G_{N} (1+s) \ll G_{N} (1+\sigma) \ll \prod_{p|N} \left( 1+\frac{1}{p^{3/4}}\right)^2 := G_N.
$$
On  $K_1$  the corresponding integral can be estimated from above as
\begin{equation*}
\begin{split}
X_1^{\Real\gamma} G_{N} \int\limits_{1}^{+\infty} \frac{\log^2 (t+3)}{(t-|\gamma|)^2} dt
\ll
X_1^{\Real\gamma} G_{N}.
\end{split}
\end{equation*}
On  $K_2$  the corresponding integral can be estimated from above as
\begin{equation*}
\begin{split}
X_1^{\Real\gamma} G_{N} \int\limits_{-C^{-1}}^{0} X_1^{\sigma}  d\sigma
\ll X_1^{\Real\gamma} G_{N}.
\end{split}
\end{equation*}
On  $K_3$  the corresponding integral can be estimated from above as
\begin{equation*}
\begin{split}
X_1^{\Real\gamma-C^{-1}} G_{N} \int\limits_{0}^{2} \frac{dt}{t^2+ 1} dt
\ll X_1^{\Real\gamma-C^{-1}} G_{N} \ll X_1^{\Real\gamma} G_{N}.
\end{split}
\end{equation*}
The contribution from the residue is the following (suppose $X_1$ is large enough compared to $C$):
\begin{equation*}
\begin{split}
& \frac{1}{2\pi i} \int\limits_{|s+\gamma| = 1/\log X_1} \frac{X_1^{s+\gamma}}{(s+\gamma)^2}\frac{P_z (1+s) G_{N} (s+1)
}{\left( D(s+1) \right)^{1/2}} ds
\ll  G_{N} (\log X_1) \left( |\gamma| + \frac{1}{\log X_1}\right)\\
&\ll G_{N}  \left( |\gamma| \log X_1+ 1\right)
\ll X_1^{|\gamma|} G_N.
\end{split}
\end{equation*}
Therefore, the lemma is completely proved.

\end{pf}

\begin{lemma}\label{l12}
Let $X_1$ be sufficiently large, $N$ be a positive integer and $z\ne 0$ be a complex number.
For $l=1,2$, define
$$
S_{z} ^{(l,l)}(X_1, \gamma, N) = \sum_{\substack{ \nu\le X_1
\\
(\nu, N)=1}} \frac{\alpha(\nu) K_{l,l}(\nu,
z)}{\nu^{1-\gamma}} \log \frac{X_1}{\nu}.
$$
Suppose
$\gamma=0$ or $\gamma=z (\ne 0)$. The following estimates hold true:
\begin{equation}\label{eq1000}
S_{z}^{(l,l)} (X_1, \gamma, N) \ll
\begin{cases}
& X_1^{\Real\gamma} G_N \log (|z|+3) (\log X_1)^{1/2}, \quad \Real z \ge 0, \\
& X_1^{\Real\gamma} G_N (|z| \log X_1 +1)^{1/2}, \quad 0\le \Real z \le \frac{\log\log X_1}{\log X_1}, |\Image z| \le 1.
\end{cases}
\end{equation}
If, moreover,
\begin{equation}\label{eq29}
\begin{split}
4\log^{-1} X_1 \le |z| \le (C \log\log X_1)^{-1},  \quad \Real z = 0
\end{split}
\end{equation}
(where $C$ is some large absolute constant)
then
\begin{equation}\label{eq38}
\begin{split}
S_{z}^{(l,l)} (X_1, \gamma, N)  &=  C(D, d_l) \Gamma^{-1}(3/2) R_{z, N}^{(l,l)} (1) H_{z, 1}^{(l,l)} (1) ((z-2\gamma) \log X_1)^{1/2}  \\
&+ O\left( G_N |z \log X_1|^{1/2}
|z| \log\log N
+  G_N
\right).
\end{split}
\end{equation}
\end{lemma}

\begin{pf}
The proof of this lemma starts similar to that of lemma \ref{l5}. For $\Real s
>1$, we define the generating function with multiplicative coefficients
\begin{equation*}
\begin{split}
h_{z, N}^{(l,l)} (s) &= \sum\limits_{(n, N)=1} \frac{\alpha(n) K_{l,l}(n,
z)}{n^{s}} = \prod_{(p, N)=1} \left( 1+ \frac{\alpha(p) K_{l,l}(p,
z)}{p^s} +\frac{\alpha(p^2) K_{l,l}(p^2,
z)}{p^{2s}} +\ldots \right)
\end{split}
\end{equation*}
An application of  Perron's summation formula yields
\begin{equation}\label{eq27}
\begin{split}
S_{z}^{(l,l)}(X_1,\gamma,N) &= \frac{1}{2\pi i}
\int\limits_{1-i\infty}^{1+i\infty} h_{z, N}^{(l,l)} (s+1-\gamma)
\frac{X_1^{s}}{s^2} ds.
\end{split}
\end{equation}

Let us investigate the behaviour of $h_{z, N}^{(l,l)} (s)$.
If a prime $p$ obeys the condition $(p, D) = 1$, then it follows from (\ref{eq18}) that
\begin{equation*}
\begin{split}
\alpha(p) K_{l,l} (p,z) &=  -\frac12 \left( \chi_{d_1} (p) + \chi_{d_2} (p)\right) \left( \chi_{d_l} (p) + \frac{\chi_{D/d_l} (p)}{p^{z}}\right) \left( 1 + O \left(\frac{1}{p}\right)\right) \\
&= -\frac12 \left( \chi^2_{d_l} (p) + \frac{\chi^2_{D/d_l} (p)}{p^{z}} + \chi_{D} (p)\left( 1+  \frac{1}{p^{z}}\right) \right) \left( 1 + O \left(\frac{1}{p}\right)\right) .
\end{split}
\end{equation*}

We shall approximate the coefficient at $p^{-s}$ of the generating function $h^{(l,l)}_{(z, N)}(s)$ by
$$A_l (p,z)=-\frac12  \left( \chi^2_{d_l} (p) + \chi^2_{D/d_l} (p) p^{-z}  + \chi_D (p) (1+p^{-z})\right),$$
which, in turn, is also the coefficient at $p^{-s}$ in the Euler product representation for  $L^{-1/2}_{\chi^2_{d_l} }(s) L^{-1/2}_{\chi^2_{D/d_l}}(s+z) L_{\chi_{D}} (s)^{-1/2} L_{\chi_{D}} (s+z)^{-1/2}$. Therefore,
we come to the equality
\begin{equation*}
\begin{split}
h_{z, N}^{(l,l)} (s) &= L^{-1/2}_{\chi^2_{d_l} }(s) L^{-1/2}_{\chi^2_{D/d_l}}(s+z) L_{\chi_{D}} (s)^{-1/2} L_{\chi_{D}} (s+z)^{-1/2} H_{z, N}^{(l,l)} (s),
\end{split}
\end{equation*}
where the function $H_{z, N}^{(l,l)} (s)$  for $\Real z \ge 0$ is analytic in the region $\Real s >1/2$ and can be written as the product
\begin{equation}\label{eq101}
H_{z, N}^{(l,l)} (s) = \mathcal{H}_{z, N}^{(l,l)} (s) G_{z,N}^{(l,l)} (s),
\end{equation}
where
$$
G_{z,N}^{(l,l)} (s) = \prod_{p\mid N} \left( 1-\frac{\chi^2_{d_l} (p)}{p^{s}}\right)^{-1/2} \left( 1-\frac{\chi^2_{D/d_l} (p)}{p^{s+z}}\right)^{-1/2} \left( 1-\frac{\chi_{D} (p)}{p^{s}}\right)^{-1/2} \left( 1-\frac{\chi_{D} (p)}{p^{s+z}}\right)^{-1/2}
$$
and, setting $\sigma =  \Real s$,
\begin{equation}\label{eq1003}
\begin{split}
&\mathcal{H}_{z, N}^{(l,l)} (s)  =  \prod_{
(p, N)=1} \left( 1-\frac{\chi^2_{d_l} (p)}{p^{s}}\right)^{-1/2} \left( 1-\frac{\chi^2_{D/d_l} (p)}{p^{s+z}}\right)^{-1/2} \left( 1-\frac{\chi_{D} (p)}{p^{s}}\right)^{-1/2} \left( 1-\frac{\chi_{D} (p)}{p^{s+z}}\right)^{-1/2}  \times\\
&\times
 \left( 1 + \frac{\alpha(p) K_{l,l}(p, z)}{p^{s}} + \frac{\alpha(p^2) K_{l,l}(p^2, z)}{p^{2s}} +\ldots  \right) \\
&= \prod\limits_{(p, N D)=1} \left( 1 -
\frac{A_l (p,z)}{p^{s}} + O\left( \frac{1}{p^{2\sigma}}\right)
\right)
\left( 1 +
\frac{A_l (p,z)}{p^{s}} + O\left( \frac{1}{p^{\sigma+1}}\right) + O\left( \frac{1}{p^{2\sigma}}\right)\right)\times\\
&\times \prod\limits_{(p, N)=1, p\mid D} \left( 1 +
 O\left( \frac{1}{p^{\sigma}}\right)
\right) \\
&= \prod\limits_{(p, N)=1} \left( 1 + O\left( \frac{1}{p^{\sigma+1}}\right) + O\left( \frac{1}{p^{2\sigma}}\right)\right)
\prod\limits_{(p, N)=1, \; p\mid D} \left( 1 + O\left( \frac{1}{p^{\sigma}}\right)
\right).
\end{split}
\end{equation}

To prove the first estimate of the lemma we shall move the path of integration in (\ref{eq27}) from the line $\Real s =1$ to the line $\Real s = \Real \gamma + 1/\log X_1$. On the new contour we shall use the following estimates
for $\Real w \ge 0$:
$$
\mathcal{H}_{z, N}^{(l,l)} (1+w) \ll 1,
$$
$$
G_{z,N}^{(l,l)} (1+w) \ll \prod_{p|N} \left( 1+\frac{1}{p}\right)^2 \le G_N,
$$
$$
L_{|\chi^{2}|}^{-1}(1+w)\ll \min (|w|, \log (|w|+3)),
$$
$$
L_{\chi_D}^{-1} (1+w) \ll \log (|w|+3).
$$
Therefore, we obtain the following inequality:
\begin{equation*}
\begin{split}
S_{z}^{(l,l)}(X_1,\gamma,N) & \ll G_N
\int\limits_{-\infty}^{+\infty}
\frac{X_1^{\Real \gamma}}{\log^{-2} X_1 + t^2} \left(\min(|t-\Image\gamma|+\log^{-1} X_1, \log (|t-\Image\gamma|+3) \right)^{1/2} \times \\
&\min \left( |t+\Image z-\Image\gamma| + \log^{-1} X_1, \log (|t+\Image z-\Image\gamma|+3) \right)^{1/2} \times\\
&\times
 \log^{1/2} (|t-\Image\gamma|+3) \log^{1/2} (|t+\Image z-\Image\gamma|+3)dt.
\end{split}
\end{equation*}
Since $\gamma$ equals either $0$ or $z$, we note that the pair of numbers $(t-\Image\gamma, t+\Image z-\Image\gamma)$ can get two values, namely, $(t, t+\Image z)$ or $(t-\Image z, t)$. Consequently, for $|t|\le \log^{-1} X_1$
$$
\min (|t|, |t\pm \Image z|) \le \log^{-1} X_1.
$$
Hence, we get
\begin{equation*}
\begin{split}
S_{z}^{(l,l)}(X_1,\gamma,N) &\ll
G_N X_1^{\Real \gamma} \int\limits_{|t| \le \log^{-1} X_1}  \log^2 X_1 \cdot \log^{-1/2} X_1 \cdot \log(|\Image z|+3) dt \\
&+ G_N X_1^{\Real \gamma} \int\limits_{\log^{-1} X_1 \le |t| \le 1 }  t^{-2} t^{1/2}  \log(|\Image z|+3) dt \\
&+ G_N X_1^{\Real \gamma} \int\limits_{|t| \ge 1 }  t^{-2}   \log(|t|+3)  \log(|t \pm\Image z|+3) dt\\
&\ll G_N X_1^{\Real \gamma} (\log X_1)^{1/2} \log (|\Image z|+3).
\end{split}
\end{equation*}

To obtain the second estimate of the lemma in case
\begin{equation}\label{eq17}
0 \le \Real z \le \frac{\log\log X_1}{\log X_1}, |\Image z| \le 1,
\end{equation}
we change the path of integration in (\ref{eq27}) in a more subtle way. Suppose that the constant $C$ is large enough that  there are no zeros of $L_{\chi^2_{d_j}} (s), L_{\chi_{D}} (s)$ in the region
$$
|\Image s| \le 8 \log^4 X_1, \quad \Real s \ge 1-\frac{8}{C \log\log X_1}
$$
for every $X_1 \ge 10$. If $|\Image z| \le \log^{-1} X_1$, set the new contour of integration as the union of the following segments:
\begin{equation*}
\begin{split}
\tilde{K}_1&=\{ s: \Real s =\Real\gamma, |\Image s| \ge 4\log^4 X_1\}, \\
\tilde{K}_2&=\{  s:  -2(C\log\log X_1)^{-1} + \Real \gamma \le\Real s\le \Real\gamma, |\Image s|=4 \log^4 X_1\}, \\
\tilde{K}_3&= \{s:  \Real s= -2(C\log\log X_1)^{-1} + \Real \gamma,  2 \log^{-1} X_1 < |\Image s| \le 4 \log^4 X_1  \}, \\
\tilde{K}_4&=\{ s:  -2(C\log\log X_1)^{-1} + \Real \gamma\le\Real s\le \Real \gamma, |\Image s|= 2 \log^{-1} X_1\}, \\
\tilde{K}_5&=\{ s:  |s-\Real \gamma| = 2 \log^{-1} X_1   e^{i\phi},  \quad |\phi| \le \pi\}.
\end{split}
\end{equation*}
We see that $h_{z, N}^{(l,l)} (s+1-\gamma)$ is analytic to the right of the new contour for $z$ satisfying (\ref{eq17}) if $X_1$ is large enough (say, when $4\Real z \le  (C\log\log X_1)^{-1}$).

The integral over $\tilde{K}_1$  from the integrand of (\ref{eq27})  can be estimated from above as
\begin{equation*}
\begin{split}
X_1^{\Real\gamma} G_{N} \int\limits_{4\log^4 X_1}^{+\infty} \frac{\log^2 t}{t^2} dt
\ll X_1^{\Real\gamma} G_{N} \frac{\log\log X_1}{\log^4 X_1} \ll G_N,
\end{split}
\end{equation*}
since
$$
\frac{X_1^{\Real \gamma}}{\log X_1} \le \frac{X_1^{\frac{\log\log X_1}{\log X_1}}}{\log X_1} \ll 1.
$$
On  $\tilde{K}_2$  the corresponding integral can be estimated from above as
\begin{equation*}
\begin{split}
X_1^{\Real\gamma} G_{N} \int\limits_{-2(C\log\log X_1)^{-1}}^{0} X_1^{\sigma} \frac{(\log\log X_1)^2}{\log^{8} X_1}  d\sigma
\ll X_1^{\Real\gamma} G_{N} \frac{(\log\log X_1)^2}{\log^{9} X_1} \ll  G_N.
\end{split}
\end{equation*}
On $\tilde{K}_3$  the corresponding integral can be estimated from above as
\begin{equation*}
\begin{split}
&X_1^{\Real\gamma-2(C\log\log X_1)^{-1}} G_{N} \int\limits_{|t| \le 5\log^4 X_1 }  \frac{(\log\log X_1)^2}{ (C \log\log X_1)^{-2} + t^2} dt\\
&\ll X_1^{\Real\gamma-2(C\log\log X_1)^{-1}} G_{N} (\log\log X_1)^3 \ll G_N.
\end{split}
\end{equation*}
On $\tilde{K}_4$ we can estimate the integral from above as
\begin{equation*}
\begin{split}
&G_{N}  \int\limits_{-2(C\log\log X_1)^{-1} +\Real \gamma}^{\Real \gamma} \frac{X_1^{\sigma}}{|\sigma +2 i  \log^{-1} X_1|^{3/2}}  (|\sigma  + 2 i  \log^{-1} X_1| + |z|)^{1/2} d\sigma \\
&\ll X_1^{\Real\gamma}  G_{N}   \int\limits_{-2(C\log\log X_1)^{-1} }^{0} \frac{X_1^{\sigma}}{|\sigma +\Real\gamma+ 2 i  \log^{-1} X_1|^{3/2}}  (|\sigma +\Real\gamma+ 2 i  \log^{-1} X_1|^{1/2}  + |z|^{1/2}) d\sigma  \\
&\ll  X_1^{\Real\gamma}
G_{N}   \int\limits_{0}^{+\infty} X_1^{-\sigma} \left(\frac{1 }{\log^{-1} X_1}  + \frac{|z|^{1/2}}{\log^{-3/2} X_1} \right)d\sigma \ll X_1^{\Real\gamma}
G_{N}   \int\limits_{0}^{+\infty} e^{-\sigma} \left( 1  + |z|^{1/2}\log^{1/2} X_1 \right) d\sigma \\
&\ll  X_1^{\Real\gamma} G_{N}  \left( 1 + |z|\log X_1\right)^{1/2}.
\end{split}
\end{equation*}
On $\tilde{K}_5$ we can estimate the integral from above as
\begin{equation*}
\begin{split}
&X_1^{\Real\gamma} G_{N} \frac{(| \Real\gamma -\Real z |+ \log^{-1} X_1)^{1/2}}{(\Real\gamma + \log^{-1} X_1)^{3/2} \log X_1}   \ll  X_1^{\Real \gamma} G_{N} \frac{(|z| +\log^{-1} X_1)^{1/2}}{\log^{-1/2} X_1} \\
&\ll X_1^{\Real\gamma} G_{N}  \left( 1 + |z|\log X_1\right)^{1/2}.
\end{split}
\end{equation*}

If $ \log^{-1} X_1 \le |\Image z| \le 1$
set the new contour of integration as the union of the following segments:
\begin{equation*}
\begin{split}
K_1&=\{ s: \Real s =\Real\gamma, |\Image s - \Image \gamma| \ge 4\log^4 X_1\}, \\
K_2&=\{  s:  -2(C\log\log X_1)^{-1} + \Real \gamma \le\Real s\le \Real\gamma, |\Image s - \Image \gamma|=4 \log^4 X_1\}, \\
K_3&= \{s:  \Real s= -2(C\log\log X_1)^{-1} + \Real \gamma, \{ 0< |\Image s - \Image \gamma| \le 4 \log^4 X_1\} \\
&\cap \{ |\Image s -\Image \gamma + \Image z| >0 \} \}, \\
K_4&=\{ s:  -2(C\log\log X_1)^{-1} + \Real \gamma\le\Real s\le - r_1, |\Image s|=+0\}, \\
K'_4&=\{ s:  -2(C\log\log X_1)^{-1} +\Real \gamma \le\Real s< 2\Real \gamma-\Real z , |\Image s  - 2\Image \gamma +\Image z|= +0 \}, \\
K_5&=\{ s:  |s| = \frac12 (\log X_1)^{-1}  e^{i\phi},  \quad |\phi| <\pi\}, \\
K'_5&=\{ s:  |s+z-2\gamma| = r e^{i\phi},  \quad |\phi| <\pi, \quad r=+0\}.
\end{split}
\end{equation*}
Again,  $h_{z, N}^{(l,l)} (s+1-\gamma)$ is analytic to the right of the new contour for $z$ satisfying (\ref{eq17}).
Similarly to the above estimates, the integral
over $K_1$ admits the following estimate from above
\begin{equation*}
\begin{split}
X_1^{\Real\gamma} G_{N} \int\limits_{3\log^4 X_1}^{+\infty} \frac{\log^2 t}{t^2} dt
\ll X_1^{\Real\gamma} G_{N} \frac{\log\log X_1}{\log^4 X_1} \ll G_N
\end{split}
\end{equation*}
as well as over $K_2$ and $K_3$.
On $K'_5$ we can estimate the integral from above as
\begin{equation*}
\begin{split}
&X_1^{\Real \gamma} G_{N} \lim\limits_{r\to +0} \frac{||\gamma-z|+r|^{1/2} |\gamma+r|^{1/2}}{|z|^2}   r \ll  X_1^{\Real \gamma} G_{N} \lim\limits_{r\to +0} \frac{r^{3/2}}{|z|^{3/2}}   =0
\end{split}
\end{equation*}
On $K'_4$ we can estimate the integral from above as
\begin{equation*}
\begin{split}
&G_{N}  \int\limits_{-2(C\log\log X_1)^{-1} +\Real \gamma}^{2\Real \gamma -\Real z} \frac{X_1^{\sigma}}{|\sigma + i\Image (2\gamma-z)|^{3/2}}   |\sigma +\Real z-2\Real\gamma |^{1/2} d\sigma \\
&\ll G_{N}  X_1^{2\Real \gamma -\Real z} \int\limits_{-2(C\log\log X_1)^{-1} +\Real z -\Real \gamma}^{0} \frac{X_1^{\sigma}}{|\sigma + 2\gamma-z|^{3/2}}  |\sigma |^{1/2} d\sigma  \\
&\ll
G_{N}  X_1^{\Real \gamma} \int\limits_{0}^{+\infty} \frac{X_1^{-\sigma}  \sigma^{1/2}}{|-\sigma + z|^{3/2}}   d\sigma \ll  G_{N}  X_1^{\Real \gamma} \int\limits_{0}^{+\infty} \frac{X_1^{-\sigma}  \sigma^{1/2}}{|\Image z|^{3/2}}   d\sigma   \\
&\ll G_{N}  X_1^{\Real \gamma} \int\limits_{0}^{+\infty} \frac{e^{-\sigma} \sigma^{1/2}}{(|\Image z| \log X_1)^{3/2}}   d\sigma \ll G_{N} X_1^{\Real \gamma} \left( |\Image z|\log X_1\right)^{-3/2} \ll G_{N} X_1^{\Real \gamma}.
\end{split}
\end{equation*}

 Let us get, first, the upper bound for the integrals over $K_4 \cup K_5$.

On $K_5$ we can estimate the integral from above as
\begin{equation*}
\begin{split}
& G_{N} (\log^{3/2} X_1)  |z|^{1/2} \log^{-1} X_1 \ll
G_{N}  |z \log X_1|^{1/2} .
\end{split}
\end{equation*}

On $K_4$ we can estimate the integral from above as
\begin{equation*}
\begin{split}
&G_{N}  \int\limits_{-2(C\log\log X_1)^{-1}}^{-(C\log X_1)^{-1}} \frac{X_1^{\sigma}}{|\sigma|^{2}}  |\sigma|^{1/2} |\sigma + z|^{1/2} d\sigma \\
&\ll G_{N}  \int\limits_{-2(C\log\log X_1)^{-1}}^{-(C\log X_1)^{-1}} \frac{X_1^{\sigma}}{|\sigma|}  d\sigma + G_{N}  \int\limits_{-2(C\log\log X_1)^{-1}}^{-(C\log X_1)^{-1}} \frac{X_1^{\sigma} |z|^{1/2}}{|\sigma|^{3/2}}  d\sigma\\
&\ll G_{N}  \int\limits_{C^{-1}}^{+\infty}  e^{-t} t^{-1} dt + G_{N} |z\log X_1|^{1/2} \int\limits_{C^{-1}}^{+\infty}  e^{-t} t^{-3/2} dt \ll G_{N} \left( |z\log X_1|^{1/2} +1\right).
\end{split}
\end{equation*}


Now, if $z$ satisfies (\ref{eq29}), then we shall obtain
exact expression
for the integral over the path $K_4 \cup K_5$ to prove (\ref{eq38}). Without loss of generality we may suppose that
\begin{equation}\label{eq30}
\begin{split}
|z| \log\log  N \le 1,
\end{split}
\end{equation}
since otherwise the statement  (\ref{eq38}) follows from  (\ref{eq1000}).
Let us
estimate trivially the integral over
the two segments $\left( -2(C\log\log X_1)^{-1}  \pm i\cdot 0, - X_0 \pm i\cdot 0 \right)$, where $X_0$ is a parameter to be chosen later as $|z|/2$, but generally
\begin{equation}\label{eq40}
\log^{-1} X_1  \le X_0 \le |z|/2.
\end{equation}
This part can be estimated from above  as
\begin{equation*}
\begin{split}
&G_{N}  \int\limits_{-2(C\log\log X_1)^{-1}}^{-X_0} \frac{X_1^{\sigma}}{|\sigma|^{2}}  |\sigma|^{1/2} |\sigma + z|^{1/2} d\sigma
\ll G_{N}  \int\limits_{-2(C\log\log X_1)^{-1}}^{-X_0} \frac{X_1^{\sigma} (|z|^{1/2}+ |\sigma|^{1/2})}{|\sigma|^{3/2}}  d\sigma\\
&\ll  G_{N} |z\log X_1|^{1/2} \int\limits_{X_0 \log X_1}^{+\infty}  e^{-t} t^{-3/2} dt +
G_{N}  \int\limits_{X_0 \log X_1}^{+\infty}  e^{-t} t^{-1} dt \\
&\ll G_{N} |z\log X_1|^{1/2} \left( X_0 \log X_1 \right)^{-3/2}\int\limits_{X_0 \log X_1 }^{+\infty}  e^{-t} dt +
G_{N} \left( X_0 \log X_1 \right)^{-1} \int\limits_{X_0 \log X_1}^{+\infty}  e^{-t} dt \\
&\ll G_{N} e^{- X_0 \log X_1 } \left( X_0 \log X_1 \right)^{-1} \left( |z /X_0|^{1/2}  + 1\right)
\ll G_{N} e^{- X_0 \log X_1 } \left( X_0 \log X_1 \right)^{-1}  |z /X_0|^{1/2}.
\end{split}
\end{equation*}

On the remaining path (which is the loop $\mathcal{L}$, that starts at the point $-X_0 -i\cdot 0$, then passes around zero in the positive direction and finishes at the point $-X_0 +i\cdot 0$) we shall approximate the integrand by its value at $s = 0$.
Note that $\left( (s-\gamma) L_{\chi^2_{d_l}} (1+s-\gamma)\right)^{-1/2}$,  $\left( (s+z-\gamma) L_{\chi^2_{D/d_l}} (1+s+z-\gamma)\right)^{-1/2}$ are analytic functions for
$|s| \le c$ and
$|s| +|z|\le c$ respectively (where $c>0$ is some absolute constant).  Moreover, for $|z| > 2 |s|$
\begin{equation*}
\begin{split}
& (s\pm z)^{1/2} = (\pm z)^{1/2} \cdot \left( 1 \pm \frac{s}{z}\right)^{1/2} = (\pm z)^{1/2} \cdot \left( 1 + O\left(\frac{|s|}{|z|}\right) \right).\\
\end{split}
\end{equation*}
Hence, on the loop $\mathcal{L}$ and for $z$ satisfying (\ref{eq29}) we have the relation
\begin{equation}\label{eq29_1}
\begin{split}
h_{z, N}^{(l,l)} (1+s-\gamma) =
C(D, d_l) s^{1/2} (z-2\gamma)^{1/2}  \left( 1 + O\left(\frac{|s|}{|z|} \right) + O(|z|) \right) H_{z, N}^{(l,l)} (1+s-\gamma)
\end{split}
\end{equation}
for some constant $C(D, d_l)$ (note that this constant coincide for $l=1,2$).
We shall further approximate $H_{z, N}^{(l,l)} (1+s-\gamma)$ by its value at $1$. Notice that
\begin{equation*}
\begin{split}
H_{z, N}^{(l,l)} (1+s-\gamma) = \mathcal{H}_{z, 1}^{(l,l)} (1+s-\gamma) R_{z, N}^{(l,l)} (1+s-\gamma),
\end{split}
\end{equation*}
where
\begin{equation}\label{eq31}
\begin{split}
R_{z, N}^{(l,l)} (1+w) &= \prod_{
p \mid N}
 \left( 1 + f_{p}(l, z, w) \right)^{-1} \\
f_{p}(l, z, w) &= \frac{\alpha(p) K_{l,l}(p, z)}{p^{1+w}} + \frac{\alpha(p^2) K_{l,l}(p^2, z)}{p^{2(1+w)}} +\ldots
\end{split}
\end{equation}

Recall that $\mathcal{H}_{z, 1}^{(l,l)} (1+s-\gamma)$ is holomorphic in the region $0 \le \Real z \le 1/200$, $\Real s > -1/2+ 1/200$. Thus, in the domain
\begin{equation*}
\begin{split}
&0 \le \Real z \le 1/200, \\
&|s| \ll 1, \Real s \ge -1/2 + 1/200
\end{split}
\end{equation*}
we can write
$$
\mathcal{H}_{z, 1}^{(l,l)} (1+s-\gamma) = \mathcal{H}_{z, 1}^{(l,l)} (1)  + O(|s-\gamma|),
$$
where the constant in $O$ is absolute (this can be explained by means of the Taylor series expansion for a holomorphic function and the Cauchy integral formula for the coefficients of the Taylor series).
We shall show next that for $z$ satisfying (\ref{eq29}) and
\begin{equation}\label{eq32}
\begin{split}
 -1/200 &\le s \le 1/200
\end{split}
\end{equation}
the function $R_{z, N}^{(l,l)} (1+s-\gamma)$ is analytic (of course, the analyticity will be in some open domains over $s$ and $z$, which contain the aforementioned intervals). For this, we exploit the statements of lemma \ref{l10} and \ref{l11} and show that
$1+f_p(l,z,w)$ (defined in (\ref{eq31})) is never close to zero.
We may suppose that the parameter $X_0 \le 1/200$ (since $X_1$ is large enough). Define
\begin{equation}\label{eq33}
\begin{split}
 w &= s-\gamma.
\end{split}
\end{equation}
Note, that
$$
1/200 \ge \Real w \ge -X_0 -\Real z = -X_0  \ge -0.01.
$$
For every $p \ge 2$ we exploit the statements of lemmas~\ref{l6}-\ref{l11} and (\ref{eq1002}), (\ref{eq1004})  to show that the following estimate holds:
\begin{equation*}
\begin{split}
&\left|f_{p}(l, z, w)
\right| \le \max \left( \frac23 \cdot 2^{0.02}, \max\limits_{p\ge 3} \left( \frac{2}{p+1}, \frac{3+1/p}{p^2}\right) p^{-2\Real w}  \right) \\
&+ \sum\limits_{k\ge 2}
\frac{1}{4k} \cdot (6k+1)  \frac{1}{p^{2k(1+\Real  w)}}\\
&\le
\frac23 \cdot 2^{0.02} + \frac{13}{8}\cdot\sum\limits_{k\ge 2}
\frac{1}{2^{2k(1+\Real w)}} =  \frac23 \cdot 2^{0.02} + \frac{13}{8}\cdot 2^{-4(1+\Real w)} \left( 1-2^{-2(1+\Real w)}\right)^{-1} \\
&\le  \frac23 \cdot 2^{0.02} + \frac{13}{8}\cdot \frac{1}{2^4}  \left( 1-2^{-1.98}\right)^{-1} 2^{0.04} \le 0.82.\\
\end{split}
\end{equation*}
Therefore,
\begin{equation*}
\begin{split}
\left| 1 + f_{p}(l, z, w) \right|^{-1} \le (0.18)^{-1}
\end{split}
\end{equation*}
which proves analyticity of $R_{z, N}^{(l,l)} (1+s-\gamma)$ in some domain over $s, z$ containing the
intervals (\ref{eq32}), (\ref{eq29}). Let us also estimate $\frac{d}{dz} R_{z, N}^{(l,l)} (1)$ via lemma~6, which we shall use later:
\begin{equation*}
\begin{split}
\frac{d}{dz} R_{z, N}^{(l,l)} (1) &= -R_{z, N}^{(l,l)} (1) \sum_{
p \mid N}
 \left( 1 + \frac{\alpha(p) K_{l,l}(p, z)}{p} + \frac{\alpha(p^2) K_{l,l}(p^2, z)}{p^{2}} +\ldots  \right)^{-1} \times \\
&\times
\left(  \frac{\alpha(p)}{p}  \frac{d}{dz} K_{l,l}(p, z)+ \frac{\alpha(p^2) }{p^{2}} \frac{d}{dz} K_{l,l}(p^2, z) +\ldots  \right) \\
&\ll |R_{z, N}^{(l,l)} (1)|  \sum_{
p \mid N} \left( \frac{\tau^2(p) \log p}{p} + \frac{\tau^2(p^2) \log p^2}{p^2} +\ldots \right) \ll |R_{z, N}^{(l,l)} (1)| \sum_{
p \mid N} \frac{\log p}{p} \\
&\ll |R_{z, N}^{(l,l)} (1)| \log\log N.
\end{split}
\end{equation*}

Now, let us approximate the value of $R_{z, N}^{(l,l)} (1+w)$ for $w=s-\gamma$ ($\gamma = 0$ or $\gamma = z$) in the domain (\ref{eq32}) by the value of $R_{z, N}^{(l,l)} (1)$ with some error.
It is clear that
\begin{equation}\label{eq34}
\begin{split}
&(1+f_{p}(l, z, w))^{-1} = (1 + f_{p}(l, z, 0))^{-1} \left( 1+\frac{r_{p}(l, z, w)}{1 + f_{p}(l, z, 0)}\right)^{-1},
\end{split}
\end{equation}
where
\begin{equation*}
\begin{split}
&r_{p}(l, z, w) = f_{p}(l, z, w) - f_{p}(l, z, 0) = \frac{\alpha(p) K_{l,l}(p, z)}{p} \left( p^{-w}-1\right) + \frac{\alpha(p^2) K_{l,l}(p^2, z)}{p^{2}} \left( p^{-2w}-1\right)  +\ldots   \\
&= \left( p^{-w}-1\right) \left( \frac{\alpha(p) K_{l,l}(p, z)}{p}  + \frac{\alpha(p^2) K_{l,l}(p^2, z)}{p^{2}} \left( p^{-w}+1\right)  +\ldots \right)
\end{split}
\end{equation*}
and $(1.82)^{-1} \le \left| 1 + f_{p}(l, z, w) \right|^{-1} \le (0.18)^{-1}$ for $\Real w \ge -0.01$ (therefore, the second multiplier in (\ref{eq34}) is holomorphic). Therefore,
\begin{equation*}
\begin{split}
&(1+f_{p}(l, z, w))^{-1} = (1 + f_{p}(l, z, 0))^{-1} \left( 1  + O\left(|r_{p}(l, z, w)|\right)\right)
\end{split}
\end{equation*}
and, hence,
\begin{equation}\label{eq35}
\begin{split}
&R_{z, N}^{(l,l)} (1+w) = R_{z, N}^{(l,l)} (1) \exp \left( O\left( \sum\limits_{p \mid N}  |r_{p}(l, z, w)|\right) \right).
\end{split}
\end{equation}
We see from lemma \ref{l10} that $\alpha(p^k) \le k^{-1}$ for $k \ge 1$ and from (\ref{eq23}),  that $K_{l,l}(p^k, z) \le 3k+1$. Therefore, $\alpha(p^k) K_{l,l}(p^k, z) \le 3 + k^{-1} \le 4$. Thus,
\begin{equation*}
\begin{split}
|r_{p}(l, z, w)| &\ll \left| p^{-w}-1\right| \left| \frac{4}{p}  + \frac{4}{p^{2}} 2 p^{-\Real w} + \frac{4}{p^{3}} 3 p^{-\Real w} +\ldots \right| \ll
\frac{1}{p} \left| p^{- w}-1\right|.
\end{split}
\end{equation*}
In case $|w| \log p\ll 1$,
we may use the bound
\begin{equation*}
\left| p^{-w}-1\right|  = \left| e^{-w \log p}-1\right| \ll |w| \log p.
\end{equation*}
Also, if $\frac{1}{|w|} \le \log p \le \frac{1}{|\Real w|}$, then
$$
\left| p^{- w}-1\right| \le  p^{- \Real w } + 1 \ll |w| \log p.
$$
Besides, we have the trivial estimate when $|w| \log p \ge 1$:
$$
\left| p^{- w}-1\right| \le 1 + p^{-\Real w} \le 2 \max (1, p^{-\Real w})  \ll \max( |w| \log p,  p^{-\Real w}).
$$
Consequently, recalling (\ref{eq32}), (\ref{eq33})
\begin{equation*}
\begin{split}
\sum\limits_{p \mid N}  |r_{p}(l, z, w)| &\ll \sum\limits_{p \mid N} \frac{1}{p} |w| \log p + \sum\limits_{\substack{ p \mid N, \\ p>e^{\frac{1}{X_0}} }} p^{-1-\Real w} \ll |w| \log\log N +
(e^{\frac{1}{X_0}})^{-1+X_0} X_0 \log N \\
&\ll |w| \log\log N +  e^{-\frac{1}{X_0}} X_0 \log N.
\end{split}
\end{equation*}
Further, from (\ref{eq35}) we get, recalling that $|s-\gamma| \ll |z|$ on the loop $\mathcal{L}$ and that $|z|\log\log N \le 1$:
\begin{equation*}
\begin{split}
R_{z, N}^{(l,l)} (1+s-\gamma) &= R_{z, N}^{(l,l)} (1) \left( 1 +  O\left( |s-\gamma| \log\log N + e^{-\frac{1}{X_0}} X_0 \log N\right) \right)
\end{split}
\end{equation*}
and, finally,
\begin{equation*}
\begin{split}
H_{z, N}^{(l,l)} (1+s-\gamma) &= (H_{z, 1}^{(l,l)} (1) + O(|s-\gamma|)) R_{z, N}^{(l,l)} (1) \left( 1 +  O\left( |s-\gamma| \log\log N + e^{-\frac{1}{X_0}} X_0 \log N\right) \right),
\end{split}
\end{equation*}
where
$$
R_{z, N}^{(l,l)} (1) \ll \prod_{
p \mid N}
 \left( 1 +  \frac{2}{p} \right) \ll G_N,
H_{z, 1}^{(l,l)} (1) = O(1),
$$
and the above estimates do not depend on $z$. Hence, (\ref{eq29_1}) and the last equality give
\begin{equation}\label{eq37}
\begin{split}
E_{z, N}^{(l,l)} (s-\gamma) &:= h_{z, N}^{(l,l)} (1+s-\gamma) - C(D, d_l) s^{1/2} (z-2\gamma)^{1/2} R_{z, N}^{(l,l)} (1) H_{z, 1}^{(l,l)} (1)  \\
&\ll
 G_N |z|^{1/2} |s|^{1/2} \left( \frac{|s|}{|z|} +  |z| \log\log N + e^{-\frac{1}{X_0}} X_0 \log N \right).
\end{split}
\end{equation}
The contribution of the error term $E_{z, N}^{(l,l)} (s-\gamma)$ to the initial integral over $\cal{L}$ does not change if we consider the integral over the loop $\tilde{\cal{L}}$, which starts and finishes at the same points as $\cal{L}$ but goes over the circle around zero with radius $1/\log X_1$. Therefore
\begin{equation*}
\begin{split}
&\int\limits_{\tilde{\cal{L}}} \frac{X_1^s}{s^2} E_{z, N}^{(l,l)} (s-\gamma) ds = \int\limits_{\tilde{\cal{L}}} e^{s \log X_1}\frac{1}{s^2} E_{z, N}^{(l,l)} (s-\gamma) ds \\
&\ll G_N |z|^{1/2} \left| \int\limits_{\tilde{\cal{L}}  } \left(|s|^{-3/2} \left( |z| \log\log N + e^{-\frac{1}{X_0}} X_0 \log N \right) + |s|^{-1/2} |z|^{-1}\right) ds \right| \\
&\ll  G_N |z|^{1/2} \left( \left( |z| \log\log N + e^{-\frac{1}{X_0}} X_0 \log N \right) \log^{1/2} X_1 +
|z|^{-1} X_0^{1/2}\right) \\
&\ll G_N |z \log X_1|^{1/2}  \left( |z| \log\log N + e^{-\frac{1}{X_0}} X_0 \log N+ |z|^{-1}\left( \frac{X_0}{ \log X_1}\right)^{1/2}\right).
\end{split}
\end{equation*}

We see from(\ref{eq37}) that to find the main term of the integral (\ref{eq27}), we need to evaluate the integral
$$
\frac{1}{2\pi i}\int\limits_{\cal{L}} \frac{X_1^s}{s^{3/2}} ds = \frac{1}{2\pi i}\int\limits_{\log X_1 \cdot  {\cal{L}}}  \frac{e^{s}}{s^{3/2}} \sqrt{\log X_1} ds,
$$
which is, due to the Hankel integral formula,  $\dfrac{\sqrt{\log X_1} }{\Gamma(3/2)}$ plus an error term, which can be estimated as $\sqrt{\log X_1}$ multiplied by the integral
\begin{equation}\label{eq508}
\begin{split}
\frac{1}{\pi}\int\limits_{-\infty}^{-X_0 \log X_1} \frac{e^s}{|s|^{3/2}} ds \ll
\left(X_0 \log X_1\right)^{-3/2} \int\limits_{-\infty}^{-X_0 \log X_1} e^{s} ds
\ll \left(X_0 \log X_1\right)^{-3/2} e^{-X_0 \log X_1}.
\end{split}
\end{equation}
Therefore, collecting all the results, we find for 
$z$, $N$  and $X_0$ satisfying (\ref{eq29}), (\ref{eq30}) and (\ref{eq40}) that the sum of the lemma equals
\begin{equation*}
\begin{split}
&C(D, d_l) \Gamma^{-1}(3/2) R_{z, N}^{(l,l)} (1) H_{z, 1}^{(l,l)} (1) ((z-2\gamma) \log X_1)^{1/2}  \\
&+ O\left(G_N |z \log X_1|^{1/2} \left( |z| \log\log N + e^{-\frac{1}{X_0}} X_0 \log N\right)\right) \\
&+ O\left( G_N \left(X_0 \log X_1\right)^{-1} e^{-X_0 \log X_1}  \left(\frac{|z|}{X_0}\right)^{1/2} + G_N X_1^{\Real z}\right).
\end{split}
\end{equation*}
Therefore,
we can put $X_0 = |z|/2$ in the last expression and obtain (\ref{eq38}), since, by  (\ref{eq29}) and (\ref{eq30}),
\begin{equation*}
\begin{split}
&\left( |z| \log X_1\right)^{-1} e^{-\frac12 |z| \log X_1} \ll 1, \\
&e^{-\frac{2}{|z|}} |z| \log N \ll |z| \log\log N.
\end{split}
\end{equation*}
The proof of the lemma is now complete.

\end{pf}

\begin{pfl3}

Using M\"obius inversion formula
\begin{equation*}
f(q) = \sum_{d|q} \sum_{m|d} \mu(m) f\biggl(\frac{d}{m}\biggr),
\end{equation*}
we have
\begin{align*}
S(z) &= \sum_{\nu_1,\dots, \nu_4 \le X} \frac{\beta(\nu_1)
\beta(\nu_2) \beta(\nu_3) \beta(\nu_4)}{ \nu_2 \nu_4
(\nu_1\nu_3)^{1-z}} \sum_{d|q} \sum_{m|d} \mu(m) \biggl(
\frac{d}{m} \biggr)^{1-z}
\\
&\quad\times K\biggl( \frac{\nu_1 \nu_4 m}{d}, 1-z\biggr)
K\biggl( \frac{\nu_2 \nu_3 m}{d}, 1-z\biggr) = \sum_{d \le X^2}
\sum_{m|d} \mu(m) \biggl( \frac{d}{m} \biggr)^{1-z} g_z^2 (d,m),
\end{align*}
where
$$
g_z(d,m)=\sum_{\substack{\nu_1\nu_4\equiv0\ (\operatorname{mod}d)
\\
\nu_j\le X}}\frac{\beta(\nu_1)\beta(\nu_4)}{\nu_1^{1-z}
\nu_4}K\biggl(\frac{\nu_1\nu_4m}d,1-z\biggr).
$$
Since $\alpha (\cdot)$ and $K(\cdot, s)$ are multiplicative funcions, as in \cite{Rezvyakova_2016} (see the proof of lemma 1) we finally obtain the following exact formula
\begin{align*}
g_z(d,m) &= \sum_{\substack{\delta_1 \delta_4
\equiv0 \ (\operatorname{mod}d)
\\
\delta_j | d^{\infty}, \delta_j \le X}} \frac{\alpha(\delta_1)
\alpha(\delta_4)}{\delta_1^{1-z} \delta_4} K\biggl(
\frac{\delta_1 \delta_4 m}{d}, 1-z\biggr)
\\
&\times \sum\limits_{(n,d)=1} \frac{1}{n^{2-z}} \sum\limits_{(r, nd)=1}  \frac{\mu(r)}{r^{2-z}}\sum_{\substack{k_1, k_4 | n^{\infty}, \\ (k_1, k_4)=1}} \frac{\alpha(n k_1) \alpha(n k_4) K(n^2 k_1 k_4, 1-z)}{k_1^{1-z} k_4}  \\
&\times \sum_{l_j | r^{\infty}} \frac{\alpha(r l_1) \alpha(r l_4)}{l_1^{1-z}
l_4} K(r l_1, 1-z) K( r l_4, 1-z) \\
&\times \sum_{\substack{\nu_j \delta_j k_j l_j n r \le X
\\
(\nu_j,ndr) =1}} \frac{\alpha(\nu_1) \alpha(\nu_4)}{\nu_1^{1-z}
\nu_4} K(\nu_1, 1-z) K( \nu_4, 1-z)  \EuScript{L} (\nu_1 l_1 \delta_1 k_1 n r)
\EuScript{L}( \nu_4 l_4 \delta_4  k_4 n r ) .
\end{align*}
Set
\begin{align}
\notag S_{z}(X, R,\gamma,N) &= \sum_{\substack{\nu\le X/R
\\
(\nu, N)=1}} \frac{\alpha(\nu) K(\nu,
1-z)}{\nu^{1-\gamma}} \EuScript{L}(R \nu).
\end{align}
From the definition of $\EuScript{L}(\nu)$ we obtain
$$
S_{z}(X, R,\gamma,N) = \sum_{\substack{1 \le \nu < \sqrt{X}/R
\\
(\nu, N)=1}} \frac{\alpha(\nu) K(\nu,
1-z)}{\nu^{1-\gamma}}  + \sum_{\substack{\sqrt{X}/R \le \nu\le X/R
\\
(\nu, N)=1}} \frac{\alpha(\nu) K(\nu,
1-z)}{\nu^{1-\gamma}} 2 \frac{\log \frac{X}{R\nu}}{\log X}.
$$
For $\nu < \sqrt{X}/R$, using the identity
$$
2 \frac{\log \frac{X}{R\nu}}{\log X} - 2 \frac{\log \frac{\sqrt{X}}{R\nu}}{\log X} =1,
$$
we find that
$$
S_{z}(X, R,\gamma,N) = \frac{2}{\log X} \left( S_{z}(X/R,\gamma,N)-S_{z}(\sqrt{X}/{R},\gamma,N) \right),
$$
where the function $S_{z} (X, \gamma, N)$ of three variables is defined in lemma \ref{l5}.

Thus, it follows that the last sum over $\nu_1, \nu_4$ in the formula for $g_z(d,m)$ equals to the following expression
\begin{align*}
\frac{4}{\log^2 X} &\left( S_{z}(X/(\delta_1 k_1 l_1 n r), z, ndr)-S_{z}(\sqrt{X}/{(\delta_1 k_1 l_1 n r)}, z, ndr) \right) \\
&\times \left( S_{z}(X/(\delta_4 k_4 l_4 n r),0,ndr)-S_{z}(\sqrt{X}/{( \delta_4 k_4 l_4 n r)},0,ndr) \right)
\end{align*}
and by the lemma \ref{l5} this can be estimated from above (up to some constant) by
$$
\frac{1}{\log^2 X} \left(\frac{X}{\delta_1 k_1 l_1 n r}\right)^{|z|} \prod\limits_{p| (ndr)} \left( 1+\frac{1}{\sqrt{p}}\right)^4.
$$
Now
\begin{align*}
|g_z(d,m)| &\ll \frac{ X^{|z|}}{\log^2 X} \prod\limits_{p| d} \left( 1+\frac{1}{\sqrt{p}}\right)^4
\sum_{\substack{\delta_1 \delta_4
\equiv0 \ (\operatorname{mod}d)
\\
\delta_j | d^{\infty}, \, \delta_j \le X}} \frac{|\alpha(\delta_1)
\alpha(\delta_4)|}{\delta_1 \delta_4} \left| K\biggl(
\frac{\delta_1 \delta_4 m}{d}, 1-z\biggr)\right| \\
&\le  \frac{ X^{|z|}}{\log^2 X}  \frac{1}{d}\prod\limits_{p| d} \left( 1+\frac{1}{\sqrt{p}}\right)^4
\sum\limits_{n\mid d^{\infty}} \frac{|K(nm, 1-z)|}{n}\sum_{\substack{\delta_1 \delta_4 = nd
\\
\delta_j | d^{\infty}, \, \delta_j \le X}} |\alpha (\delta_1)
\alpha (\delta_4)|.
\end{align*}

We need some extra bounds to proceed further with the estimation of $g_z(d,m)$.
Since $\alpha (n)$ is a multiplicative function, the function
$$
b(n) = \sum\limits_{n_1 n_2 =n} |\alpha
(n_1) \alpha (n_2)|
$$
is also multiplicative and satisfies the relations  (see \cite{Rezvyakova_2016},  the proof of lemma 1)
\begin{align*}
b(n) &\le \tau(n), \\
b(p) &= |r(p)| \text { for a prime } p.
\end{align*}
We define the multiplicative function $B(n)$ by the formulae
\begin{equation}\label{eq22}
B(p) = |r(p)|, \quad B(p^{\beta}) = \tau(p^{\beta}) \text{ for } \beta >1.
\end{equation}
Recalling that
$
|K(n, s)| \le \tau_{2000}(n)
$ for $\Real s >1/2$ and using the equality
$$
|K(nm, s)| \le \tau_{2000}(n) \tau_{2000}(m),
$$
we see that
\begin{align*}
|g_z (d, m)|
&\ll \frac{ X^{|z|}}{\log^2 X} \frac{\tau_{2000}(m)}{d} \prod_{p|d} \biggl( 1+\frac{1}{\sqrt{p}} \biggr)^4 \sum_{n| d^{\infty}} \frac{\tau_{2000} (n)}{n}  b(nd) \\
&\ll \frac{ X^{|z|}}{\log^2 X} \frac{\tau_{2000}(m)}{d} \prod_{p^{\beta}|| d} \biggl( 1+\frac{1}{\sqrt{p}} \biggr)^4 \biggl( b(p^{\beta}) +\frac{b(p^{\beta+1}) \tau_{2000}(p)}{p} + \frac{b(p^{\beta+2}) \tau_{2000}(p^2)}{p^2}  +\ldots\biggr)\\
&\ll \frac{ X^{|z|}}{\log^2 X} \frac{\tau_{2000}}{d} \prod_{p^{\beta}|| d} \biggl( 1+\frac{1}{\sqrt{p}} \biggr)^4 \biggl( B(p^{\beta}) +\frac{B(p^{\beta+1}) C}{p} \biggr)
\end{align*}
with some positive constant $C$.
Hence, we obtain
\begin{align*}
|S(z)| &\ll  \frac{ X^{2|z|}}{\log^4 X}  \sum_{d \le X^2} \frac{1}{d^{1+\Real z}} \prod_{p^{\beta}||d} \biggl( 1+\frac{1}{\sqrt{p}} \biggr)^{8} \left(  B(p^{\beta}) +\frac{B(p^{\beta+1}) C}{p}\right)^2
\sum_{m|d} \frac{\mu^2 (m)  \tau_{2000}^{2} (m)}{m^{1-\Real z}} \\
&\ll \frac{ X^{2|z|}}{\log^4 X}  \sum_{d \le X^2} \frac{1}{d^{1+\Real z}}  \prod_{p^{\beta}||d} \biggl( 1+\frac{1}{\sqrt{p}} \biggr)^{8} \biggl( 1+\frac{\tau_{2000}^{2} (p)}{p^{1-\Real z}} \biggr)
 \left(  B(p^{\beta}) +\frac{B(p^{\beta+1}) C}{p}\right)^2 \\
& \ll \frac{ X^{2|z|}}{\log^4 X}  \sum_{d \le X^2} \frac{1}{d^{1+\Real z}}  \prod_{p^{\beta}||d} \biggl( 1+\frac{1}{\sqrt{p}} \biggr)^{9} \left(  B(p^{\beta}) +\frac{B(p^{\beta+1}) C}{p}\right)^2.
\end{align*}
As in the end of the proof of lemma 1 in \cite{Rezvyakova_2016} we see that
\begin{align*}
|S(z)| &\ll \frac{ X^{2|z|}}{\log^4 X} \sum_{d \le X^2} \frac{r^2 (d)}{d^{1+\Real z }} \ll
\frac{ X^{2|z|}}{\log^2 X},
\end{align*}
since $\sum\limits_{d \le u} r^2 (d)/d  \ll \log^2 u$ thanks to the pole of order two at $s=1$ of $D(s)$.
Hence the lemma is proved.

\end{pfl3}

\begin{pfl4}

The proof of this lemma is similar to the proof above.
Using M\"obius inversion formula, we have
\begin{align}
S_{l,l}(z) &= \sum_{\nu_1,\dots, \nu_4 \le X} \frac{\beta(\nu_1)
\beta(\nu_2) \beta(\nu_3) \beta(\nu_4)}{ \nu_2 \nu_4
(\nu_1\nu_3)^{1-z}} \sum_{d|q} \sum_{m|d} \mu(m) \biggl(
\frac{d}{m} \biggr)^{1-z} \nonumber
\\
&\quad\times K_{l,l}\biggl( \frac{\nu_1 \nu_4 m}{d}, z\biggr)
K_{l,l}\biggl( \frac{\nu_2 \nu_3 m}{d}, z\biggr) = \sum_{d \le X^2}
\sum_{m|d} \mu(m) \biggl( \frac{d}{m} \biggr)^{1-z} g_{l,l}^2 (d,m,z), \label{eq505}
\end{align}
where
$$
g_{l,l}(d,m,z)=\sum_{\substack{\nu_1\nu_4\equiv0\ (\operatorname{mod}d)
\\
\nu_j\le X}}\frac{\beta(\nu_1)\beta(\nu_4)}{\nu_1^{1-z}
\nu_4}K_{l,l}\biggl(\frac{\nu_1\nu_4m}d,z\biggr)
$$
As in \cite{Rezvyakova_2016} we finally obtain the following exact formula
\begin{align}
g_{l,l}(d,m,z) &= \sum_{\substack{\delta_1 \delta_4
\equiv0 \ (\operatorname{mod}d) \nonumber
\\
\delta_j | d^{\infty}}} \frac{\alpha(\delta_1)
\alpha(\delta_4)}{\delta_1^{1-z} \delta_4} K_{l,l}\biggl(
\frac{\delta_1 \delta_4 m}{d}, z\biggr) \nonumber
\\
&\times \sum\limits_{(n,d)=1} \frac{1}{n^{2-z}} \sum\limits_{(r, nd)=1}  \frac{\mu(r)}{r^{2-z}}\sum_{\substack{k_1, k_4 | n^{\infty}, \\ (k_1, k_4)=1}} \frac{\alpha(n k_1) \alpha(n k_4) K_{l,l}(n^2 k_1 k_4, z)}{k_1^{1-z} k_4} \nonumber \\
&\times \sum_{l_j | r^{\infty}} \frac{\alpha(r l_1) \alpha(r l_4)}{l_1^{1-z}
l_4} K_{l,l}(r l_1, z) K_{l,l}( r l_4, z) \nonumber \\
&\times \sum_{\substack{\nu_j \delta_j k_j l_j n r \le X
\\
(\nu_j,ndr) =1}} \frac{\alpha(\nu_1) \alpha(\nu_4)}{\nu_1^{1-z}
\nu_4} K_{l,l}(\nu_1, z) K_{l,l}( \nu_4, z)  \EuScript{L} (\nu_1 l_1 \delta_1 k_1 n r)
\EuScript{L}( \nu_4 l_4 \delta_4  k_4 n r ). \label{eq506}
\end{align}
From the definition of $\EuScript{L}(\nu)$ it follows, that the last sum over $\nu_1, \nu_4$ in (\ref{eq506}) equals to the following expression
\begin{align}
\frac{4}{\log^2 X} &\left( S^{(l,l)}_{z}(X/(\delta_1 k_1 l_1 n r), z, ndr)-S^{(l,l)}_{z}(\sqrt{X}/{(\delta_1 k_1 l_1 n r)}, z, ndr) \right)\times \nonumber\\
&\times \left( S^{(l,l)}_{z}(X/(\delta_4 k_4 l_4 n r),0,ndr)-S^{(l,l)}_{z}(\sqrt{X}/{(\delta_4 k_4 l_4 n r)},0,ndr) \right), \label{eq507}
\end{align}
where the function $S^{(l,l)}_{z} (X_1, \gamma, N)$ of three variables is defined and estimated in lemma \ref{l12} (note, that $N = ndr \le X^3$). Thus, the last expression can be bounded from above  by
$$
\frac{|z|\log X  +1}{\log^2 X} \left(\frac{X}{\delta_1 k_1 l_1 n r}\right)^{\Real z}\prod\limits_{p| (ndr)} \left( 1+\frac{1}{\sqrt{p}}\right)^4
$$
for $|\Image z|\le 1$, and by
$$
\frac{ \log^2 (|z|+3)}{\log X} \left(\frac{X}{\delta_1 k_1 l_1 n r}\right)^{\Real z}\prod\limits_{p| (ndr)} \left( 1+\frac{1}{\sqrt{p}}\right)^4
$$
for all $z$ satisfying conditions of the lemma.

For $|\Image z|\le 1$ we, therefore, obtain
\begin{align*}
|g_{l,l}(d, m, z)| &\ll \frac{X^{\Real z} (|z|\log X  +1) }{\log^2 X}  \prod_{p|d} \biggl( 1+\frac{1}{\sqrt{p}} \biggr)^4\sum_{\substack{\delta_1 \delta_4 \equiv0 \ (\operatorname{mod}d)
\\
\delta_j | d^{\infty}}}\!\! \frac{|\alpha(\delta_1)
\alpha(\delta_4)|}{\delta_1 \delta_4} \left| K_{l,l}\biggl( \frac{\delta_1
\delta_4 m}{d}, z\biggr)\right|  \\
&\le \frac{X^{\Real z} (|z|\log X  +1)}{\log^2 X}   \frac{1}{d} \prod_{p|d} \biggl( 1+\frac{1}{\sqrt{p}} \biggr)^4 \sum_{n| d^{\infty}} \frac{\left|K_{l,l}(nm, z)\right|}{n}  \sum_{\substack{\delta_1 \delta_4 = nd
\\
\delta_j | d^{\infty}}}\!\! |\alpha(\delta_1)
\alpha(\delta_4)|.
\end{align*}

 As in the proof of lemma \ref{l3},  by (\ref{eq22}) and the inequality
$
|K_{l,l}(n, s)| \ll \tau^2 (n)
$ for $\Real s >1/2$,
we see  that
\begin{align*}
|g_{l,l}(d, m, z)| &\ll  \frac{X^{\Real z}  (|z|\log X+1) }{\log^2 X} \frac{\tau^2 (m)}{d} \prod_{p^{\beta}|| d} \biggl( 1+\frac{1}{\sqrt{p}} \biggr)^4 \biggl( B(p^{\beta}) +\frac{B(p^{\beta+1}) C}{p} \biggr)
\end{align*}
with some positive constant $C$.
Hence, we obtain
\begin{align*}
|S^{(l,l)}(z)| &\ll \frac{X^{2\Real z}  (|z|\log X+1)^2 }{\log^4 X} \sum_{d \le X^2} \frac{1}{d^{1+\Real z}}  \prod_{p^{\beta}||d} \biggl( 1+\frac{1}{\sqrt{p}} \biggr)^9 \left(  B(p^{\beta}) +\frac{B(p^{\beta+1}) C}{p}\right)^2 \\
&\ll
X^{2\Real z} \frac{(|z|\log X+1)^2 }{\log^4 X}
  \sum_{d \le X^2} \frac{r^2 (d)}{d^{1+\Real z}} \ll  X^{2\Real z}  \frac{(|z|\log X+1)^2 }{\log^2 X}.
\end{align*}

Similarly, if $0\le \Real z \le \dfrac{\log\log X}{\log X}$  we get
\begin{align*}
|S^{(l,l)}(z)| &\ll  X^{2\Real z} \log^4 (|z|+3).
\end{align*}

Now, let us write the more elaborate but precise formula for $g_{l,l}(d, m, z)$ (when $z$ satisfies (\ref{eq1006})), inserting the exact formula (\ref{eq38}) into (\ref{eq507}). Note, that in (\ref{eq507}) the parameter $X_1$ from (\ref{eq38}) takes the value between $1$ and $X$. If $|z| = \frac{A_1}{\log X}$, where $10 \le A_1 \le \left( \frac{\log X}{C \log\log X}\right)$
we can use the asymptotic formula (\ref{eq38}) in case $X_1$ satisfies $\log X_1\ge \frac{4 \log X}{A_1}$ (or, $|z| \ge \frac{4}{\log X_1}$) and the trivial asymptotic formula with the same main term but with an error term written as
$$
O\left( G_N \left(\sqrt{|z| \log X_1} +1\right) \right)
$$
for any $X_1$ (which is actually
$
O\left( G_N \right)
$
for $|z|\log X_1 \le 4$).

For $|z| \log X_1 \ge 4$, the error term in (\ref{eq38}) is estimated as
$$
O\left( G_N  A_1 ^{3/2} (\log X)^{-1} \log\log X
+ G_N\right) \ll G_N
$$
if $A_1 \le \left( \frac{\log  X}{\log\log X}\right)^{2/3}$ and $X$ is large. Thus, we shall use asymptotic formula (\ref{eq38}) (which is sometimes trivial) for the sums $S^{(l,l)}_{z}(\cdot, \cdot, \cdot)$ in (\ref{eq507}) with the error term $O(G_N)$. Then we get
\begin{align}
g_{l,l}(d,m,z) &= \text{Main term}+ \text{Error term},
\end{align}
where the ``Main term" is given by (\ref{eq509}) and
and the error term can be estimated from above as
\begin{align*}
\text{Error term} \ll \frac{X^{\Real z} |z|^{1/2}}{\log^{3/2} X} \frac{\tau^2 (m)}{d} \prod_{p^{\beta}|| d} \biggl( 1+\frac{1}{\sqrt{p}} \biggr)^4 \biggl( B(p^{\beta}) +\frac{B(p^{\beta+1}) C}{p} \biggr).
\end{align*}
Thus, we shall obtain (\ref{eq510}) following the above calculations on the estimate of $S^{(l,l)}(z)$.
Hence, lemma~\ref{l4} is
proved.

\end{pfl4}

\begin{center}
{\bf \S 6. Estimation of the diagonal term for the case of the real Hecke character.}
\end{center}

First, let us introduce a non-negative
function $\phi(u)$ defined by the formula
$$
\phi(u) = \frac{1+ \Delta^4}{\Delta^4 + u^{-4}}, \quad \Delta = \cos\delta.
$$
Note, that
\begin{itemize}
\item $\phi(u)$ is increasing and non-negative for $u\ge 0$,
\item $\phi(u) \ge 1$ for $u \ge 1$,
$\phi (u)$ is bounded for $u\ge 0$.
\end{itemize}
Therefore, we have
\begin{align*}
J(x, \theta) &= \int\limits_{x}^{+\infty} G(u) u^{-\theta} du \le J (1, \theta)= 
\int\limits_{1}^{+\infty} G(u)  u^{-\theta} du \le
\int\limits_{0}^{+\infty} \phi (u) G(u)  u^{-\theta} du.
\end{align*}
To estimate the above integral we rewrite $G(u)$ as in \cite{Rezvyakova} (in the proof of lemma 5), asserting $r(\,\cdot\,)$ to be zero for
non-integer argument:
\begin{align*}
& G (u) =\!\sum_{\nu_1,\nu_2, \nu_3, \nu_4}\!
\frac{\beta(\nu_1)\beta(\nu_2)\beta(\nu_3)\beta(\nu_4)}{\nu_2\nu_4}
\biggl( \sum_{n=1}^{+\infty} r(n) r\biggl( \frac{n m_1}{m_2}\biggr)
\exp\biggl( -\frac{4\pi}{\sqrt{D}}\frac{n m_1}{Q} u
\sin\delta\biggr)
\\
&\qquad\qquad\qquad  + 2 \operatorname{Re} \sum_{l\ge1} \exp\biggl(
-\frac{2\pi i}{\sqrt{D}} \frac{l}{Q} u\cos\delta\biggr) \times
\\
&\qquad\ \times\sum_{n=1}^{+\infty} r(n) r\biggl(
\frac{n m_1 +l }{m_2}\biggr) \exp\biggl( -\frac{4\pi}{\sqrt{D}}
\frac{(n m_1 + \frac l2)}{Q} u \sin\delta\biggr) \biggr),
\end{align*}
where
\begin{align*}
m_1 &= \frac{\nu_1\nu_4}{(\nu_1\nu_4, \nu_2\nu_3)}, \\
m_2 &= \frac{\nu_2\nu_3}{(\nu_1\nu_4, \nu_2\nu_3)}, \\
Q &=\frac{\nu_2\nu_4}{(\nu_1\nu_4, \nu_2\nu_3)}.
\end{align*}
Define
\begin{gather*}
E_{m_1, m_2} (y, l) = \sum_{n=1}^{+\infty} r(n) r \biggl( \frac{m_1
n +l}{m_2} \biggr) \exp\biggl(-y \biggl(m_1 n + \frac
l2\biggr)\biggr),
\\
E_{m_1, m_2} (y) = E_{m_1, m_2} (y,0),
\\
B = \frac{4\pi \sin\delta}{\sqrt{D}Q}, \qquad B_1 = \frac{
\cos\delta}{\sqrt{D}Q}.
\end{gather*}
Then we obtain
\begin{align*}
J(1, \theta) &\ll \sum_{\nu_1,\nu_2, \nu_3, \nu_4}
\frac{\beta(\nu_1)\beta(\nu_2)\beta(\nu_3)\beta(\nu_4)}{\nu_2\nu_4}
\biggl( \int_{0}^{+\infty} \phi(u) E_{m_1, m_2} (Bu) u^{-\theta}\,du
\\
&\qquad + 2
 \sum_{l\ge1} \int_{0}^{+\infty} \phi(u)
\cos(2\pi  l B_1 u) E_{m_1, m_2} (Bu, l) u^{-\theta}\,du \biggr).
\end{align*}
For $c>1$ (say, $c=2$), the Mellin transform formula yields
\begin{equation*}
E_{m_1, m_2} (y, l) = \frac{1}{2\pi i} \int_{c-i\infty}^{c+i\infty}
D_{m_1, m_2} (s,l) \Gamma(s) y^{-s}\,ds,
\end{equation*}
where
$$
D_{m_1,m_2}(s,l)=\sum_n\frac{r(n)r\bigl(\frac{m_1
n+l}{m_2}\bigr)}{(m_1n+\frac l2)^s}.
$$
Eventually we find
\begin{align*}
J(1, \theta) &\ll  J_1(1,\theta) + J_2 (1, \theta),
\end{align*}
where
\begin{align*}
J_1(1,\theta) &= \sum_{\nu_1,\nu_2, \nu_3, \nu_4}
\frac{\beta(\nu_1)\beta(\nu_2)\beta(\nu_3)\beta(\nu_4)}{\nu_2\nu_4}
\frac{1}{2\pi i} \int_{c-i\infty}^{c+i\infty}
D_{m_1, m_2} (s,0) \Gamma(s) B^{-s} \Phi(s+\theta, 0)\,ds,
\\
J_2 (1, \theta) &= \sum_{\nu_1,\nu_2, \nu_3, \nu_4}
\frac{\beta(\nu_1)\beta(\nu_2)\beta(\nu_3)\beta(\nu_4)}{\nu_2\nu_4}\operatorname{Re} \sum_{l\ge1} \frac{1}{2\pi i}
\int_{c-i\infty}^{c+i\infty} D_{m_1, m_2} (s,l) \Gamma(s) B^{-s}
\Phi(s+\theta, B_1 l)\,ds
\end{align*}
and
\begin{align*}
\Phi(s,y)=\int_0^{+\infty}\phi(u)u^{-s}\cos(2\pi yu) \,du \quad \text{ for }\quad 1 <\operatorname{Re}s < 5
\end{align*}
(since $\phi(+\infty) \ll 1$  and $\phi(u) \asymp u^4$ for $u\asymp 0$).

Let us now estimate  $J_1(x,\theta)$ (the ``diagonal'' term).
First,
we find a meromorphic continuation for $D_{m_1,m_2}(s,0)$ to the
left of the half-plane $\operatorname{Re}s>1$. Since $(m_1,m_2)=1$,
we have
\begin{align*}
D_{m_1,m_2}(s)&=D_{m_1,m_2}(s,0)=\frac1{(m_1
m_2)^s}\sum_{n=1}^{+\infty}\frac{r(nm_2)r(nm_1)}{n^s} \\
&=\frac{1}{(m_1 m_2)^s} \sum_{k \mid (m_1 m_2)^{\infty}}
\frac{r(m_1 k) r(m_2 k)}{k^{s}} \sum_{(n, m_1 m_2) =1}
\frac{r^{2}(n)}{n^{s}}
\\
& = \frac{1}{(m_1 m_2)^s} \sum_{k_1 \mid m_1^{\infty}} \frac{r(m_1 k_1)
r(k_1)}{k_1^{s}} \sum_{k_2 \mid m_2^{\infty}} \frac{r(m_2 k_2)
r(k_2)}{k_2^{s}}
\\
& \qquad\times\sum_{n =1}^{+\infty} \frac{r^{2}(n)}{n^{s}}
\prod_{p\mid (m_1 m_2)^{\infty}} \biggl( 1+ \frac{r^{2}(p)}{p^{s}} +
\frac{r^{2}(p^2)}{p^{2s}}+\cdots\biggr)^{-1}
\\
& = \frac{1}{(m_1 m_2)^s} K(m_1, s) K(m_2, s) D(s),
\end{align*}
where
\begin{gather*}
D(s) = \sum_{n =1}^{+\infty} \frac{r^{2}(n)}{n^{s}},
\\
\begin{align*}
K(m,s) & = \sum_{k \mid m^{\infty}} \frac{r(m k) r(k)}{k^{s}} \biggl(
\sum_{k \mid m^{\infty}} \frac{r^{2}(k)}{k^{s}} \biggr)^{-1} \\
&= \prod_{p| m} \biggl( 1+ \frac{r^{2}(p)}{p^{s}} +
\frac{r^{2}(p^2)}{p^{2s}}+\dotsb\biggr)^{-1}
\\
&\times\prod_{p^{\alpha} \| m} \biggl(
r(p^{\alpha}) + \frac{r(p^{\alpha+1})r(p)}{p^{s}} +
\frac{r(p^{\alpha+2})r(p^2)}{p^{2s}}+\cdots\biggr).
\end{align*}
\end{gather*}
So, $D_{m_1,m_2}(s)$ is a meromorphic function on the
half-plane $\operatorname{Re}s>\frac 12$ due to lemma~\ref{l2} and formula (\ref{eq21}).

Let us continue the function $\Phi(s,0)$ to the left of the half-plane $\Real s >1$.
Integrating once by parts, we get (supposing that $1 <\operatorname{Re}s < 5$)
$$
(s-1)\Phi(s,0) =- (s-1) \left. \phi (u) u^{1-s}\right|_{0}^{+\infty} + \int_0^{+\infty}\phi'(u) u^{1-s}\,du = \int_0^{+\infty}\phi'(u) u^{-s+1}\,du.
$$
Since~$\phi' (u) = -4 u^{-5} (1+\Delta^4) (\Delta^4+u^{-4})^{-2}$, then
\begin{align}\label{eq13.1}
&\phi'(u) \asymp
\begin{cases}  u^3 &\text{if }\ u \to +0,
\\
  u^{-5} &\text{if }\ u \to +\infty,
\end{cases}
& \text{ and also }
&\phi''(u) \asymp
\begin{cases}  u^2 &\text{if }\ u \to +0,
\\
  u^{-6} &\text{if }\ u \to +\infty.
\end{cases}
\end{align}
Therefore,  the function
$$
\Phi^*(s) :=\int_0^{+\infty}\phi'(u) u^{-s+1}\,du
$$
is well-defined in the region $-3< \Real s < 5$ and gives us  a continuation of $(s-1)\Phi(s,0)$ in that region.
Note, that the estimate
$$
\Phi^*(s)\ll 1
$$
holds uniformly if $|\operatorname{Re}s| \le 2$, for example.

Let us move the contour of integration in
\begin{align*}
& \int_{c-i\infty}^{c+i\infty} D_{m_1, m_2} (s) \Gamma(s) B^{-s}
\Phi(s+\theta, 0)\,ds
\\
&\qquad = \int_{c-i\infty}^{c+i\infty} \frac{K(m_1, s) K(m_2, s)
D(s) \Gamma(s) \Phi^{*} (s+\theta) (B m_1 m_2)^{-s}}{s+\theta-1}\,ds
\end{align*}
to the line $\operatorname{Re}s=\frac23$. We then pass through
a simple pole at $s=1-\theta$ and a second order pole at $s=1$.
By virtue of lemma~\ref{l2} and Stirling's
formula for~$\Gamma(s)$, we conclude that
\begin{align*}
&\frac{1}{2\pi i} \int_{c-i\infty}^{c+i\infty} D_{m_1, m_2} (s)
\Gamma(s) B^{-s} \Phi(s+\theta, 0) ds \\
&=\frac{1}{2\pi i} \int\limits_{|s-(1-\theta)| = r} \frac{ K(m_1, s)
K(m_2, s) D(s) \Gamma(s) \Phi^{*} (s+\theta) (B m_1 m_2)^{-s}}{s - (1-\theta)} ds
\\
&\qquad+ \frac{1}{2\pi i} \int\limits_{|s-1| = r} \frac{ K(m_1, s)
K(m_2, s) D(s) \Gamma(s) \Phi^{*} (s+\theta) (B m_1 m_2)^{-s}}{s +\theta- 1} ds + O \biggl(
\frac{(m_1 m_2)^{1/3}}{(B m_1 m_2)^{2/3}}\biggr),
\\
\end{align*}
where $r= \frac{1}{2\log X}$.
Having in mind that $B m_1 m_2 = \dfrac{\nu_1 \nu_3}{q}\dfrac{4\pi \sin\delta}{\sqrt{D}}$ and inserting the last equality into the expression of $J_1$, we get
\begin{align*}
J_1 (1, \theta) &\le  \frac{1}{2\pi i}  \int\limits_{|s-(1-\theta)| = r}
 \frac{D(s) \Gamma(s) \Phi^{*} (s+\theta) }{s - 1+\theta}  \left(\dfrac{4\pi \sin\delta}{\sqrt{D}}\right)^{-s} S(1-s) ds
\\
&\qquad+ \frac{1}{2\pi i} \int\limits_{|s-1| = r} \frac{D(s) \Gamma(s) \Phi^{*} (s+\theta) }{s - 1+\theta}  \left(\dfrac{4\pi \sin\delta}{\sqrt{D}}\right)^{-s} S(1-s) ds \\
&+ O \biggl( \delta^{-2/3} \sum\limits_{\nu_1, \ldots, \nu_4\le X} (\nu_1 \nu_3)^{-1/3} (\nu_2 \nu_4)^{-2/3} \biggr),
\end{align*}
where $S(z)$ is defined in lemma~\ref{l3}.  We shall estimate the two integrals from the last expression taking the maximum of the absolute value of the integrand and, then, trivially integrating  $|dz|$ over the circles. Therefore, using the estimates $\sin\delta \asymp \delta$
and $D(1-z)\ll|z|^{-2}$ (for small $|z|$), we get
\begin{align*}
J_1 (1, \theta) &\ll \frac{1}{\theta^2\delta^{1-\theta+r}} \max\limits_{|z-\theta| = r} |S(z)| +   \frac{1}{\theta r \delta^{1+r}} \max\limits_{|z| = r} |S(z)| + O(\delta^{-2/3} X^2).
\end{align*}
Using the result of lemma~\ref{l3}, we obtain the
relation
\begin{align*}
J_1 (1, \theta) &\ll \frac{X^{2\theta}}{\theta^2 \delta^{1-\theta} \log^2 X } + \frac{\log X}{\theta  \delta \log^2 X} +
\delta^{-2/3} X^2
\\
&\ll\frac{\delta^{-1}}{\theta  \log X} \bigl( 1+
(X^2  \delta)^{\theta} + X^2\delta^{1/3} (\theta \log X) \bigr) \ll\frac{\delta^{-1}}{\theta  \log X},
\end{align*}
since $\delta^{1/3} X^2 \log X\ll1$.


\begin{center}
{\bf \S 7. Auxiliary statements for the  non-diagonal part.}
\end{center}

We shall formulate several main statements.

\begin{lemma}\label{l8}
Let $N\gg1$, $(m_1,m_2)=1$, $(l, m_j)=1$, $m_1^{9}m_2^{11}\le N^2$, $l\le
N^{11/13}$, and
\begin{equation*}
S = \sum_{\substack{m_2 n_2 - m_1 n_1 = l, \\ n_1 \le N-1}} r (n_2) r (n_1).
\end{equation*}
Then for any $\varepsilon>0$ the following formula holds:
\begin{equation*}
S = \sigma (l, m_1, m_2) \frac{\pi^2 h^2 (-D) N}{m_2}  + O_{\varepsilon}\left( N^{\frac{11}{13}+\varepsilon} m_1^{9/13} m_2^{-2/13} \right),
\end{equation*}
where
\begin{equation}\label{eq1001}
\sigma (l, m_1, m_2) = \sum\limits_{q=1}^{+\infty} \varepsilon_q (l)\frac{(q, m_1 m_2)}{q^2  \sqrt{r_1 r_2}},
\end{equation}
$$
r_j =\dfrac{D}{\left( q_j, D\right)}, \quad q_j = \frac{q}{(q, m_j)},
$$
\begin{align*}
\varepsilon_q (l) &= \sum\limits_{\substack{0 < a^{*} < q \\a a^{*} \equiv 1 (\mod q)}} e^{-2\pi i
\frac{al}{q}} \left( \frac{G(\chi_{d_1})}{\sqrt{d_1}} \chi_{d_1} (a m_1') \chi_{d_2} \left( \frac{q_1}{d_1} \right) + \frac{G(\chi_{d_2})}{\sqrt{d_2}} \chi_{d_2} (a m_1') \chi_{d_1} \left( \frac{q_1}{d_2} \right)\right) \times \\
&\times
\left( \frac{\overline{G(\chi_{d_1})}}{\sqrt{d_1}} \chi_{d_1} (a m_2') \chi_{d_2} \left( \frac{q_2}{d_1} \right)+ \frac{\overline{G(\chi_{d_2})}}{\sqrt{d_2}} \chi_{d_2} (a m_2') \chi_{d_1} \left( \frac{q_2}{d_2} \right)\right),
\end{align*}
$m_j' = \frac{m_j}{(m_j, q)}$, and $G(\chi)$ is the quadratic Gauss sum for the character $\chi$ and we suppose that $\chi(\cdot)$ is zero for non-integer values. Also, the following estimate hold:
\begin{equation*}
\sigma (l, m_1, m_2) \ll \tau^2 (m_1 m_2) \tau(l).
\end{equation*}


\end{lemma}

\begin{pf}
The proof of this result is given in \cite{Rezvyakova_2017}, but the statement there is correct when $m_1 = m_2 = 1$. To get the correct statement in the general case one needs to operate along the proof with the precise value of the parameter $\sigma$ from \cite{IC_2002} on p. 278. See \cite{IC_2002}, \cite{Heap} for similar result in certain cases ($m_1 = m_2 = 1$ or $d_1 = 1$ consequently).
To get the estimate for $\sigma (l,m_1,m_2)$, note that it is a sum of four terms introduced in the beginning of the proof of lemma \ref{l9} (corresponding to all $q$ in the sum (\ref{eq1001}) of non-zero terms satisfy one of the four different conditions). In each case $\varepsilon_q(l)$ is estimated from above by the absolute value of Ramanujan sum with non-primitive Dirichlet character (see again the transformations in the proof of lemma \ref{l9}) and, hence, by (110) from \cite{Heap} we get the bound $\varepsilon_q(l) \ll \tau(q)\sqrt{q} (l,q)^{1/2}$. Therefore,
\begin{align*}
\sigma (l,m_1,m_2) &\ll \sum\limits_{q=1}^{+\infty} \frac{\tau (q)(q, m_1 m_2) (l,q)^{1/2}}{q^{3/2}} = \sum\limits_{d| m_1 m_2} \sum\limits_{(q, m_1 m_2) = d} \frac{\tau (q)(q, m_1 m_2) (l,q)^{1/2}}{q^{3/2}} \\
&\le
\sum\limits_{d| m_1 m_2} \frac{\tau (d)(l,d)^{1/2}}{d^{1/2}} \sum\limits_{(q, m_1 m_2) = 1}\frac{\tau (q) (l,q)^{1/2}}{q^{3/2}} \le
\sum\limits_{d| m_1 m_2} \frac{\tau (d)(l,d)^{1/2}}{d^{1/2}}\sum\limits_{d_1| l} \frac{\tau (d)}{d} \sum\limits_{q = 1}^{+\infty} \frac{\tau (q) }{q^{3/2}} \\
&\ll \sum\limits_{d| m_1 m_2} \tau (d)\sum\limits_{d_1| l} 1 \ll
\tau^2 (m_1 m_2) \tau(l).
\end{align*}
\end{pf}

From lemma~\ref{l8} we get
\begin{cor}
\label{c1} Consider the function
$$
D_{m_1,m_2}(s,l)=\sum_{n=1}^{+\infty}\frac{r(n)r
\bigl(\frac{m_1n+l}{m_2}\bigr)}{\bigl(m_1n+\frac l2\bigr)^s}.
$$
Then
the  difference
$$
D_{m_1,m_2}(s,l) - \frac{\pi^2 h^2(-D) \sigma (l, m_1, m_2)}{m_1^s m_2} \zeta (s)
$$
with $\sigma (l, m_1, m_2)$ defined by (\ref{eq1001}) is analytic in the half-plane $\operatorname{Re}s\ge\frac{11}{13}+\varepsilon_0$ and, for $0<\varepsilon_0\le\frac{2}{13}$, in the region  $\frac{11}{13}+\varepsilon_0 \le \operatorname{Re}s\le 2$ the estimate
$$
D_{m_1,m_2}(s,l) - \frac{\pi^2 h^2(-D) \sigma (l, m_1, m_2)}{m_1^s m_2} \zeta (s)\ll_{\varepsilon_0}
\frac{|s|}{m_1^{{11}/{13}+\varepsilon_0}m_2}((m_1^{9/2}
m_2^{11/2})^{2/13-{\varepsilon_0}/2}+l^{2/{11}- {13\varepsilon_0}/{22}})
$$
holds.
\end{cor}

\begin{proof} Take
$N=\max\bigl(m_1^{9/2}m_2^{11/2},l^{{13}/{11}}\bigr)$. Using the inequality
$\tau(n)\ll_{\varepsilon_0}n^{\varepsilon_0/8}$,  for
$\Real s = \sigma\ge\frac{11}{13}+\varepsilon_0$ we obtain
\begin{align*}
S_1 (s)& := \sum_{n\le N} \frac{r(n) r \bigl( \frac{m_1 n+l}{m_2}
\bigr)}{\bigl(m_1 n + \frac l2\bigr)^{s}} \ll\sum_{\substack{n\le N
\\
m_1 n \equiv-l \ (\operatorname{mod}m_2)}} \frac{\tau(n) \tau(
\frac{m_1 n+l}{m_2})}{\bigl(m_1 n + \frac l2\bigr)^{\sigma}}
\\
& \ll_{\varepsilon_0} N^{3\varepsilon_0/8} \sum_{\substack{n\le N
\\
m_1 n \equiv-l \ (\operatorname{mod}m_2)}} \frac{1}{(m_1
n)^{11/13+\varepsilon_0}} \ll_{\varepsilon_0}
\frac{N^{{2}/{13}-{\varepsilon_0}/{2}}}{m_1^{{11}/{13}+\varepsilon_0}
m_2}.
\end{align*}

By Abel's partial summation formula, for $\operatorname{Re}s>1$ and
$M>N$,
\begin{align*}
&\sum_{N < n \le M} \frac{r(n) r \bigl( \frac{m_1 n+l}{m_2}
\bigr)}{\bigl(m_1 n + \frac l2\bigr)^{s}}  - \frac{\pi^2 h^2 (-D) \sigma (l, m_1, m_2)}{m_2} \sum_{N < n \le M } \frac{1}{\bigl(m_1 n + \frac l2\bigr)^{s}} \\
& = s \int_{N}^{M}
\frac{\mathbb{C} (u) m_1}{\bigl(m_1 u + \frac l2\bigr)^{s+1}}\,du +
\frac{\mathbb{C} (M)}{\bigl(m_1 M + \frac l2\bigr)^{s}},
\end{align*} where
$$
\mathbb C(u)=\sum_{N<n\le u } \biggl( r(n)r\biggl(\frac{m_1 n+l}{m_2}\biggr) - \frac{\pi^2 h^2(-D)\sigma (l, m_1, m_2)}{m_2}\biggr).
$$
Letting $M\to+\infty$, we arrive at the relation
\begin{equation}
\label{eq13} S_2 (s):= \sum_{n > N } \frac{r(n) r \bigl( \frac{m_1
n+l}{m_2} \bigr) - \pi^2 h^2(-D)\sigma (l, m_1, m_2)/ m_2}{\bigl(m_1 n +\frac l2\bigr)^{s}} = s
\int_{N}^{+\infty} \frac{\mathbb{C} (u) m_1}{\bigl(m_1 u + \frac
l2\bigr)^{s+1}}\,du.
\end{equation}
Since $u\ge N$, by lemma ~\ref{l8} we see that $\mathbb
C(u)\ll_\varepsilon u^{11/13+\varepsilon}m_1^{9/13}m_2^{-2/13}$.
Hence the right-hand side of~\eqref{eq13} defines a holomorphic
function in the region $\operatorname{Re}s>\frac{11}{13}$ and,
furthermore, for $\operatorname{Re}s\ge\frac{11}{13}+\varepsilon_0$
it obeys the estimate
\begin{align*}
S_2 (s) &\ll_{\varepsilon_0} |s| \int_{N}^{+\infty} \frac{
m_1^{9/13} m_2^{-2/13} u^{11/13+\varepsilon_0/2} m_1}{(m_1
u)^{11/13+\varepsilon_0+1}}\,du
\\
&\ll_{\varepsilon_0} |s| (m_1 m_2)^{-2/13} m_1^{-\varepsilon_0}
N^{-\varepsilon_0/2} \ll_{\varepsilon_0} \frac{|s|
N^{{2}/{13}-{\varepsilon_0}/{2}}}{m_1^{{11}/{13}+\varepsilon_0}
m_2}.
\end{align*}
Also, for $11/13 + \varepsilon_0 \le \Real s\le 2$, since $0 < \frac{l}{2 m_1} < l \le N^{11/13}$, we have
\begin{align*}
&\sum_{n > N} \frac{1}{\bigl(m_1 n + \frac l2\bigr)^{s}} - \sum_{n > N} \frac{1}{\bigl(m_1 n\bigr)^{s}} =
\frac{1}{m_1^s} \left( \sum_{ n > N+[l/2m_1]} \left( \frac{1}{\bigl(n + \{\frac{l}{2 m_1} \}\bigr)^{s}} -\frac{1}{n^{s}} \right) - \sum_{N < n \le N+[l/2m_1]} \frac{1}{n^{s}} \right)\\
& \ll m_1^{-11/13-\varepsilon_0} \left(  N^{-11/13-\varepsilon_0} N^{11/13}+ \sum_{ n > N+[l/2m_1]} \frac{|s|}{n^{1+11/13+\varepsilon_0}}\right)  \ll |s| m_1^{-11/13-\varepsilon_0}  N^{-\varepsilon_0},
\end{align*}
and
\begin{align*}
\sum_{n \le N} \frac {1}{n^s} \ll N^{2/13-\varepsilon_0}.
\end{align*}
Since, by lemma \ref{l8}, $\sigma (l,m_1,m_2) \ll \tau^2 (m_1 m_2) \tau(l) \ll N^{\varepsilon_0/2}$, we thus come to the statement of the corollary.
\end{proof}

\begin{lemma}\label{l9}
For $(m_1, m_2)=1$ and $\sigma (l, m_1, m_2)$ defined by (\ref{eq1001}),
the function
$$
Z_{m_1, m_2} (s) = \sum\limits_{l=1}^{+\infty} \frac{\sigma (l, m_1, m_2) }{l^s}
$$
has a meromorphic continuation to the whole complex plane with a simple pole at $s=1$ only in case $d_1 =1$ or $d_2=1$ and with a simple pole at $s=0$. Moreover, we can write $Z_{m_1, m_2} (s)$ as the sum of  four terms
$$
Z_{m_1, m_2} (s) = Z_{m_1, m_2}^{(1,1)} (w)+ Z_{m_1, m_2} ^{(2,2)}(w)+ Z_{m_1, m_2}^{(1,2)} (w)+Z_{m_1, m_2}^{(2,1)} (w),
$$
where for $j=1,2$
\begin{align*}
&Z_{m_1, m_2} ^{(j,j)}(w) = \frac{\zeta(s) \zeta(s+1)}{\zeta(2) d_j^s D} \prod\limits_{p | d_j}  \left( 1-p^{s-1}\right)\prod\limits_{p | (D/d_j)} \left( 1 - \frac{1}{p^{ s+1}} \right)
 \prod\limits_{p|D} \left( 1- \frac{1}{p^{2}} \right)^{-1}  K_{j,j} (m_1m_2, s),
\end{align*}
and for $k\ne j$
\begin{align*}
&Z_{m_1, m_2}^{(k,j)} (w) = \frac{\overline{G^2(\chi_{d_j})}}{d_j} D^{-1/2} L_{\chi_D}(s) L_{\chi_D}(s+1) L_{\chi_D}^{-1} (2) K_{1,2} (m_k, m_j, s),
\end{align*}
where $K_{k,j}  (\cdot)$ are defined on page \pageref{p1}.
\end{lemma}

\begin{pf}
We follow \cite{Heap} here. Write $\sigma (l, m_1, m_2)$ as the sum of four summands as follows:
$$
\sigma (l, m_1, m_2) = \sum\limits_{k,j=1}^{2} \sigma_{k,j} (l, m_1, m_2), \quad
\sigma_{k,j} (l, m_1, m_2) = \sum\limits_{q \in Q_{k,j}} \varepsilon_{q} (l)\frac{(q, m_1 m_2)}{q^2 \sqrt{r_1 r_2}},
$$
where for the given $d_1, d_2, m_1, m_2$ the four non-intersecting classes $Q_{k,j}$ ($k,j = 1,2$) of positive integers $q$ (that correspond to non-zero terms in  $\sigma (l, m_1, m_2)$)
are defined by the following conditions.
We say that $q \in Q_{k,j}$ if
$$
\left( \frac{q}{(q, m_1)}, d_1 d_2\right) = d_k, \quad \left( \frac{q}{(q, m_2)}, d_1 d_2\right) = d_j.
$$
Consider , first, all $q\in Q_{1,1}$ and write the explicit form for them. The above conditions imply that
$d_1| q$ and, since $(m_1, m_2) = 1$ and $(d_1, d_2)=1$, that $(q, d_2)= 1$. Suppose that a prime $p$ divides $(m_1, d_1)$ and that $p^{\alpha_1} || d_1$ ,
$p^{\beta_1} || m_1$ (notice, that $\alpha_1 = 1$ for all prime $p$ except for at most $p=2$  since $-D = d_1 d_2$ is the fundamental discriminant).
From our conditions we deduce that
$$
\left( \frac{q}{(q, p^{\beta_1})}, p^{\alpha_1}\right) = p^{\alpha_1},
$$
which implies that $p^{\alpha_1+\beta_1} | q$. So, all $q\in Q_{1,1}$ can be represented as
$$
q = d_1 (m_1 m_2, d_1^{\infty}) f_{1,1} f,
$$
where $(f, D)=1$, $f_{1,1} | d_1^{\infty}$.

Similarly, for $q\in Q_{2,2}$ we get the representation
$$
q = d_2 (m_1 m_2, d_2^{\infty}) f_{2,2} f,
$$
where $(f, D)=1$, $f_{2,2} | d_2^{\infty}$.

Let $q\in Q_{2,1}$.
It follows that $D|q$, $d_1 | m_1$, $d_2 | m_2$, $(d_2, m_1) = 1$, $(d_1, m_2) = 1$. Also, if $p^{\alpha_1} || d_1, p^{\beta_1} || m_1/ d_1$,
then our conditions give
$$
\left( \frac{q}{(q, p^{\alpha_1 +\beta_1})},  p^{\alpha_1}\right) = 1, \quad p^{\alpha_1} | q,
$$
and, thus, the following representation in case $d_1 | m_1$, $d_2 | m_2$:
$$
q = D f_{2,1} f,
$$
where $(f, D)=1$, $f_{2,1} | (m_1 m_2/ D, D^{\infty})$.

The same representation for $q \in Q_{1,2}$ is valid when $d_1 | m_2$, $d_2 | m_1$.

All the rest values of $q$ (which do not belong to any of $Q_{k,j}$) correspond to zero terms of $\sigma (l, m_1, m_2)$.

Let us transform now the value
$$
\sigma_{1,1}(l, m_1, m_2) = \sum\limits_{q\in Q_{1,1}} \varepsilon_q (l) \frac{(q, m_1 m_2)}{q^2 \sqrt{r_1 r_2}}.
$$
For $q\in Q_{1,1}$ (i.e. when $\left( q_j, d_1 d_2\right)=d_1$ for $j=1,2$) we may write $q$ in the form $q = d_1 (m_1 m_2, d_1^{\infty}) f_{1,1} f$.  Therefore, we have $r_j = d_2$  ($j=1,2$) and, hence, denoting by $c_q (l)$ the Ramanujan sum
$$
c_q (l) = \sum\limits_{\substack{a (\mod q) \\ (a,q)=1}} e^{-2\pi i l \frac{a}{q}},
$$
we can write
$$
\varepsilon_q (l) = c_q (l) \varepsilon_q,
$$
where (with $m'_j$ defined in lemma\ref{l8})
\begin{align*}
\varepsilon_q &=\chi_{d_1} (m_1' m_2') \chi_{d_2} \left( \frac{q_1}{d_1} \cdot \frac{q_2}{d_1}\right)=\chi_{d_1} \left( \frac{m_1 m_2}{(q, m_1 m_2)}\right) \chi_{d_2} \left( \frac{q^2}{d_1^2 (q, m_1 m_2)}\right) \\
&=\chi_{d_1} \left( \frac{m_1 m_2}{(m_1 m_2, d_1^{\infty}) (m_1 m_2, f)}\right) \chi_{d_2} \left( \frac{(f_{1,1} f)^2 (m_1 m_2, d_1^{\infty})}{(f, m_1 m_2)}\right) \\
&= \chi_{d_1} \left( \frac{m_1 m_2}{(m_1 m_2, d_1^{\infty})}\right) \chi_{d_2} \left( (m_1 m_2, d_1^{\infty}) \right) \chi_{D} \left( (f, m_1 m_2)\right).
\end{align*}
Using the formula
$$
c_q (l) = \sum\limits_{d| (q,l)} d \mu\left(\frac{q}{d}\right),
$$
we see that
$$
\sum\limits_{l=1}^{\infty} \frac{c_q(l)}{l^s} = \zeta(s) \sum\limits_{d|q} \mu(d) \left( \frac{q}{d} \right)^{1-s}.
$$
Therefore,
\begin{align*}
&\sum\limits_{l=1}^{\infty} \frac{\sigma_{1,1} (l, m_1, m_2)}{l^s} = \zeta(s)\sum\limits_{q\in Q_{1,1}} \varepsilon_q \frac{ (q,m_1 m_2)}{q^{s+1} \sqrt{r_1 r_2}} \sum\limits_{d|q} \mu(d) {d}^{s-1} \\
&= \frac{\zeta(s)}{d_1^{s} D (m_1 m_2, d_1^{\infty})^{s}} \chi_{d_1} \left( (m_1 m_2, d_2^{\infty}) \right)  \chi_{d_2} \left( (m_1 m_2, d_1^{\infty}) \right)  \chi_{d_1} \left( \frac{m_1 m_2}{(m_1 m_2, D^{\infty})}\right)  \times \\
&\times\sum\limits_{q_1 | d_1^{\infty}}\frac{1}{q_1^{s+1}} \prod\limits_{p| d_1}  (1-p^{s-1}) \sum\limits_{(q, D)=1} \chi_{D} \left( (q,  m_1 m_2)\right) \frac{ (q,m_1 m_2)}{q^{s+1}}\prod\limits_{p|q} (1- {p}^{s-1}) \\
&= \frac{\zeta(s)}{d_1^s D (m_1 m_2, d_1^{\infty})^{s}} \chi_{d_1} \left( (m_1 m_2, d_2^{\infty}) \right)  \chi_{d_2} \left( (m_1 m_2, d_1^{\infty}) \right)  \chi_{d_1} \left( \frac{m_1 m_2}{(m_1 m_2, D^{\infty})}\right)  \times \\
&\times\prod\limits_{p | d_1} \left( 1-\frac{1}{p^{s+1}} \right)^{-1} \left( 1-p^{s-1}\right)\prod\limits_{(p, D m_1 m_2)=1} \left( 1 + (1- {p}^{s-1}) \left( \frac{1}{p^{s+1}} + \frac{1}{p^{2(s+1)}}+\ldots \right)\right) \times\\
&\times \prod\limits_{\substack{(p, D)=1, \\ p^{\alpha} || m_1 m_2}} \left( 1 + (1- {p}^{s-1}) \left( \frac{\chi_D (p)}{p^{s}} + \frac{\chi_D (p^2)}{p^{2s}}+\ldots+\frac{\chi_D (p^{\alpha})}{p^{\alpha s}} +  \frac{\chi_D (p^{\alpha})}{p^{(\alpha +1) s+1}} \left( 1 - \frac{1}{p^{ s+1}} \right)^{-1}
 \right)\right)\\
&= \frac{\zeta(s)}{d_1^s D (m_1 m_2, d_1^{\infty})^{s}} \chi_{d_1} \left( (m_1 m_2, d_2^{\infty}) \right)  \chi_{d_2} \left( (m_1 m_2, d_1^{\infty}) \right)  \chi_{d_1} \left( \frac{m_1 m_2}{(m_1 m_2, D^{\infty})}\right)   \times\\
&\times \prod\limits_{p | d_1} \left( 1-\frac{1}{p^{s+1}} \right)^{-1} \left( 1-p^{s-1}\right) \prod\limits_{(p, D m_1 m_2)=1} \left( 1-\frac{1}{p^{s+1}} \right)^{-1} \left( 1- \frac{1}{p^{2}} \right)\times \\
&\times\prod\limits_{\substack{(p, D)=1, \\ p^{\alpha} || m_1 m_2}} \left( 1 + (1- {p}^{s-1}) \left( \frac{\chi_D (p)}{p^{s}} \left( 1-\frac{\chi_D (p^{\alpha})}{p^{\alpha s}} \right) \left( 1- \frac{\chi_D (p)}{p^{s}}\right)^{-1} +  \frac{\chi_D (p^{\alpha})}{p^{(\alpha +1) s+1}} \left( 1 - \frac{1}{p^{ s+1}} \right)^{-1}
 \right)\right) \\
&= \frac{\zeta(s) \zeta(s+1) }{\zeta(2) d_1^s D (m_1 m_2, d_1^{\infty})^{s}} \chi_{d_1} \left( (m_1 m_2, d_2^{\infty}) \right)  \chi_{d_2} \left( (m_1 m_2, d_1^{\infty}) \right)  \chi_{d_1} \left( \frac{m_1 m_2}{(m_1 m_2, D^{\infty})}\right)   \times\\
&\times \prod\limits_{p | d_1}  \left( 1-p^{s-1}\right) \prod\limits_{p | d_2} \left( 1 - \frac{1}{p^{ s+1}} \right) \prod\limits_{p|D} \left( 1- \frac{1}{p^{2}} \right)^{-1} \prod\limits_{\substack{(p, D)=1, \\ p^{\alpha} || m_1 m_2}} \left( 1- \frac{1}{p^{2}} \right)^{-1} \left( 1 - \frac{\chi_D (p)}{p^{ s}} \right)^{-1} \times \\
&\times\left( \left( 1 - \frac{\chi_D (p)}{p^{ s}} \right) \left( 1-\frac{1}{p^{s+1}} \right) + \right. \\
&\left.+(1- {p}^{s-1}) \left( \frac{\chi_D (p)}{p^{s}} \left( 1-\frac{\chi_D (p^{\alpha})}{p^{\alpha s}} \right)  \left( 1-\frac{1}{p^{s+1}} \right) + \frac{\chi_D (p^{\alpha})}{p^{(\alpha +1) s+1}}   \left( 1 - \frac{\chi_D (p)}{p^{ s}} \right)\right)\right) \\
&= \frac{\zeta(s) \zeta(s+1)}{\zeta(2) d_1^s D (m_1 m_2, d_1^{\infty})^{s}} \chi_{d_1} \left( (m_1 m_2, d_2^{\infty}) \right)  \chi_{d_2} \left( (m_1 m_2, d_1^{\infty}) \right)  \chi_{d_1} \left( \frac{m_1 m_2}{(m_1 m_2, D^{\infty})}\right)   \times\\
&\times \prod\limits_{p | d_1}  \left( 1-p^{s-1}\right)\prod\limits_{p | d_2} \left( 1 - \frac{1}{p^{ s+1}} \right)\prod\limits_{p|D} \left( 1- \frac{1}{p^{2}} \right)^{-1}  \times \\
&\times\prod\limits_{\substack{(p, D)=1, \\ p^{\alpha} || m_1 m_2}} \left( 1- \frac{1}{p^{2}} \right)^{-1} \left( 1 - \frac{\chi_D (p)}{p^{ s}} \right)^{-1} \left( 1 - \frac{\chi_D (p)}{p} \right) \times \left( 1-\frac{1}{p^{s+1}}  + \frac{\chi_D (p^{\alpha+1})}{p^{\alpha s+1}}  \left( 1-\frac{1}{p^{s-1}} \right) \right) \\
&= \frac{\zeta(s) \zeta(s+1)}{\zeta(2) d_1^s D} \prod\limits_{p | d_1}  \left( 1-p^{s-1}\right)\prod\limits_{p | d_2} \left( 1 - \frac{1}{p^{ s+1}} \right)\prod\limits_{p|D} \left( 1- \frac{1}{p^{2}} \right)^{-1}  K_{1,1} (m_1 m_2, s).
\end{align*}

Similarly, we get
\begin{align*}
&\sum\limits_{l=1}^{\infty} \frac{\sigma_{2,2} (l, m_1, m_2)}{l^s}= \frac{\zeta(s) \zeta(s+1)}{\zeta(2) d_2^s D} \prod\limits_{p | d_1}  \left( 1-p^{s-1}\right)\prod\limits_{p | d_2} \left( 1 - \frac{1}{p^{ s+1}} \right)\prod\limits_{p|D} \left( 1- \frac{1}{p^{2}} \right)^{-1}  K_{2,2} (m_1 m_2, s).
\end{align*}

Now, for $d_j \mid m_j$, consider the quantity $\sigma_{2,1}(l, m_1, m_2)$ corresponding to all the summands in $\sigma (l, m_1, m_2)$
with $q\in Q_{2,1}$ (i.e. when $\left( q_1, d_1 d_2\right)=d_2$, $\left( q_2, d_1 d_2\right)=d_1$). For such $q$ we already get the representation $q = D  f_{2,1} f$, where $f_{2,1} | (m_1 m_2/D, D^{\infty})$, $(f, D)=1$. Hence,  $r_1 r_2 = D$ and
\begin{align*}
\varepsilon_q (l)&= \frac{\overline{G(\chi_{d_1})}}{\sqrt{d_1}}  \frac{G(\chi_{d_2})}{\sqrt{d_2}} \chi_{d_1} \left( \frac{q_1}{d_2} \right) \chi_{d_2} \left( \frac{q_2}{d_1} \right) \chi_{d_2} (m_1')
\chi_{d_1} (m_2')\sum\limits_{\substack{a=1,\\ (a,q)=1}}^{q} e^{-2\pi i l\frac{a}{q}} \chi_{D} (a).
\end{align*}
 First of all, we note that
\begin{align*}
&\chi_{d_2} \left( \frac{m_1}{(q, m_1)}\right) \chi_{d_1} \left( \frac{m_2}{(q, m_2)}\right) \chi_{d_1} \left( \frac{q}{d_2 (q, m_1)}\right) \chi_{d_2} \left( \frac{q}{d_1 (q, m_2)}\right)\\
&= \chi_{d_1} \left( \frac{m_2/d_2}{(f_{2,1}, d_2^{\infty}) (m_2, f)}\right) \chi_{d_2} \left( \frac{m_1/d_1}{(f_{2,1}, d_1^{\infty}) (m_1, f)}\right)\times\\
&\times \chi_{d_1} \left( \frac{f (f_{2,1}, d_2^{\infty})}{(f, m_1)}\right)  \chi_{d_2} \left( \frac{f (f_{2,1}, d_1^{\infty})}{(f, m_2)}\right)\\
&= \chi_{d_1} \left( m_2/ d_2\right) \chi_{d_2} \left( m_1/d_1\right) \chi_D (f/ (m_1 m_2, f)).
\end{align*}
Second, for $q \equiv 0 (\mod D)$ we shall use the formula (see lemma 4 of \cite{Heap})
$$
c_q (l, \chi_D) :=\sum\limits_{\substack{a=1, \\ (a,q)=1}}^{q} \chi_D (a) e^{-2\pi i \frac{al}{q}} = \overline{G(\chi_D)} \sum\limits_{d| (q/D,l)} d \mu\left(\frac{q/D}{d}\right) \chi_D \left(\frac{q/D}{d}\right) \chi_D \left(\frac{l}{d}\right)
$$
to see that
$$
\sum\limits_{l=1}^{\infty} \frac{c_q(l, \chi_D)}{l^s} = L_{\chi_D}(s)  \overline{G(\chi_D)} \sum\limits_{d|q/D} \mu(d) \chi_D (d) \left( \frac{q}{Dd} \right)^{1-s}.
$$
Therefore,
\begin{align*}
&\sum\limits_{l=1}^{\infty} \frac{\sigma_{2,1} (l, m_1, m_2)}{l^s} = \frac{\overline{G(\chi_{d_1})} G(\chi_{d_2})   \overline{G(\chi_D)} L_{\chi_D}(s) D^{s-1}}{\sqrt{D}}\chi_{d_1} \left( m_2/ d_2\right) \chi_{d_2} \left( m_1/d_1\right) \times\\
&\times \sum\limits_{q\in Q_{2,1}} \frac{ (q,m_1 m_2)}{q^{s+1}} \chi_D (f/ (m_1 m_2, f))\sum\limits_{d|q/D} \mu(d) \chi_D (d) {d}^{s-1}.
\end{align*}
We transform further the double sum from the last expression:
\begin{align*}
&\sum\limits_{q\in Q_{2,1}} \frac{ (q,m_1 m_2)}{q^{s+1}} \chi_D (f/ (m_1 m_2, f))\sum\limits_{d|q/D} \mu(d) \chi_D (d) {d}^{s-1}\\
&=  \sum\limits_{q_1 | (m_1 m_2/D, D^{\infty})}\frac{1}{q_1^{s} D^{s}}
\sum\limits_{(q, D)=1} \chi_{D} \left( \frac{q}{(q,  m_1 m_2)}\right) \frac{ (q,m_1 m_2)}{q^{s+1}}\prod\limits_{p|q} (1- \chi_D (p){p}^{s-1}) \\
&= D^{-s}  \sum\limits_{q_1 | (m_1 m_2/D, D^{\infty})}\frac{1}{q_1^{s}}
\prod\limits_{(p, D m_1 m_2)=1} \left( 1 + (1- \chi_D (p) {p}^{s-1}) \left( \frac{\chi_D (p)}{p^{s+1}} + \frac{\chi_D (p^2)}{p^{2(s+1)}}+\ldots \right)\right) \times\\
&\times \prod\limits_{\substack{(p, D)=1, \\ p^{\alpha} || m_1 m_2}} \left( 1 + (1- \chi_D (p){p}^{s-1}) \left( \frac{1}{p^{s}} + \frac{1}{p^{2s}}+\ldots+\frac{1}{p^{\alpha s}} +  \frac{\chi_D (p)}{p^{(\alpha +1) s+1}} \left( 1 - \frac{\chi_D(p)}{p^{ s+1}} \right)^{-1}
 \right)\right)\\
&= D^{-s} \sum\limits_{q_1 | (m_1 m_2/D, D^{\infty})}\frac{1}{q_1^{s}}
\prod\limits_{(p, D m_1 m_2)=1} \left( 1 + (1- \chi_D (p) {p}^{s-1}) \frac{\chi_D (p)}{p^{s+1}} \left( 1 -  \frac{\chi_D (p)}{p^{s+1}}\right)^{-1}\right) \times\\
&\times \prod\limits_{\substack{(p, D)=1, \\ p^{\alpha} || m_1 m_2}} \left( 1 + (1- \chi_D (p){p}^{s-1}) \left( \frac{1}{p^{s}} \left( 1- \frac{1}{p^{\alpha s}}\right) \left(1-\frac{1}{p^{s}} \right)^{-1}+  \frac{\chi_D (p)}{p^{(\alpha +1) s+1}} \left( 1 - \frac{\chi_D(p)}{p^{ s+1}} \right)^{-1}
 \right)\right)\\
&= D^{-s} \sum\limits_{q_1 | (m_1 m_2/D, D^{\infty})}\frac{1}{q_1^{s}}
\prod\limits_{(p, D m_1 m_2)=1} \left(1 -  \frac{\chi_D (p)}{p^{s+1}}\right)^{-1} \left(1 -  \frac{\chi_D^2 (p)}{p^{2}}\right) \times \\
&\times\prod\limits_{\substack{(p, D)=1, \\ p^{\alpha} || m_1 m_2}}  \left( 1 - \frac{1}{p^{s}} \right)^{-1} \left( 1 - \frac{\chi_D (p)}{p^{s+1}} \right)^{-1} \left( 1 - \frac{\chi_D (p)}{p} \right) \times \\
&\times\left( 1-\frac{\chi_D(p)}{p^{s+1}}  + \frac{\chi_D (p)}{p^{\alpha s+1}}  \left( 1-\frac{\chi_D(p)}{p^{s-1}} \right) \right).
\end{align*}
Finally, we get
\begin{align*}
&\sum\limits_{l=1}^{\infty} \frac{\sigma_{2,1} (l, m_1, m_2)}{l^s} = \overline{G(\chi_{d_1})} G(\chi_{d_2})  \overline{G(\chi_D)} D^{-3/2} L_{\chi_D}(s) L_{\chi_D}(s+1) L_{\chi_{D}^2}^{-1} (2)\chi_{d_1} \left( m_2/ d_2\right) \chi_{d_2} \left( m_1/d_1\right) \times\\
&\times \sum\limits_{q | (m_1 m_2/D, D^{\infty})}\frac{1}{q^{s}}  \prod\limits_{\substack{(p, D)=1, \\ p^{\alpha} || m_1 m_2}} \left( 1- \frac{\chi^2_{D} (p)}{p^{2}} \right)^{-1} \left( 1 - \frac{1}{p^{ s}} \right)^{-1} \left( 1 - \frac{\chi_D (p)}{p} \right) \times \\
&\times \left( 1-\frac{\chi_D(p)}{p^{s+1}}  + \frac{\chi_D (p)}{p^{\alpha s+1}}  \left( 1-\frac{\chi_D(p)}{p^{s-1}} \right) \right) \\
&= \overline{G(\chi_{d_1})} G(\chi_{d_2})  \overline{G(\chi_D)} D^{-3/2}  \chi_{d_1} (d_2) \chi_{d_2} (d_1) L_{\chi_D}(s) L_{\chi_D}(s+1) L_{\chi_{D}}^{-1} (2)  K_{1,2} (m_2, m_1, s).
\end{align*}
Using the equality (see  \S 4 of chapter 1 in \cite{Chudakov} or (3.16) in \cite{Iwaniec_Kow})
\begin{align} \label{eq88}
G(\chi_D) =G(\chi_{d_1}) G(\chi_{d_2}) \chi_{d_1} (d_2) \chi_{d_2} (d_1),
\end{align}
for real primitive characters $\chi_{d_1}, \chi_{d_2}$ with the property $\chi_D (\cdot)= \chi_{d_1} (\cdot) \chi_{d_2} (\cdot)$, we finally get
\begin{align*}
&\sum\limits_{l=1}^{\infty} \frac{\sigma_{2,1} (l, m_1, m_2)}{l^s} =  \overline{(G(\chi_{d_1}))^2} d_1^{-1} D^{-1/2}  L_{\chi_D}(s) L_{\chi_D}(s+1) L_{\chi_{D}}^{-1} (2)  K_{1,2} (m_2, m_1, s).
\end{align*}

Similarly, if $d_2| m_1$ and $d_1| m_2$, we have
\begin{align*}
&\sum\limits_{l=1}^{\infty} \frac{\sigma_{1,2} (l, m_1, m_2)}{l^s}=  \overline{(G(\chi_{d_2}))^2}  d_2^{-1} D^{-1/2} L_{\chi_D}(s) L_{\chi_D}(s+1) L_{\chi_{D}}^{-1} (2)  K_{1,2} (m_1, m_2, s).
\end{align*}

\end{pf}

\begin{lemma}
$$
K_{1,1} (m, -s) m^{-s} = K_{2,2} (m, s),
$$
\end{lemma}
\begin{pf}
Since $K_{j,j} (\cdot, z)$ is the multiplicative function, we have to prove the desired equality for $m=p^{\alpha}$ where  $p\mid D$ (p is a prime) or $(p, D)=1$.

In case $p\mid D$  we have to show that
$$
(p^{\alpha}, d_1^{\infty})^{s} p^{-\alpha s} = (p^{\alpha}, d_2^{\infty})^{-s}
$$
which is true since $D = d_1 d_2$ and $(d_1, d_2) = 1$.

In  case $(p, D) = 1$ we have to show that
\begin{equation}\label{eq1111}
\begin{split}
&\chi_{d_1} (p^{\alpha}) p^{-\alpha s} \left( 1 - \chi_D (p) p^{ s} \right)^{-1}  \left( 1-p^{s-1}  + \chi_D (p^{\alpha+1}) p^{\alpha s-1}  \left( 1-p^{s+1} \right) \right)\\
&= \chi_{d_2} (p^{\alpha})
\left( 1 - \frac{\chi_D (p)}{p^s} \right)^{-1}  \left( 1-\frac{1}{p^{s+1}}  + \frac{\chi_D (p^{\alpha+1})}{p^{\alpha s+1}}  \left( 1-\frac{1}{p^{s-1}} \right) \right)
\end{split}
\end{equation}
We transform, for example, the left side (L.S.) of (\ref{eq1111}) as follows.
\begin{equation}\label{eq1112}
\begin{split}
L.S. &= \chi_{d_1} (p^{\alpha}) \chi_{D} (p) p^{-s} \left( 1 - \chi_D (p) p^{-s} \right)^{-1}  \chi_{D} (p^{\alpha+1})p^{-\alpha s} \left( p^{\alpha s-1}  \left(p^{s+1} -1\right) + \chi_D (p^{\alpha+1}) (p^{s-1}-1)  \right)\\
&= \chi_{d_2} (p^{\alpha})
\left( 1 - \frac{\chi_D (p)}{p^s} \right)^{-1}  \left( 1-\frac{1}{p^{s+1}}  + \frac{\chi_D (p^{\alpha+1})}{p^{\alpha s+1}}  \left( 1-\frac{1}{p^{s-1}} \right) \right).
\end{split}
\end{equation}
The proof is completed.

\end{pf}

\begin{cor}
Let
\begin{equation*}
\begin{split}
\tilde{S}_{1,1} (w) = S_{1,1} (w) d_1^{-w/2} d_2^{w/2} \prod\limits_{p\mid d_1} (1-p^{w-1}) \prod\limits_{p\mid d_2} (1-p^{-w-1}),
\end{split}
\end{equation*}
\begin{equation*}
\begin{split}
\tilde{S}_{2,2} (w) = S_{2,2} (w) d_2^{-w/2} d_1^{w/2} \prod\limits_{p\mid d_2} (1-p^{w-1}) \prod\limits_{p\mid d_1} (1-p^{-w-1}).
\end{split}
\end{equation*}
Then
$$
\tilde{S}_{2,2} (-w) = \tilde{S}_{1,1} (w).
$$
\end{cor}

\begin{pf}
The statement can be proved just by substituting the result of the previous lemma to the formulae for $S_{j,j} (w)$ of lemma \ref{l5} and making the substitutions $\nu_1 \leftrightarrow \nu_3$, $\nu_2 \leftrightarrow  \nu_4$.
\end{pf}

\begin{center}
{\bf \S 8. Estimation of the non-diagonal part.}
\end{center}

We shall estimate in this section
\begin{equation*}
\begin{split}
J_2(1, \theta) &= \sum\limits_{\nu_1,\nu_2, \nu_3,
\nu_4}
\frac{\beta(\nu_1)\beta(\nu_2)\beta(\nu_3) \beta(\nu_4)}{\nu_2\nu_4}
\times \\
& \times \left(  \Real \sum\limits_{l\ge 1} \frac{1}{2\pi i}
\int\limits_{2-i\infty}^{2+i\infty} D_{m_1, m_2} (s,l) \Gamma(s)
B^{-s} \Phi (s+\theta, B_1 l)ds \right),
\end{split}
\end{equation*}
where
\begin{equation*}
\begin{split}
&D_{m_1, m_2} (s,l) = \sum\limits_{n} \frac{r(n) r\left(
\frac{m_1 n+l}{m_2}\right)}{(m_1 n+l/2)^s}, \\
&\Phi(s,y) =\int\limits_{0}^{+\infty}\phi(u) u^{-s} \cos(2\pi yu)
du, \\
&B = \dfrac{4\pi \sin\delta}{\sqrt{D}Q}, \quad  B_1 = \dfrac{
\cos\delta}{\sqrt{D}Q}, \\
&q = (\nu_1 \nu_4, \nu_2 \nu_3), \quad Q= \frac{\nu_2 \nu_4}{q},
\quad m_1 = \frac{\nu_1 \nu_4}{q} , \quad m_2 = \frac{\nu_2
\nu_3}{q}.
\end{split}
\end{equation*}

Let us move the path of integration in $J_2(1, \theta)$ to the line $\Real s = 11/13+\varepsilon_0$ (which is admissible since $\Gamma(s)$ decreases exponentially when $|s| \to +\infty$, and other terms increase only polynomially for large $|s|$), passing a simple pole at $s=1$ (see corollary 1), which gives us the main contribution.
Let us estimate the integral over $\Real s = 11/13+\varepsilon_0$.
For  $y>0$ and $1< \operatorname{Re}s < 3$, using (\ref{eq13.1}) and integrating by parts twice, we have
imply
\begin{align*}
\Phi(s,y)= -\int\limits_{0}^{+\infty} \left( \frac{\phi^{''}(u)}{u^s}-2s \frac{\phi^{'}(u)}{u^{s+1}} +s (s+1) \frac{\phi (u)}{u^{s+2}}\right) \frac{\cos(2\pi  yu)}{(2\pi  y)^2}du.
\end{align*}
Therefore, we obtain an analytic continuation of $\Phi(s,y)$ to the region $-1 < \Real s <3$, and, in addition,  for
$-1/2 \le \Real s \le 2$ we get the estimate $\Phi(s,y) = O\left( \dfrac{|s|^2+1}{y^{2}} \right)$.
Hence, the integral over the line $\Real s = 11/13+\varepsilon_0$ can be estimated by means of corollary 1 as follows:
\begin{equation*}
\begin{split}
& \sum\limits_{\nu_1, \ldots, \nu_4
\le X} \frac{1}{\nu_2 \nu_4} \sum\limits_{l\ge 1} \frac{1}{B^{11/13}
B_1^2 l^2} \left( (m_1 m_2)^{-2/13} + l^{2/11} m_1^{-11/13}
m_2^{-1}\right) \\
&\ll \delta^{-11/13} \sum\limits_{\nu_1, \ldots,
\nu_4 \le X} \frac{Q^{2+11/13}}{\nu_2 \nu_4} (m_1 m_2)^{-2/13} \ll
\delta^{-11/13} \sum\limits_{\nu_1, \ldots, \nu_4
\le
X} \frac{Q^{1+11/13}}{q (m_1 m_2)^{2/13}}  \\
&\ll \delta^{-11/13} \sum\limits_{\nu_1, \ldots,
\nu_4 \le X} \frac{Q^{1+11/13}}{(\nu_1 \nu_2 \nu_3 \nu_4)^{2/13}}
\ll \delta^{-11/13} (X^{2})^{1+\frac{11}{13}}
(X^{\frac{11}{13}})^4 \\
&\ll \delta^{-11/13} X^{92/13} \ll
\frac{\delta^{-1}}{ \log^2 X},
\end{split}
\end{equation*}
since $X \le \delta^{-1/50}$.

Now we shall estimate the residue $R$ at $s=1$. For any $v \ge 0$ we have
$\Phi (1+\theta,v)  \ll \Phi (1+\theta,0)  \ll \theta^{-1}$ ($0 < \theta \le 1$) and for $v \ge 1$ that $\Phi (1+\theta, v) \ll v^{-2}$. Therefore, the integral
$$
M_{1+\theta} (w) = \int\limits_{0}^{\infty} \Phi (1+\theta,v) v^{w-1} dv
$$
is absolutely convergent when $0 < \Real w < 2$. Consequently, the theory of Mellin transforms yields
$$
\Phi(1+\theta, y) = \frac{1}{2\pi i} \int\limits_{3/2-i\infty}^{3/2+i\infty} M_{1+\theta} (w) y^{-w} dw.
$$

We then have the following expression for $R$ by means of corollary~1 with the notation of lemma~\ref{l8} and lemma~\ref{l9} :
\begin{equation*}
\begin{split}
R &= \Real  \sum\limits_{\nu_1,\nu_2, \nu_3,
\nu_4}
\frac{\beta(\nu_1)\beta(\nu_2)\beta(\nu_3)\beta(\nu_4)}{\nu_2\nu_4}
B^{-1}\sum\limits_{l\ge 1} \Phi (1+\theta, B_1 l)\mathop{\text { Res }}_{s=1}
 D_{m_1, m_2} (s,l)  \\
&= C \delta^{-1} \Real \sum\limits_{\nu_1,\nu_2, \nu_3,
\nu_4}
\frac{\beta(\nu_1)\beta(\nu_2)\beta(\nu_3)\beta(\nu_4)}{\nu_2\nu_4} \frac{Q}{m_1m_2} \sum\limits_{l\ge 1}
\sigma (l, m_1, m_2) \frac{1}{2\pi i} \int\limits_{3/2-i\infty}^{3/2+i\infty} M_{1+\theta} (w) (l B_1)^{-w} dw \\
&=  C\delta^{-1} \Real \frac{1}{2\pi i} \int\limits_{3/2-i\infty}^{3/2+i\infty} M_{1+\theta} (w) \left(\frac{ \cos \delta}{\sqrt{D}}\right)^{-w} \sum\limits_{\nu_1,\nu_2, \nu_3,
\nu_4}
\frac{\beta(\nu_1)\beta(\nu_2)\beta(\nu_3)\beta(\nu_4)}{\nu_2\nu_4} \frac{Q^{1+w}}{m_1m_2} Z_{m_1, m_2} (w) dw \\
&=  C\delta^{-1} \Real \frac{1}{2\pi i} \int\limits_{3/2-i\infty}^{3/2+i\infty} M_{1+\theta} (w) \left(\frac{ \cos \delta}{\sqrt{D}}\right)^{-w} \sum\limits_{\nu_1,\nu_2, \nu_3,
\nu_4}
\frac{\beta(\nu_1)\beta(\nu_2)\beta(\nu_3)\beta(\nu_4)}{\nu_1\nu_3} \left(\frac{q}{\nu_2\nu_4}\right)^{1-w} Z_{m_1, m_2} (w) dw
\end{split}
\end{equation*}
(where $C$ is some absolute constant). According to the partition (see lemma~\ref{l9})
$$
Z_{m_1, m_2} (w) = Z_{m_1, m_2}^{(1,1)} (w)+ Z_{m_1, m_2} ^{(2,2)}(w)+ Z_{m_1, m_2}^{(1,2)} (w)+Z_{m_1, m_2}^{(2,1)} (w),
$$
the residue $R$ can be therefore
written as the sum of four terms:
$$
R = R^{(1,1)} + R^{(2,2)} + R^{(1,2)} + R^{(2,1)}.
$$
Now we proceed with estimation of $R^{(1,2)} + R^{(2,1)}$.  By definition of $Z_{m_1, m_2}^{(l,j)} (w)$ from lemma~\ref{l9} (for $l\ne j$), the following formula holds with some real constant $C$:
\begin{equation*}
\begin{split}
R^{(l,j)}
&=  C\delta^{-1} \Real \frac{1}{2\pi i} \int\limits_{3/2-i\infty}^{3/2+i\infty} M_{1+\theta} (w) \left(\frac{ \cos \delta}{\sqrt{D}}\right)^{-w} \overline{(G(\chi_{d_j}))^2} d_j^{-1} L_{\chi_{D}} (w) L_{\chi_{D}} (1+w) S_{l,j} (w) dw,
\end{split}
\end{equation*}
where
\begin{equation*}
\begin{split}
S_{l,j} (w) = \sum\limits_{\nu_1,\nu_2, \nu_3,
\nu_4}
\frac{\beta(\nu_1)\beta(\nu_2)\beta(\nu_3)\beta(\nu_4)}{\nu_1\nu_3} \left(\frac{q}{\nu_2\nu_4}\right)^{1-w} K_{1,2} (m_l, m_j, w)
\end{split}
\end{equation*}
Since $S_{1,2} (w) = S_{2,1} (w)$ as can be seen by the change of variables $\nu_1 \leftrightarrow \nu_3$, $\nu_2 \leftrightarrow  \nu_4$, and since
$$
\frac{\overline{G^2(\chi_{d_1})}}{d_1} + \frac{\overline{G^2(\chi_{d_2})}}{d_2} = 0
$$
(which can be seen by the equality (\ref{eq88}) and the properties $G(\chi_{d_j}) = \sqrt{\pm d_j}$, $G(\chi_{D}) = i\sqrt{D}$) we get, therefore, that
$$
R^{(1,2)} + R^{(2,1)} = 0.
$$

We are left to get the estimate
\begin{equation*}
\begin{split}
R^{(1,1)}  + R^{(2,2)}\ll \frac{T}{(\theta \log T)^{1/3}}.
\end{split}
\end{equation*}
By definition of $Z_{m_1, m_2}^{(j,j)} (w)$  ($j=1,2$) of lemma \ref{l9}, the following formula holds with some real constant $C$ (which does not depend on $j$):
\begin{equation*}
\begin{split}
R^{(j,j)}
&=  C\delta^{-1} \Real \frac{1}{2\pi i} \int\limits_{3/2-i\infty}^{3/2+i\infty} M_{1+\theta} (w) \left(\frac{ \cos \delta}{\sqrt{D}}\right)^{-w}  \zeta (w) \zeta (1+w) S_{j,j} (w) d_j^{-w} \times \\
&\times\prod\limits_{p\mid d_j} (1-p^{w-1}) \prod\limits_{p\mid (D/d_j)} (1-p^{-w-1})dw,
\end{split}
\end{equation*}
where $S_{j,j} (w)$ is defined in lemma \ref{l4}.
From the theory of Mellin transforms we know that
\begin{equation*}
\begin{split}
&\int\limits_{0}^{+\infty} \cos (cv) v^{w-1} dv = c^{-w} \Gamma (w) \cos (\pi w/2),
c>0, \\
&\int\limits_{0}^{+\infty} (c+v)^{-1} v^{w-1} dv = c^{w-1} \Gamma (w) \Gamma (1-w) = \pi c^{w-1}/\sin (\pi w), \quad 0 < \Real w <1.
\end{split}
\end{equation*}
Also, for two given functions $f_1 , f_2$ denote
$$
F_j (w) = \int\limits_{0}^{+\infty} f_j (v)  v^{w-1} dv,
$$
$$
f_{\alpha} (v) = \int\limits_{0}^{+\infty} u^{-\alpha} f_1(vu) f_2(u) du.
$$
The theory of Mellin transforms yields:
$$
\int\limits_{0}^{+\infty} f_{\alpha} (v)  v^{w-1} dv = F_1 (w) F_2 (1-w-\alpha),
$$
$$
\int\limits_{0}^{+\infty} f_j(v^{-l})  v^{w-1} dv = l^{-1} F_j  (-w/l).
$$
Hence, setting $f_1(v) = \cos (2\pi v)$, $f_2(v) = \dfrac{1+ \Delta^4 }{\Delta^4  + v^{-4}}$, $\alpha = 1+\theta$ in the notation of $f(v)$ above, we obtain
\begin{equation}\label{eq888}
\begin{split}
M_{1+\theta} (w) &= \frac14 (\Delta^4 +1) (\Delta^4)^{\frac{w+\theta}{4}-1} (2\pi)^{-w} \Gamma (w) \cos (\pi w/2)  \pi \sin^{-1} (\pi (w+\theta)/4) \\
&=\frac14 (1+\Delta^{-4}) \Delta^{w+\theta} (2\pi)^{-w} \Gamma (w) \cos (\pi w/2)  \pi \sin^{-1} (\pi (w+\theta)/4).
\end{split}
\end{equation}
We can write the functional equation for the Riemann zeta-function in the form
$$
2^{1-s} \pi^{-s} \zeta(s) \Gamma(s) \cos(\pi s/2) =  \zeta(1-s).
$$
Hence,
\begin{equation*}
\begin{split}
R^{(j,j)}
&=  C\delta^{-1}  (1+\Delta^{-4}) \Delta^{\theta}\Real \frac{1}{2\pi i} \int\limits_{3/2-i\infty}^{3/2+i\infty} 
\zeta (1-w) \zeta (1+w) \sin^{-1} (\pi (w+\theta)/4) \times \\
&\times S_{j,j} (w) d_j^{-w/2} (D/d_j)^{w/2} \prod\limits_{p\mid d_j} (1-p^{w-1}) \prod\limits_{p\mid (D/d_j)} (1-p^{-w-1})dw\\
&= C\delta^{-1}  (1+\Delta^{-4}) \Delta^{\theta}\Real \frac{1}{2\pi i} \int\limits_{3/2-i\infty}^{3/2+i\infty} 
\zeta (1-w) \zeta (1+w) \sin^{-1} (\pi (w+\theta)/4) \times \\
&\times \tilde{S}_{j,j} (w) dw
\end{split}
\end{equation*}
for some absolute constant $C$ (which does not depend on $j$), where $\tilde{S}_{j,j} (w)$ is defined in the corollary 2.
We shall move the path of integration from the line $\Real s = 3/2$ to the contour $\Gamma$ which consists of the two rays $\Real w =0, |\Image w| \ge \log^{-1} X$ and the semicircle $|w| = \log^{-1} X, \Real w > 0$.

We shall show that  the contribution to the sum $R^{(1,1)} + R^{(2,2)}$
from the integrals over the two rays with $|\Image w| \ge 1$ is of desired order (we denote this contribution from $R^{(1,1)}$ and  $R^{(2,2)}$ by $R_1$). To obtain this we shall use
the functional equation between $\tilde{S}_{1,1} (w)$ and $\tilde{S}_{2,2} (w)$ from the corollary 2 and
the property $\sin (-s) = - \sin s$.
In the above formula for $R^{(2,2)}$ let us make the change of variables $w\to -w$ in the integral    and use the functional equation from the corollary 2.  We thus obtain
\begin{equation*}
\begin{split}
R_1 &=  C (\theta, \Delta) \delta^{-1} \frac{1}{2\pi} \int\limits_{1}^{+\infty}  \zeta (1-it) \zeta (1+it) (\tilde{S}_{1,1} (it) + \tilde{S}_{1,1} (-it)) \times \\
&\times \left( \frac{1}{\sin (\pi (it+\theta)/4)}  + \frac{1}{\sin (\pi (-it+\theta)/4)}  \right) dt,
\end{split}
\end{equation*}
where $C(\theta, \Delta)$ is a constant depending on $\theta$ and $\Delta$, and which is bounded for $0 \le \theta \le 1,  0 <\delta \le 1/2$.

We shall transform the sum of inverse sinuses in $R_1$ as follows
\begin{equation*}
\begin{split}
\left( \frac{1}{\sin (\pi (it+\theta)/4)}  + \frac{1}{\sin (\pi (-it+\theta)/4)}  \right)  = \frac{4\sin \frac{\pi\theta}{4} (e^{\pi t/4} + e^{-\pi t/4})}{e^{\pi t/2} + e^{-\pi t/2} -2\cos \frac{\pi\theta}{2}}
\end{split}
\end{equation*}
which is $O(\theta e^{-\frac{\pi t}{4}})$ for $t\ge 1$. Now, the estimate of $S_{1,1}$ from lemma \ref{l4} gives
\begin{equation*}
\begin{split}
R_1 \ll \delta^{-1} \theta \int\limits_{1}^{+\infty} \frac{\log^6 (t+3)}{e^{\pi t/4 }} dt \ll \delta^{-1} \theta.
\end{split}
\end{equation*}

For $w$ on the part of the contour $\Gamma$ with the property $|\Image w| \le 1$ we shall use the bound
\begin{equation}\label{eq511}
\begin{split}
&\left( \frac{1}{\sin (\pi (w+\theta)/4)}  + \frac{1}{\sin (\pi (-w+\theta)/4)}  \right)  = \frac{4\sin \frac{\pi\theta}{4} (e^{i\pi w/4} + e^{-i\pi w/4})}{2 \cos\frac{\pi w}{2} -2\cos \frac{\pi\theta}{2}} \\
&\ll \frac{\theta}{|w|^2 + \theta^2}
\end{split}
\end{equation}
due to the inequality $|\Real w| < \theta/2$
and
\begin{equation}\label{eq512}
\zeta(1+w) \ll \frac{|w|+1}{|w|}\log (|\Image w|+3).
\end{equation}
The contribution from $R^{(1,1)}$ and $R^{(2,2)}$ which comes from the integrals over the part of the contour $\Gamma$ with the property $|\Image w| \le \theta/\sqrt{\theta \log X}$ we denote by $R_2$ and estimate it from above by means of the estimate from lemma~\ref{l4} as
\begin{align*}
R_2 \ll \delta^{-1} \left( \theta^{-1} \log^{-1} X + \int\limits_{\log^{-1} X}^{\theta/{\sqrt{\theta \log X}}} \theta^{-1}  dt  \right) \ll \delta^{-1} (\theta \log X)^{-1/2}.
\end{align*}

We denote by $R_3$ the contribution from $R^{(1,1)}$ and $R^{(2,2)}$ which comes from the integrals over the part of the contour $\Gamma$ with the property $\sqrt{\theta \log X}  \theta \le |\Image w| \le 1$.
We then have by (\ref{eq511}), (\ref{eq512}) and lemma 8
\begin{align*}
R_3 \ll \delta^{-1} \int\limits_{ \theta \sqrt{\theta \log X}}^{1} \theta t^{-2}  dt \ll \delta^{-1} (\theta \log X)^{-1/2}.
\end{align*}

By $R_0$ we denote the main contribution into $R^{(1,1)}$ and $R^{(2,2)}$ which comes from the integrals over the leftover part of the contour $\Gamma$ satisfying $\theta/\sqrt{A} \le |\Image w| \le  \theta \sqrt{A}$, where $A = \theta \log X$.
We then have by corollary~2
\begin{align*}
R_0 &= C (\theta, \Delta) \delta^{-1} \frac{1}{\pi} \int\limits_{\theta/{\sqrt{A}}}^{ \theta \sqrt{A}}  \zeta (1-it) \zeta (1+it) (\tilde{S}_{1,1} (it) + \tilde{S}_{1,1} (-it)) \times \\
&\times \left( \frac{1}{\sin (\pi (it+\theta)/4)}  + \frac{1}{\sin (\pi (-it+\theta)/4)}  \right) dt \\
&= C (\theta, \Delta) \delta^{-1} \frac{1}{\pi} \Real \int\limits_{\theta/{\sqrt{A}}}^{\theta \sqrt{A}}  \zeta (1-it) \zeta (1+it) S_{1,1} (it)  d_1^{-it/2} d_2^{it/2} \prod\limits_{p\mid d_1} (1-p^{it-1}) \prod\limits_{p\mid d_2} (1-p^{-it-1})\times \\
&\times \left( \frac{1}{\sin (\pi (it+\theta)/4)}  + \frac{1}{\sin (\pi (-it+\theta)/4)}  \right) dt
\end{align*}
with some bounded constant $C (\theta, \Delta)$.
First, we shall estimate the part of $R_0$ coming from the error term of (\ref{eq510}).
By means of (\ref{eq511}), (\ref{eq512}) this part can be estimated from above up to some absolute constant  as
\begin{align*}
&\delta^{-1} \int\limits_{\theta/{\sqrt{A}}}^{\theta \sqrt{A}}  \frac{\theta}{\theta^2 + t^2}  \cdot
\frac{t^{3/2} \log^{-1/2} X}{t^2}dt \\
&\ll
\delta^{-1}
\left(
\int\limits_{\theta/{\sqrt{A}}}^{\theta}  \theta^{-1} (t\log X)^{-1/2}  dt + \int\limits_{\theta}^{\theta \sqrt{A}}  \theta t^{-2}  (t \log X)^{-1/2})  dt \right) \\
&\ll  \delta^{-1} (\theta \log  X)^{-1/2} \ll \delta^{-1} A^{-1/2} .
\end{align*}

Now, from (\ref{eq509}) and (\ref{eq510}) we get, that the main term of $R_0$ is the sum of four terms of the following type:
\begin{align}\label{eq514}
&\frac{\delta^{-1} C_1(\theta, \Delta) }{\log^4 X}
\sum_{d \le X^2} d \sum_{m|d}  \frac{\mu(m)}{m}
\sum_{\substack{\delta_1 \delta_4
\equiv0 \ (\operatorname{mod}d) \\
\delta_j | d^{\infty}}} \frac{\alpha(\delta_1)
\alpha(\delta_4)}{\delta_1 \delta_4}
\times \sum\limits_{(n,d)=1} \frac{1}{n^{2}} \sum_{\substack{k_1, k_4 | n^{\infty}, \\ (k_1, k_4)=1}} \frac{\alpha(n k_1) \alpha(n k_4)}{k_1 k_4} \times\nonumber \\
&\times\sum\limits_{(r, nd)=1}  \frac{\mu(r)}{r^{2}} \sum_{l_j | r^{\infty}} \frac{\alpha(r l_1) \alpha(r l_4)}{l_1
l_4}
\left( \log \left(\max\left( 1, X^{a}/(\delta_1 k_1 l_1 n r) \right) \right)
\log \left(\max\left( 1, X^{b}/(\delta_4 k_4 l_4 n r)\right) \right)\right)^{1/2} \times \nonumber \\
&\times\sum_{\substack{\delta'_1 \delta'_4
\equiv0 \ (\operatorname{mod}d) \\
\delta'_j | d^{\infty}}} \frac{\alpha(\delta'_1)
\alpha(\delta'_4)}{\delta'_1 \delta'_4}
\times \sum\limits_{(n',d)=1} \frac{1}{{n'}^{2}} \sum_{\substack{k'_1, k'_4 | {n'}^{\infty}, \\ (k'_1, k'_4)=1}} \frac{\alpha(n' k'_1) \alpha(n' k'_4)}{k'_1 k'_4} \sum\limits_{(r', n'd)=1}  \frac{\mu(r')}{{r'}^{2}} \nonumber \\
&\times  \sum_{l'_j | {r'}^{\infty}} \frac{\alpha(r' l'_1) \alpha(r' l'_4)}{l'_1
l'_4}
\left( \log \left(\max\left( 1, X^{a'}/(\delta'_1 k'_1 l'_1 n' r') \right) \right)
\log \left(\max\left( 1, X^{b'}/(\delta'_4 k'_4 l'_4 n' r')\right) \right)\right)^{1/2} \times\nonumber \\
&\times\Real \int\limits_{\theta/{\sqrt{A}}}^{\theta \sqrt{A}}  \biggl( \frac{d\sqrt{d_1}}{\sqrt{d_2}m \delta_1 \delta'_1 n n' r r' k_1 k'_1 l_1 l'_1} \biggr)^{-it}   F(t) dt,
\end{align}
where each of $a, a',b, b'$ can take one of the two values, namely,  $1/2$ or $1$,
\begin{align*}
&f(t) := f(t, \theta, d, m, \delta_1, \delta_4, \delta'_1, \delta'_4, n, n', r, r', k_1, k_4, k'_1 k'_4, l_1, l_4, l'_1, l'_4)\\ &=  F_0 (t) F_1 (t) \ldots F_{11} (t)
\end{align*}
 and
\begin{align*}
F_0 (t) &= \prod\limits_{p\mid d_1} (1-p^{it-1}) \prod\limits_{p\mid d_2} (1-p^{-it-1}), \\
F_1 (t) &= t^2 \zeta (1-it) \zeta (1+it), \\
F_2 (t) &= H_{it,1}^{(j,j)} (1)^4, \\
F_3 (t) &= \left( \frac{1}{\sin (\pi (it+\theta)/4)}  + \frac{1}{\sin (\pi (-it+\theta)/4)}  \right), \\
F_4 (t) &= (R^{(j,j)}_{it, n d r})^2, \\
F_5 (t) &= (R^{(j,j)}_{it, n' d r'})^2, \\
F_6 (t) &= K_{j,j}\biggl(\frac{\delta_1 \delta_4 m}{d}, it\biggr), \\
F_7 (t) &=  K_{j,j}(n^2 k_1 k_4, it), \\
F_8 (t) &=  K_{j,j}(r l_1, it) K_{j,j}( r l_4, it), \\
F_9 (t) &=  K_{j,j}\biggl(\frac{\delta'_1 \delta'_4 m}{d}, it\biggr), \\
F_{10} (t) &= K_{j,j}({n'}^2 k'_1 k'_4, it), \\
F_{11} (t) &= K_{j,j}(r' l'_1, it) K_{j,j}( r' l'_4, it).
\end{align*}
We state (and show further) that  for $0 < t <1$
\begin{align*}
F_0 (t) &\ll 1, &\; \dfrac{dF_0 (t)}{dt}  &\ll 1, \\
F_1 (t) &\ll 1, &\; \dfrac{dF_1 (t)}{dt}  &\ll 1, \\
F_2 (t) &\ll 1, &\; \dfrac{dF_2 (t)}{dt}  &\ll 1, \\
F_3 (t) &\ll \dfrac{\theta}{\theta^2 + t^2} \ll \theta^{-1}, &\; \dfrac{dF_3 (t)}{dt}  &\ll \dfrac{\theta t}{\theta^4 + t^4}, \\
F_4 (t) &\ll G_{ndr} \ll_{\varepsilon}  (nr)^{\varepsilon} G_d, &\; \dfrac{dF_4 (t)}{dt}  &\ll  G_{ndr} \log\log (ndr+2) \ll_{\varepsilon}  (nr)^{\varepsilon} G_d \log\log (d+2), \\
F_5 (t) &\ll G_{n'dr'} \ll_{\varepsilon}  (n' r')^{\varepsilon} G_d, &\; \dfrac{dF_5 (t)}{dt}  &\ll G_{n'dr'} \log\log (n'dr'+2) \ll_{\varepsilon} (n'r')^{\varepsilon} G_d \log\log (d+2), \\
F_6 (t) &\ll \tau^2 \biggl(\frac{\delta_1 \delta_4}{d}\biggr) \tau^2 (m)  \ll_{\varepsilon}  \left(\frac{\delta_1 \delta_4 m}{d}\right)^{\varepsilon}, &\; \dfrac{dF_6 (t)}{dt}  &\ll \tau^2 \biggl(\frac{\delta_1 \delta_4}{d}\biggr) \tau^2 (m) \log \frac{\delta_1 \delta_4 m}{d} \ll_{\varepsilon}  \left(\frac{\delta_1 \delta_4 m}{d}\right)^{\varepsilon}, \\
F_7 (t) &\ll  \tau^2 (n^2) \tau^2  (k_1) \tau^2  (k_4)\ll_{\varepsilon} (n k_1 k_4)^{\varepsilon},  &\; \dfrac{dF_7 (t)}{dt} &\ll \tau^2 (n^2) \tau^2  (k_1) \tau^2  (k_4) \log (n^2 k_1 k_4) \ll_{\varepsilon} (n k_1 k_4)^{\varepsilon},\\
F_8 (t) &\ll  \tau^4 (r) \tau^2 (l_1) \tau^2 (l_4) \ll_{\varepsilon} (r l_1 l_4)^{\varepsilon}, &\; \dfrac{dF_8 (t)}{dt} &\ll \tau^4 (r) \tau^2 (l_1) \tau^2 (l_4) \log (r l_1) \log (r l_4) \ll_{\varepsilon} (r l_1 l_4)^{\varepsilon}
\end{align*}
where $\varepsilon>0$ is as small as we wish (also similar estimates hold true for other $F_j$ ($j=9, 10, 11$) and their first derivatives). Indeed, the estimate for $F_0$ is evident, since $d_1, d_2$ are absolute constants. The estimates for $F_1$ hold true since $z^2 \zeta (1-z) \zeta (1+z)$ is an analytic function for $|z| \le 1$. The estimate for $F_2$ follows from (\ref{eq1003}). Also, we can write by (\ref{eq1003}) the equality
$$
H_{it,1}^{(j,j)} (1) = \prod\limits_{(p, D)= 1}  \left(1+ \frac{B_j (p, it)}{p^2}\right) \prod\limits_{p\mid D} \left(1+ \frac{B_j (p, it)}{p}\right),
$$\
where, for every $p$ and $0<t<1$,
\begin{align*}
B_j (p, it) \ll 1, \\
\frac{d}{dt} B_j (p, it) \ll \log p.
\end{align*}
Thus, the estimate for $\frac{d}{dz} F_2 (t)$ is clear. For $F_4, F_5$ the above estimates follow from (\ref{eq1004}), lemma~\ref{l6} and the bound
$$
\sum\limits_{p\mid N} \frac{\log p}{p} \ll \log\log (N+2).
$$
The estimates for $F_j$, $j= 6, \ldots 11$ are consequences of lemma \ref{l6} and the property $\tau(xy) \le\tau (x) \tau (y)$ (for any $x, y \in \bbN$). We are left to estimate $\dfrac{dF_3 (t)}{dt}$, which we manage to do by Taylor series expansion for $\sin z$ and $\cos z$:
\begin{align*}
\dfrac{dF_3 (t)}{dt} &= -i \frac{\pi}{4}\frac{\cos (\pi (it+\theta)/4)}{\sin^2 (\pi (it+\theta)/4)}  + i \frac{ \pi}{4}\frac{\cos (\pi (-it+\theta)/4)}{\sin^2 (\pi (-it+\theta)/4)} \\
&\ll \frac{|(it+\theta)^2 - (it-\theta)^2| + O(|it+\theta|^4)}{|it+\theta|^2\cdot |it-\theta|^2} \ll \dfrac{\theta t}{\theta^4 + t^4}
\end{align*}
for $0 < \theta \le 1$, $0 \le t \le 1$. Or, we can get just the same result using the series expansion at $z=0$ for meromorphic in the region $|z| \le 2$ function.

Now, if
\begin{align}\label{eq515}
\left| \log \frac{d\sqrt{d_1}}{\sqrt{d_2} m \delta_1 \delta'_1 n n' r r' k_1 k'_1 l_1 l'_1} \right| \ge \log (d+2)^{A^{-2/3}},
\end{align}
we shall take the integral in (\ref{eq514}) once by parts and get
\begin{align*}
&\int\limits_{\theta/{\sqrt{A}}}^{\theta \sqrt{A}} \biggl( \frac{d \sqrt{d_1}}{\sqrt{d_2} m \delta_1 \delta'_1 n n' r r' k_1 k'_1 l_1 l'_1} \biggr)^{-it}   F(t) dt = \biggl( -i \log \frac{d \sqrt{d_1}}{\sqrt{d_2} m \delta_1 \delta'_1 n n' r r' k_1 k'_1 l_1 l'_1}\biggr)^{-1} \times \\
&\times \biggl(  \left.\biggl( \frac{d \sqrt{d_1}}{\sqrt{d_2} m \delta_1 \delta'_1 n n' r r' k_1 k'_1 l_1 l'_1} \biggr)^{-it}   F(t) \right|_{\theta/{\sqrt{A}}}^{\theta \sqrt{A}} -
 \int\limits_{\theta/{\sqrt{A}}}^{\theta \sqrt{A}}  \biggl( \frac{d \sqrt{d_1}}{\sqrt{d_2} m \delta_1 \delta'_1 n n' r r' k_1 k'_1 l_1 l'_1} \biggr)^{-it}   dF(t)  \biggr) \\
&\ll (\log (d+2)^{A^{-2/3}})^{-1} \biggl( \max\limits_{t \in [0, 1]} |F(t)| + \int\limits_{\theta/{\sqrt{A}}}^{\theta \sqrt{A}}  \left| dF(t)\right|  \biggr).
\end{align*}
Therefore, the contribution into (\ref{eq514}) of the terms satisfying (\ref{eq515}) can be estimated from above by means of the above estimates for $F_j$ and their derivatives as
\begin{align}\label{eq516}
&\frac{\delta^{-1} C_1(\theta, \Delta) }{\log^2 X}
\sum_{d \le X^2} \frac{A^{2/3}d^{1-2\varepsilon} G^2_d}{\log (d+2)} \sum_{m|d}  \frac{\mu^2(m)}{m^{1-2\varepsilon}}  \left( \theta^{-1} +\int\limits_{\theta/{\sqrt{A}}}^{1}  \frac{\theta t}{\theta^4 + t^4} dt  + \log\log (d+2) \int\limits_{\theta/{\sqrt{A}}}^{1}  \frac{\theta}{\theta^2 + t^2} dt\right) \times \\
&\times \left( \sum_{\substack{\delta_1 \delta_4
\equiv0 \ (\operatorname{mod}d) \\
\delta_j | d^{\infty}}} \frac{|\alpha(\delta_1)
\alpha(\delta_4)|}{(\delta_1 \delta_4)^{1-\varepsilon}}
\times \sum\limits_{(n,d)=1} \frac{1}{n^{2-2\varepsilon}} \sum_{\substack{k_1, k_4 | n^{\infty}, \\ (k_1, k_4)=1}} \frac{1}{(k_1 k_4)^{1-\varepsilon}}
\sum\limits_{(r, nd)=1}  \frac{1}{r^{2-2\varepsilon}} \sum_{l_j | r^{\infty}} \frac{1}{(l_1
l_4)^{1-\varepsilon}}  \right)^2 \nonumber,
\end{align}
The sum over $n, k_1, k_4$ (as well as over $r, l_1, l_4$) is estimated from above by the following value:
\begin{align*}
\sum\limits_{n=1}^{+\infty} \frac{1}{n^{2-2\varepsilon}} \left(\sum_{k | n^{\infty}} \frac{1}{k^{1-\varepsilon}}\right)^2 \ll \prod_{p} \left( 1 + \frac{1}{p^{2(1-\varepsilon)}} \left(1-\frac{1}{p^{1-\varepsilon}} \right)^{-2}+ \frac{1}{p^{4(1- \varepsilon)}} \left(1-\frac{1}{p^{1-\varepsilon}} \right)^{-2} + \ldots \right) \ll 1.
\end{align*}
We rewrite the sum over $\delta_1, \delta_4$ as follows:
\begin{align*}
&\sum_{\substack{\delta_1 \delta_4
\equiv0 \ (\operatorname{mod}d) \\
\delta_j | d^{\infty}}} \frac{|\alpha(\delta_1)
\alpha(\delta_4)|}{(\delta_1 \delta_4)^{1-\varepsilon}} = d^{-1+\varepsilon} \sum\limits_{n\mid d^{\infty}} n^{-1+\varepsilon} \sum_{\delta_1 \delta_4 = nd}  |\alpha(\delta_1)
\alpha(\delta_4)| = d^{-1+\varepsilon} \sum\limits_{n\mid d^{\infty}} n^{-1+\varepsilon}  b(nd).
\end{align*}
Also
\begin{align*}
&\int\limits_{\theta/{\sqrt{A}}}^{1}  \frac{\theta}{\theta^2 + t^2} dt \ll 1,\\
&\int\limits_{\theta/{\sqrt{A}}}^{1}  \frac{\theta t}{\theta^4 + t^4} dt  \ll \theta^{-1}.
\end{align*}
Therefore, (\ref{eq516})
is not of the greater order than
\begin{align}\label{eq58}
&\frac{\delta^{-1}}{\log^2 X}
\sum_{d \le X^2} \frac{A^{2/3} G_d^2}{d\log (d+2)} \sum_{m|d}  \frac{\mu^2 (m)}{m^{1-2\varepsilon}}  \left( \theta^{-1}  + \log\log (d+2) \right) \left( \sum\limits_{n\mid d^{\infty}} n^{-1+\varepsilon} b(nd)
 \right)^2.
\end{align}
Since (as in proof of lemma~\ref{l4})
\begin{align*}
&\sum_{d \le X^2} \frac{G_d^2}{d} \sum_{m|d}  \frac{\mu^2(m)}{m^{1-2\varepsilon}}
 \left( \sum\limits_{n\mid d^{\infty}} n^{-1+\varepsilon} b(nd)
 \right)^2 \ll \sum_{d \le X^2} \frac{1}{d}  \prod_{p^{\beta} || d} \left( 1 + \frac{1}{p^{1-2\varepsilon}}\right)  \left( B(p^{\beta}) + \frac{B(p^{\beta+1}) C}{p^{1-\varepsilon}}\right)^2 \left( 1 +\frac{1}{p^{3/4}} \right)^4\\
&\ll \sum_{d \le X^2} \frac{r^2 (d)}{d} \ll \log^2 X,
\end{align*}
then by Abel's partial summation formula we get the estimate
\begin{align*}
&\sum_{d \le X^2} \frac{G_d^2}{d \log (d+2)} \sum_{m|d}  \frac{\mu^2(m)}{m^{1-2\varepsilon}}
 \left( \sum\limits_{n\mid d^{\infty}} n^{-1+\varepsilon} b(nd)
 \right)^2 \ll \log X.
\end{align*}
Finally, the contribution of terms satisfying  (\ref{eq515})  into (\ref{eq514}), estimated from above
by the expression (\ref{eq58}), is less than
\begin{align*}
&\frac{\delta^{-1}}{\log^2 X} A^{2/3} \log X
\left( \theta^{-1} + \log\log X \right) \ll \frac{\delta^{-1} A^{2/3} }{\theta \log X} \ll \delta^{-1} A^{-1/3} .
\end{align*}

Now we are left to evaluate the contribution into (\ref{eq514}) from the terms satisfying
\begin{align}\label{eq517}
\left| \log \frac{d \sqrt{d_1}}{\sqrt{d_2} m \delta_1 \delta'_1 n n' r r' k_1 k'_1 l_1 l'_1} \right| <\log (d+2)^{A^{-2/3}}.
\end{align}
Denote this condition by $\bf{(C)}$ and by $\bf{1}_{(C)}$ its characteristic function.
We shall estimate the integral in (\ref{eq514}) trivially (by a constant) in this case and
get the following contribution from above into (\ref{eq514}) of the terms satisfying (\ref{eq517}):
\begin{align}\label{eq518}
&\frac{\delta^{-1} C_1(\theta, \Delta) }{\log^2 X}
\sum_{d \le X^2} d G_d^2 \sum_{m|d}  \frac{\mu^2 (m)}{m}
\sum_{\substack{\delta_1 \delta_4
\equiv0 \ (\operatorname{mod}d) \\
\delta_j | d^{\infty}}} \frac{|\alpha(\delta_1)
\alpha(\delta_4)|}{\delta_1 \delta_4}
\tau^2 \biggl(\frac{\delta_1 \delta_4}{d}\biggr) \tau^2 (m) \times \nonumber\\
&\times \sum\limits_{(n,d)=1} \frac{G_n}{n^{2}} \sum_{\substack{k_1, k_4 | n^{\infty}, \\ (k_1, k_4)=1}} \frac{1}{k_1 k_4} \tau^2 (n^2) \tau^2  (k_1) \tau^2  (k_4)
\sum\limits_{(r, nd)=1}  \frac{G_r}{r^{2}}  \sum_{l_j | r^{\infty}} \frac{1}{l_1
l_4}  \tau^4 (r) \tau^2 (l_1) \tau^2 (l_4) \times \nonumber\\
&\times\sum_{\substack{\delta'_1 \delta'_4
\equiv0 \ (\operatorname{mod}d) \\
\delta'_j | d^{\infty}}} \frac{|\alpha(\delta'_1)
\alpha(\delta'_4)|}{\delta'_1 \delta'_4} \tau^2 \biggl(\frac{\delta'_1 \delta'_4}{d}\biggr) \tau^2 (m)
\sum\limits_{(n',d)=1} \frac{G_{n'}}{{n'}^{2}} \sum_{\substack{k'_1, k'_4 | {n'}^{\infty}, \\ (k'_1, k'_4)=1}} \frac{1}{k'_1 k'_4} \tau^2 (n'^2) \tau^2  (k'_1) \tau^2  (k'_4) \times\nonumber \\
&\times \sum\limits_{(r', n'd)=1}  \frac{G_{r'} }{{r'}^{2}} \sum_{l'_j | {r'}^{\infty}} \frac{1}{l'_1
l'_4}  \tau^4 (r') \tau^2 (l'_1) \tau^2 (l'_4) \bf{1}_{(C)}.
\end{align}
If variables satisfy (\ref{eq517}), then either
\begin{align}\label{eq64}
\max (m, n, n', r, r',  k_1,  k'_1, l_1, l'_1) \ge d^{1/20}
\end{align}
 or
\begin{equation}\label{eq519}
\begin{cases}
\max (m, n, n', r, r',  k_1,  k'_1, l_1, l'_1) \le d^{1/20}, \\
\frac{d (d+2)^{-A^{-2/3}} \sqrt{d_1}}{\sqrt{d_2}m n n' r r' k_1 k'_1 l_1 l'_1} \le \delta_1 \delta'_1 \le
\frac{d (d+2)^{A^{-2/3}}\sqrt{d_1}}{\sqrt{d_2} m n n' r r' k_1 k'_1 l_1 l'_1}.
\end{cases}
\end{equation}
We shall, first, estimate from above the contribution into (\ref{eq518}) of the terms staifying (\ref{eq64}).
Now suppose, for example, that $k_1 \ge d^{1/20}$. Then the following inequalities hold:
\begin{align*}
&\sum_{\substack{k_1 | n^{\infty}, \\ k_1 \ge d^{1/20}}} \frac{ \tau^2  (k_1) }{k_1} \le d^{-1/40} \sum_{k_1 | n^{\infty}} \frac{ \tau^2  (k_1) }{\sqrt{k_1}} = d^{-1/40} \prod\limits_{p|n} \left( 1 + \frac{\tau^2(p)}{\sqrt{p}} + \frac{\tau^2(p^2)}{p}  +  \frac{\tau^2(p^3)}{\sqrt{p^3}}  +\ldots \right) \\
&\ll d^{-1/40} \prod\limits_{p|n} \left( 1 +  \frac{\tau^2(p)}{\sqrt{p}} + \frac{\tau^2(p^2)}{p} \right)
\ll d^{-1/40} \prod\limits_{p|n} \left( 1 + \frac{1}{p^{1/4}}\right) \le d^{-1/40}  \sum_{k | n} \frac{ 1 }{k^{1/4}}
\end{align*}
and, hence,
\begin{align*}
&\sum\limits_{(n,d)=1} \frac{G_n}{n^{2}} \sum_{\substack{k_1, k_4 | n^{\infty}, \\ (k_1, k_4)=1, \\ k_1 \ge d^{1/20}}} \frac{1}{k_1 k_4} \tau^2 (n^2) \tau^2  (k_1) \tau^2  (k_4)  \le d^{-1/40} \sum\limits_{(n,d)=1} \frac{\tau^4 (n^2)}{n^{2}} \prod\limits_{p|n} \left( 1 + \frac{1}{p}\right)^4 \sum_{k | n} \frac{ 1 }{k^{1/4}} \\
&\ll  d^{-1/40} \sum_{k = 1}^{+\infty} \frac{ 1 }{k^{9/4}} \sum\limits_{n=1}^{+\infty}  \frac{\tau^4 (n^2 k^2) \tau^4 (n k)}{n^{2}} \ll d^{-1/40} \sum_{k = 1}^{+\infty} \frac{ \tau^4  (k)\tau^4 (k^2) }{k^{9/4}} \sum\limits_{n=1}^{+\infty} \frac{\tau^4 (n^2) \tau^4 (n)}{n^{2}} \ll d^{-1/40}.
\end{align*}
So, the subsum in (\ref{eq518}) satisfying the condition $k_1 \ge d^{1/20}$ totally amounts (up to some constant) $\delta^{-1} \log^{-2} X$ from above.
Almost the same calculations can be performed to show that all the terms satisfying (\ref{eq64}) contribute the same.

The last sum we are left to consider is the subsum of (\ref{eq518}) with summation parameters satisfying (\ref{eq519}).This part can be estimated from above by (using $G_n \ll  \tau^2 (n)$)
 \begin{align}\label{eq520}
&\sum\limits_{n=1}^{\infty}\frac{ \tau^2 (n)  \tau^2 (n^2)}{n^{2}} \sum_{k_1, k_4 | n^{\infty}} \frac{\tau^2  (k_1) \tau^2  (k_4)}{k_1 k_4}
\sum\limits_{r=1}^{\infty}  \frac{\tau^6 (r)}{r^{2}}  \sum_{l_j | r^{\infty}} \frac{\tau^2 (l_1) \tau^2 (l_4)}{l_1
l_4}    \times \nonumber\\
&\sum\limits_{n'=1}^{\infty}\frac{ \tau^2 (n')  \tau^2 (n'^2)}{n'^{2}} \sum_{k'_1, k'_4 | n'^{\infty}} \frac{\tau^2  (k'_1) \tau^2  (k'_4)}{k'_1 k'_4}
\sum\limits_{r'=1}^{\infty}  \frac{\tau^6 (r')}{r'^{2}}  \sum_{l'_j | r'^{\infty}} \frac{\tau^2 (l'_1) \tau^2 (l'_4)}{l'_1
l'_4}    \times \nonumber\\
&\frac{\delta^{-1} C_1(\theta) }{\log^2 X}
\sum_{d \le X^2} d G_d^2 \sum_{m|d}  \frac{\mu^2 (m) \tau^4 (m)}{m}
\sum_{\substack{\delta_1 \delta_4
\equiv0 \ (\operatorname{mod}d) \\
\delta_j | d^{\infty}, \delta_j \le X}} \frac{|\alpha(\delta_1)
\alpha(\delta_4)|}{\delta_1 \delta_4}
\tau^2 \biggl(\frac{\delta_1 \delta_4}{d}\biggr)  \times \nonumber\\
&\times\sum_{\substack{\delta'_1 \delta'_4
\equiv0 \ (\operatorname{mod}d) \\
\delta'_j | d^{\infty}, \delta'_j \le X, \\ \frac{d (d+2)^{- A^{-2/3}} \sqrt{d_1}}{\sqrt{d_2}\delta_1 m n n' r r' k_1 k'_1 l_1 l'_1} \le  \delta'_1 \le
\frac{d (d+2)^{A^{-2/3}}\sqrt{d_1}}{\sqrt{d_2}\delta_1 m n n' r r' k_1 k'_1 l_1 l'_1}}}
\frac{|\alpha(\delta'_1)
\alpha(\delta'_4)|}{\delta'_1 \delta'_4} \tau^2 \biggl(\frac{\delta'_1 \delta'_4}{d}\biggr).
\end{align}
For fixed $n, n', r, r', k_1, k'_1, l_1, l'_1$ define $N =  \sqrt{d_2} n n' r r' k_1 k'_1 l_1 l'_1/\sqrt{d_1}$ and notice that since all the summands in (\ref{eq520}) are positive, we can relax the limits of summation in the last sum by the following
\begin{align*}
(3X)^{-2 A^{-2/3}} \le \frac{m N\delta_1\delta'_1}{d} \le (3X)^{2A^{-2/3}}
\end{align*}
(recalling that $d +2 \le X^2+2 \le (3 X)^2$). Therefore, if we estimate from above the sum
\begin{align}\label{eq521}
W (X, N) &= \sum_{d \le X^2} d G_d^2 \sum_{m|d}  \frac{\mu^2 (m) \tau^4 (m)}{m}
\sum_{\substack{\delta_1 \delta_4
\equiv0 \ (\operatorname{mod}d) \\
\delta_j | d^{\infty}, \; \delta_j \le X}} \frac{|\alpha(\delta_1)
\alpha(\delta_4)|}{\delta_1 \delta_4}
\tau^2 \biggl(\frac{\delta_1 \delta_4}{d}\biggr)  \times \nonumber\\
&\sum_{\substack{\delta'_1 \delta'_4
\equiv0 \ (\operatorname{mod}d) \\
\delta'_j | d^{\infty}, \; \delta'_j \le X, \\ (3X)^{-2A^{-2/3}} \le \frac{m \delta_1\delta'_1 N}{d} \le (3X)^{2A^{-2/3}}}}
\frac{|\alpha(\delta'_1)
\alpha(\delta'_4)|}{\delta'_1 \delta'_4} \tau^2 \biggl(\frac{\delta'_1 \delta'_4}{d}\biggr)
\end{align}
by $A^{-1/3} \log^2 X$ for all $N$,
then the expression in (\ref{eq520}) is bounded from above by $\delta^{-1} A^{-1/3}$ which completes the proof of the estimate $R_0 \ll \delta^{-1} A^{-1/3}$.
Let us proceed with estimation of $W (X, N)$. Transform its summation as follows:
\begin{align*}
W (X, N) &=  \sum\limits_{\delta_1, \delta'_1 \le X} |\alpha(\delta_1)
\alpha(\delta'_1)| \sum_{m\le X^2}  \frac{\mu^2 (m) \tau^4 (m)}{m}  \sum_{\substack{d \le X^2, m \mid d\\ (3X)^{-2A^{-2/3}} \le \frac{d}{\delta_1\delta'_1 m N} \le (3X)^{2A^{-2/3}}, \\ \delta_1, \delta'_1  \mid d^{\infty}}}  \frac{G_d^2}{d} \times \\
&\times \sum\limits_{\substack{ n, n' \mid d^{\infty}, \\ nd /\delta_1 \le X, n'd /\delta'_1 \le X, \\ n d \equiv 0 (\mod \delta_1), \\ n' d \equiv 0 (\mod \delta'_1)}} \frac{|\alpha(nd /\delta_1)
\alpha(n'd /\delta'_1)|}{n n'} \tau^2 (n) \tau^2 (n')\\
&\le \sum\limits_{\delta_1, \delta'_1 \le X} |\alpha(\delta_1)
\alpha(\delta'_1)|
\sum_{m\le X^2}  \frac{\mu^2 (m) \tau^4 (m)}{m^2}
\sum_{\substack{d \le X^2/m, \\ (3X)^{-2 A^{-2/3} } \le \frac{d}{\delta_1\delta'_1 N} \le (3X)^{2 A^{-2/3}}, \\ \delta_1, \delta'_1  \mid (md)^{\infty}}}  \frac{G_{md}^2}{d} \times \\
&\times \sum\limits_{\substack{n, n' \mid (dm)^{\infty}, \\ nmd/\delta_1  \le X, n'md/\delta'_1 \le X, \\  n md \equiv 0 (\mod \delta_1), \\ n' md \equiv 0 (\mod \delta'_1)}} \frac{|\alpha(nmd /\delta_1)
\alpha(n'md /\delta'_1)|}{n n'} \tau^2 (n) \tau^2 (n').
\end{align*}
For fixed $\delta_1, \delta'_1$, in the last expression we divide the inner sum over $m, d,\ldots$ into $\tau(\delta_1) \tau(\delta'_1 )$ subsums according to the condition
\begin{align*}
&(md, \delta_1) = \delta, \\
&(md, \delta'_1) = \delta',
\end{align*}
where $\delta\mid \delta_1$ and $\delta'\mid \delta'_1$.
We then have
\begin{align*}
&md= \frac{\delta \delta'}{(\delta, \delta')} k, \\
&\left( \frac{md}{\delta}, \frac{\delta_1}{\delta} \right) = 1, \left( \frac{md}{\delta'}, \frac{\delta'_1}{\delta'} \right) = 1, \\
&\frac{\delta_1}{\delta} \mid \delta^{\infty}, \frac{\delta'_1}{\delta'} \mid \delta'^{\infty}.
\end{align*}
Since $n md \equiv 0 (\mod \delta_1)$ and $n \mid (dm)^{\infty}$, we have $n  = \frac{\delta_1}{\delta}  l$ (where  $l \mid (\delta \delta' k)^{\infty}$) and
$$
n md  = \delta \frac{md}{\delta}  \frac{\delta_1}{\delta}  l = \delta_1 \frac{md}{\delta}  l = \delta_1 \frac{\delta'}{(\delta, \delta')} k l.
$$
 And similarly, since $n' md \equiv 0 (\mod \delta'_1)$, we have $n'  = \frac{\delta'_1}{\delta'}  l'$ (where  $l' \mid (\delta \delta' k)^{\infty}$) and
 $$
n' md  = \delta \frac{md}{\delta'}  \frac{\delta'_1}{\delta'}  l' = \delta'_1 \frac{md}{\delta'}  l' = \delta'_1 \frac{\delta}{(\delta, \delta')} k l'.
$$
Therefore,
$$
m^2 d n n' = (mnd)(mn' d)/d =  \frac{\delta_1 \delta'  k  l}{(\delta, \delta')} \frac{\delta'_1  \delta  k l' }{(\delta, \delta')}\frac{m (\delta, \delta')}{k \delta \delta'} = \frac{\delta_1 \delta'_1  k  l l' m}{(\delta, \delta')}
$$
and
\begin{align*}
W (X, N) &\le  \sum\limits_{\delta_1, \delta'_1 \le X} \frac{|\alpha(\delta_1)
\alpha(\delta'_1)|}{\delta_1 \delta'_1}
\sum\limits_{\substack{\delta \mid \delta_1, \delta' \mid  \delta'_1, \\ \frac{\delta_1}{\delta} \mid \delta^{\infty}, \frac{\delta'_1}{\delta'} \mid \delta'^{\infty}}} (\delta, \delta')
\sum\limits_{\substack{k\le X^2 \frac{(\delta, \delta')}{\delta \delta'}, \\ (3X)^{-2A^{-2/3}} \le \frac{k \delta \delta'}{\delta_1\delta'_1 (\delta, \delta') N} \le (3X)^{2A^{-2/3}}, \\ \left(\frac{\delta'}{(\delta, \delta')} k, \frac{\delta_1}{\delta}\right) = 1, \left(\frac{\delta}{(\delta, \delta')} k, \frac{\delta'_1}{\delta'}\right) = 1
}} \frac{G_{\frac{\delta \delta'}{(\delta, \delta')} k}^2}{k} \times \\
&\times \sum\limits_{\substack{m \mid k \frac{\delta \delta'}{(\delta, \delta')} }} \frac{\mu^2 (m) \tau^4 (m)}{m}
\sum\limits_{\substack{l\le X \frac{(\delta, \delta')}{k \delta'}, \\ l' \le X \frac{(\delta, \delta')}{k \delta},\\ l, l' \mid (\delta \delta' k)^{\infty}}} \frac{|\alpha(\delta k l' /(\delta, \delta'))
\alpha(\delta' k l /(\delta, \delta'))|}{l l'} \\
&\le  \sum\limits_{\delta, \delta' \le X} \frac{(\delta, \delta')}{\delta \delta'}
\sum\limits_{\substack{\rho_1 \le X/\delta, \rho'_1 \le X/\delta', \\ \rho_1 \mid \delta^{\infty}, \rho'_1  \mid \delta'^{\infty}}} \frac{|\alpha(\delta \rho_1)
\alpha(\delta'\rho'_1)|}{\rho_1 \rho'_1}
\prod\limits_{p\mid \frac{\delta \delta'}{(\delta, \delta')}}
\left( 1+ \frac{\tau^4 (p)}{p} \right)  \times \\
&\times
\sum\limits_{\substack{k\le X^2 \frac{(\delta, \delta')}{\delta \delta'},\\ \left(\frac{\delta'}{(\delta, \delta')} k, \rho_1\right) = 1, \left(\frac{\delta}{(\delta, \delta')} k, \rho'_1\right) = 1, \\
 (3X)^{-2A^{-2/3}} \le \frac{k}{\rho_1\rho'_1 (\delta, \delta') N} \le (3X)^{2A^{-2/3}}
}} \frac{G_{\frac{\delta \delta'}{(\delta, \delta')} k}^2}{k}
\sum\limits_{\substack{m \mid k, \\ (m, \delta\delta') = 1}}
\frac{\mu^2 (m) \tau^4 (m)}{m} \times \\
&\times
 \sum\limits_{\substack{l\le X \frac{(\delta, \delta')}{k \delta'}, \\ l' \le X \frac{(\delta, \delta')}{k  \delta},\\ l, l' \mid (\delta \delta'  k)^{\infty}}} \frac{|\alpha(\delta k l' /(\delta, \delta'))
\alpha(\delta' k l /(\delta, \delta'))|}{l l'}.
\end{align*}
We shall use once again the M\"obius summation formula
$$
f(q) = \sum\limits_{s \mid q} \sum\limits_{t \mid s} \mu(t) f(s/t),
$$
with $q = (\delta, \delta')$ and the inequality
$$
\sum\limits_{m \mid k}
\frac{\mu^2 (m) \tau^4 (m)}{m} = \prod\limits_{p\mid k}
\left( 1+ \frac{\tau^4 (p)}{p} \right)  \le G_{k}^{8}
$$
to get
\begin{align}\label{eq523}
W (X, N) &\le  \sum\limits_{\delta, \delta' \le X} \frac{(G_{\delta \delta'})^{10}}{\delta \delta'} \sum\limits_{s\mid (\delta, \delta')}
\sum\limits_{t\mid s} \mu(t) \frac{s}{t}
\sum\limits_{\substack{\rho_1 \le X/\delta, \rho'_1 \le X/\delta', \\ \rho_1 \mid \delta^{\infty}, \rho'_1  \mid \delta'^{\infty}}} \frac{|\alpha(\delta \rho_1)
\alpha(\delta'\rho'_1)|}{\rho_1 \rho'_1}  \times \nonumber \\
&\times \sum\limits_{\substack{k\le X^2 \frac{s}{t \delta \delta'},\\ \left(\frac{\delta' t}{s} k, \rho_1\right) = 1, \left(\frac{\delta t}{s}  k, \rho'_1\right) = 1, \\
 (3X)^{-2A^{-2/3}} \le \frac{k t}{\rho_1\rho'_1 s N} \le (3X)(3^{2A^{-2/3}}
}} \frac{G_{k}^{10}}{k}
\sum\limits_{\substack{l\le X \frac{s}{k t \delta'}, \\ l' \le X \frac{s}{k t  \delta},\\ l, l' \mid (\delta \delta' k)^{\infty}}} \frac{|\alpha(\delta  k l' t /s)
\alpha(\delta'  k l t/s)|}{l l'} \nonumber\\
&\le \sum\limits_{s\le X} \sum\limits_{t\mid s} \mu^2 (t) \frac{s}{t} \sum\limits_{\substack{j, j' \le X/s, \\ j, j' \mid s^{\infty}}} \frac{(G_{s})^{10}}{s^2 j j'}
\sum\limits_{\substack{h \le X/(s j), h' \le X/(s j'), \\ (h h', s) = 1}} \frac{(G_{h h'})^{10}}{h h'} \times \nonumber\\
&\times \sum\limits_{\substack{\rho \le X/sj, \rho' \le X/sj', \\ \rho \mid s^{\infty}, \rho'  \mid s^{\infty}}} \frac{|\alpha(sj \rho)
\alpha(s j' \rho')|}{\rho \rho'}
\sum\limits_{\substack{\rho_1 \le X/(sjh \rho), \rho'_1 \le X/(sj'h' \rho'), \\ \rho_1 \mid h^{\infty}, \rho'_1  \mid h'^{\infty}}} \frac{|\alpha(h \rho_1)
\alpha(h'\rho'_1)|}{\rho_1 \rho'_1} \times \nonumber\\
&\times
\sum\limits_{\substack{k\le X^2 \frac{1}{t sjj' hh'},\\ \left( j' t h' k, \rho\rho_1\right) = 1, \left(j t h k, \rho'\rho'_1\right) = 1, \\
 (3X)^{-2A^{-2/3}} \le \frac{k t}{\rho_1\rho'_1 \rho \rho' s N} \le (3X)^{2A^{-2/3}}
}}
\frac{G_{ k }^{10}}{k}  \sum\limits_{\substack{l\le X \frac{1}{k t j h}, \\ l' \le X \frac{1}{k t j' h'},\\ l, l' \mid ( s hh'  k)^{\infty}}} \frac{|\alpha(j t h  k l')
\alpha(j' t h'  k l)|}{l l'}.
\end{align}
Similarly, we represent the sum over $k$ (as well as over $l, l'$) in (\ref{eq523}) as the sum over $k_s, k_{h h'}, k$, where $k_s \mid s^{\infty}, k_{h h'} \mid (h h')^{\infty}, (k, sh h')=1$. So, we get that the sums over $k, l, l'$ in  (\ref{eq523}) equal to
\begin{align*}
&\sum\limits_{\substack{k_s \mid s^{\infty}, \\ k_s\le X^2 \frac{1}{t sjj' hh'},\\ \left( j' t k_s, \rho\right) = 1, \left(j t  k_s, \rho'\right) = 1
}}
\frac{G_{ k_s }^{10}}{k_s}
\sum\limits_{\substack{k_{h h'} \mid (h h')^{\infty}, \\ k_{h h'} \le X^2 \frac{1}{t sjj' k_s hh' },\\ \left( h' k_{h h'}, \rho_1\right) = 1, \left(h k_{h h'}, \rho'_1\right) = 1
}}
\frac{G_{ k_{hh'} }^{10}}{k_{hh'}}
\sum\limits_{\substack{k\le X^2 \frac{1}{t sjj' hh' k_s k_{h h'}}, \\(k, shh') = 1, \\
 (3X)^{-2A^{-2/3}} \le \frac{k k_s k_{h h'} t}{\rho_1\rho'_1 \rho \rho' s N} \le (3X)^{2A^{-2/3}}
}}
\frac{G_{ k }^{10}}{k} \times \\
&\times
\sum\limits_{\substack{
l_s, l'_s \mid s ^{\infty}
}}
\sum\limits_{\substack{
l_{hh'}, l'_{hh'} \mid (hh')^{\infty}
}}
\sum\limits_{\substack{l\le X \frac{1}{l_s l_{hh'} k k_s k_{hh'} t j h}, \\ l' \le X \frac{1}{l'_s l'_{hh'} k k_s k_{hh'} t j' h'},\\ l, l' \mid k^{\infty}}}
\frac{|\alpha(j t  k_s l'_s)  \alpha(j' t  k_s  l_s) \alpha(h  k_{hh'} l'_{hh'}) \alpha(h'  k_{hh'} l_{hh'}) \alpha(k l') \alpha(k l)|}{l_s l'_s l_{hh'} l'_{hh'} l l'}.
\end{align*}
Thus,
\begin{align}\label{eq70}
&W (X, N) \le  \sum\limits_{s\le X} \frac{G_{s}^{10}}{s} \sum\limits_{t\mid s} \frac{\mu^2 (t)}{t} \sum\limits_{\substack{j, j' \mid s^{\infty}}} \frac{1}{j j'} \sum\limits_{\substack{\rho \mid s^{\infty}, \rho'  \mid s^{\infty}}} \frac{|\alpha(sj \rho)
\alpha(s j' \rho')|}{\rho \rho'} \sum\limits_{\substack{k_s \mid s^{\infty}
}}
\frac{G_{ k_s }^{10}}{k_s}  \sum\limits_{\substack{
l_s, l'_s \mid s ^{\infty}
}} \frac{|\alpha(j t  k_s l'_s)  \alpha(j' t  k_s  l_s) }{l_s l'_s} \times  \nonumber \\
&\sum\limits_{\substack{h \le X, h' \le X}} \frac{(G_{h h'})^{10}}{h h'}
\sum\limits_{\substack{ \rho_1 \mid h^{\infty}, \rho'_1  \mid h'^{\infty}}} \frac{|\alpha(h \rho_1)
\alpha(h'\rho'_1)|}{\rho_1 \rho'_1}
\sum\limits_{\substack{k_{h h'} \mid (h h')^{\infty}
}}
\frac{G_{ k_{hh'} }^{10}}{k_{hh'}}
\sum\limits_{\substack{
l_{hh'}, l'_{hh'} \mid (hh')^{\infty}
}}
\frac{|\alpha(h  k_{hh'} l'_{hh'}) \alpha(h'  k_{hh'} l_{hh'})|}{l_{hh'} l'_{hh'}} \times \nonumber\\
&\times
\sum\limits_{\substack{k\le X^2,\\ (3X)^{-2A^{-2/3}} \le \frac{k k_s k_{h h'} t}{\rho_1\rho'_1 \rho \rho' s N} \le (3X)^{2A^{-2/3}}
}}
\frac{G_{ k }^{10}}{k}
\sum\limits_{\substack{l, l' \mid k^{\infty}}} \frac{|\alpha(k l')
\alpha(jk l)|}{l l'} .
\end{align}

Let us estimate from above the sum over $k, l, l'$,
using the inequality
$$
G_{k}^{10} \ll \prod\limits_{p\mid k} \left( 1 + \frac{1}{\sqrt{p}}\right)
$$
and the notation $N_1 =  \max (1, \frac{(3X)^{-2A^{-2/3}} \rho_1\rho'_1 \rho \rho' s N}{k_s k_{hh'} t})$, $N_2 =  \min (X^2, \frac{(3X)^{2A^{-2/3}} \rho_1\rho'_1 \rho \rho' s N}{k_s k_{hh'} t})$ for the variables from (\ref{eq70}):
\begin{align}
&\sum\limits_{\substack{N_1 < k \le N_2
}}
\frac{G_{ k }^{10}}{k}
\sum\limits_{\substack{l, l' \mid k^{\infty}}}
\frac{|\alpha(k l') \alpha(k l)|}{l l'} \ll \sum\limits_{\substack{N_1 < k \le N_2
}}
\frac{1}{k} \prod\limits_{p^{\beta}\mid\mid k}\left( 1 + \frac{1}{\sqrt{p}}\right)  \left( |\alpha(p^{\beta})| + \frac{|\alpha(p^{\beta+1})|}{p} +\ldots\right)^2 \nonumber\\
&\ll   \sum\limits_{\substack{N_1 < k \le N_2
}}
\frac{1}{k} \prod\limits_{p^{\beta}\mid\mid k}\left( 1 + \frac{1}{\sqrt{p}}\right)  \left( |\alpha(p^{\beta})|^2 + \frac{5}{p}\right), \label{eq71}
\end{align}
where we appealed to $|\alpha(p^{\beta})| \le 1$ and $\frac{1}{p} + \frac{1}{p^2} +\ldots = \frac{1}{p (1 - 1/p)} \le \frac{2}{p}$.
In order to estimate the last quantity we write the following chain of formal equalities
\begin{align*}
&\sum\limits_{k=1}^{+\infty}\frac{1}{k^s} \prod\limits_{p^{\beta}\mid\mid k}\left( 1 + \frac{1}{\sqrt{p}}\right)  \left( |\alpha(p^{\beta})|^2 + \frac{5}{p}\right) \\
&= \prod\limits_{p}\left( 1 + \frac{1}{p^s} \left(1 + \frac{1}{\sqrt{p}}\right)  \left( |\alpha(p)|^2 + \frac{5}{p}\right) + \frac{1}{p^{2s}} \left(1 + \frac{1}{\sqrt{p}}\right)  \left( |\alpha(p^2)|^2 + \frac{5}{p}\right) +\ldots \right) \\
&= \sum\limits_{k=1}^{+\infty}\frac{1}{k^s}  \left( \sum\limits_{k_1k_2 = k} b_1 (k_1) b_2 (k_2)\right),
\end{align*}
where
\begin{align*}
&\sum\limits_{n=1}^{+\infty}  \frac{b_1 (n)}{n^s} = \prod\limits_{p} \left( 1 + \frac{|\alpha(p)|^2}{p^s}\right), \\
&\sum\limits_{n=1}^{+\infty} \frac{b_2 (n)}{n^s} = \prod\limits_{p}
\left( 1 + \frac{1}{p^s} \frac{\frac{5}{p} + \frac{5}{p^{3/2}} +\frac{|\alpha(p)|^2}{p^{1/2}}}{1 + \frac{|\alpha(p)|^2}{p^s}} + \frac{1}{p^{2s}} \frac{\left(1 + \frac{1}{\sqrt{p}}\right)  \left( |\alpha(p^2)|^2 + \frac{5}{p}\right)}{1 + \frac{|\alpha(p)|^2}{p^s}} +\ldots \right).
\end{align*}
Therefore, the sum in (\ref{eq71})
is less (up to some absolute constant) than
\begin{align*}
\sum\limits_{\substack{N_1 < k_1 \cdot k_2 \le N_2
}} \frac{b_1 (k_1) b_2 (k_2)}{k_1 k_2} \le \sum\limits_{\substack{k_2 \le N_2
}} \frac{b_2 (k_2)}{k_2}  \sum\limits_{\substack{N_1/k_2 < k_1\le N_2/k_2
}}  \frac{b_1 (k_1)}{k_1} \ll \max\limits_{k_2 \le N_2}  \sum\limits_{\substack{N_1/k_2 < k_1\le N_2/k_2
}}  \frac{b_1 (k_1)}{k_1},
\end{align*}
since
$$
\sum\limits_{\substack{n=1}}^{+\infty} \frac{b_2 (n)}{n} = \prod\limits_{p} \left( 1 + O\left( \frac{1}{p^{3/2}}\right) \right) \ll 1.
$$
So, let $K_1, K_2$ satisfy the following constraints
\begin{align}
1 \le K_1 < K_2 \le X^2, \nonumber\\
\frac{K_2}{K_1}  \le (3X)^{\frac{4}{A^{2/3}}}. \label{eq525}
\end{align}
It is convenient for us to get an upper bound for the sum $\sum\limits_{\substack{K_1 < k \le K_2}} \frac{b_1 (k_1)}{k_1}$ for every $K_1, K_2$ satisfying (\ref{eq525}) that  will depend only on $X$ and $A$. Let us provide such a bound using the Abel's partial summation formula.
Since we already have the bound
\begin{align} \label{eq526}
\bbC (u) = \sum\limits_{k \le u} b_1 (k)  \ll \frac{u}{\sqrt{\log u}}
\end{align}
then
\begin{align*}
&\sum\limits_{\substack{K_1 < k \le K_2}} \frac{b_1 (k)}{k}  \le
\frac{\bbC (K_2)}{K_2} + \int\limits_{K_1}^{K_2} \frac{\bbC (u)}{u^2} du \ll
1 + \int\limits_{K_1}^{K_2} \frac{1}{u \sqrt{\log u}} du \ll 1 +  (\sqrt{\log K_2} - \sqrt{\log K_1}) \\
&\ll
1 +  \frac{\log K_2 - \log K_1}{\sqrt{\log K_2} + \sqrt{\log K_1}} \ll 1 + \min\left( \sqrt{\log K_2}, \frac{\log X}{A^{2/3}\sqrt{\log K_2}}\right) \ll A^{-1/3} \sqrt{\log X},
\end{align*}
due to (\ref{eq525}). So, the sum in (\ref{eq71}) (which is the sum over $k, l, l'$ in (\ref{eq70}))
is less than $ A^{-1/3} \sqrt{\log X}$ (up to some absolute constant).

Similarly, we get that the sum over $s, t, j, j', \rho, \rho', k_s, l_s, l'_s$ in (\ref{eq70}) can be finally estimated by
\begin{align*}
\sum\limits_{s\le X} \frac{b_1(s)}{s} \ll  \int\limits_{1}^{X} \frac{du}{u \sqrt{\log u}} \ll \sqrt{\log X}.
\end{align*}
The sum over $h, h', \rho_1, \rho'_1, k_{h h'}, l_{hh'}, l'_{hh'}$ in (\ref{eq70})
by means of
\begin{align*}
&\sum\limits_{l\mid h} \frac{1}{l} \ll \prod_{p\mid h} \left( 1 +  \frac{1}{p} +  \frac{1}{p^2} +\ldots\right) \ll \prod_{p\mid h} \left( 1 +  \frac{1}{p}\right),\\
&\sum\limits_{k\mid h} \frac{1}{k} \prod_{p\mid k} \left( 1 +  \frac{1}{p^{3/4}}\right)^j \ll \prod_{p\mid h} \left( 1 +  \frac{1}{p}\left( 1 +  \frac{1}{p^{3/4}}\right)^j +  \frac{1}{p^2} \left( 1 +  \frac{1}{p^{3/4}}\right)^j+\ldots\right) \ll_{j} \prod_{p\mid h} \left( 1 +  \frac{1}{p}\right)
\end{align*}
can be estimated
by the expression:
\begin{align*}
&\sum\limits_{\substack{h \le X, h' \le X}} \frac{(G_{h h'})^{10}}{h h'}
\sum\limits_{\substack{ \rho_1 \mid h^{\infty}, \rho'_1  \mid h'^{\infty}}} \frac{|\alpha(h \rho_1)
\alpha(h'\rho'_1)|}{\rho_1 \rho'_1}
\sum\limits_{\substack{k_{h h'} \mid (h h')^{\infty}
}}
\frac{G_{ k_{hh'} }^{10}}{k_{hh'}}
\sum\limits_{\substack{
l_{hh'}, l'_{hh'} \mid (hh')^{\infty}
}}
\frac{1}{l_{hh'} l'_{hh'}} \\
&\ll \sum\limits_{\substack{h \le X, h' \le X}} \frac{(G_{h} G_{h'})^{10}}{h h'}
\sum\limits_{\substack{ \rho_1 \mid h^{\infty}, \rho'_1  \mid h'^{\infty}}} \frac{|\alpha(h \rho_1)
\alpha(h'\rho'_1)|}{\rho_1 \rho'_1}
\prod\limits_{p\mid h h'} \left( 1 + \frac{1}{p}\right)^{3} \\
&\le \left( \sum\limits_{\substack{h \le X}} \frac{G_{h}^{10}}{h}
\sum\limits_{\substack{ \rho_1 \mid h^{\infty}}} \frac{|\alpha(h \rho_1)|}{\rho_1}
\prod\limits_{p\mid h} \left( 1 + \frac{1}{p}\right)^{3} \right)^2
\ll \left( \sum\limits_{\substack{h \le X}} \frac{1}{h}
\prod\limits_{p\mid h} \left( 1 + \frac{1}{\sqrt{p}}\right)
\sum\limits_{\substack{ \rho_1 \mid h^{\infty}}} \frac{|\alpha(h \rho_1)|}{\rho_1}
 \right)^2 \\
&\ll \left( \sum\limits_{\substack{h \le X}} \frac{1}{h}
\prod\limits_{p^{\beta}\mid h} \left( 1 + \frac{1}{\sqrt{p}}\right)
\left( |\alpha(p^{\beta})| + \frac{|\alpha(p^{\beta+1})|}{p}  + \ldots \right)
 \right)^2 \ll \left( \sum\limits_{\substack{h \le X}} \frac{\mu^2 (h)|\alpha(h)|}{h} \right)^2,
\end{align*}
where the last bound is obtained similar to the estimation of (\ref{eq71}). By standard technique,
since
\begin{equation*}
\sum\limits_{p\le P} |\alpha(p)| \sim \frac{P}{2}
\end{equation*}
(or since the coefficients of the series, given by the Euler product $\prod\limits_{p} \left( 1+ \frac{|\alpha(p)|}{p^s}\right)$ are equal to the coefficients of the series, given by the Euler product $\prod\limits_{(p, D)=1} \left( 1+ \frac{|\alpha(p)|^2}{p^s}\right) \prod\limits_{p \mid D} \left( 1+ \frac{2|\alpha(p)|^2}{p^s}\right)$) we get that
\begin{equation*}
\left( \sum\limits_{\substack{h \le X}} \frac{\mu^2 (h)|\alpha(h)|}{h} \right)^2 \ll \log X
\end{equation*}
and the sum over $h, h', \rho_1, \rho'_1, k_{h h'}, l_{hh'}, l'_{hh'}$ in (\ref{eq70}) is bounded from above by $\log X$. Finally,
$$
W(X, N) \ll A^{-1/3} \log^2 X.
$$
Therefore, (\ref{eq520}) (as well as (\ref{eq518}), (\ref{eq514})) and, finally, $R^{(j,j)}$ is estimated from above by $\delta^{-1} A^{-1/3}$.

{\bf Acknowledgements.}

I express my deepest gratitude to all the members of the Department of Number Theory of Steklov Mathematical Institute of Russian Academy of Sciences and D.O.~Orlov  for their encouragement.

\end{document}